\documentclass{amsart}%
\usepackage{amsfonts}
\usepackage{hyperref}
\usepackage{adjustbox}
\usepackage[table]{xcolor}
\usepackage{tikz-cd}
\usepackage[all,cmtip]{xy}
\usepackage[margin=1.61in]{geometry}
\usepackage{amsmath}
\usepackage{amssymb}
\usepackage{graphicx}%
\providecommand{\U}[1]{\protect\rule{.1in}{.1in}}

\newtheoremstyle{myexample}
{\dimexpr \topsep+6pt \relax}
{\dimexpr \topsep+6pt \relax}
{\normalfont}
{}
{\bfseries}
{.}
{.5em}
{}
{}
\newtheorem{theorem}{Theorem}
\numberwithin{theorem}{section}
\theoremstyle{plain}
\newtheorem*{acknowledgement}{Acknowledgement}

\newtheorem{corollary}[theorem]{Corollary}

\newtheorem{definition}[theorem]{Definition}

\newtheorem{maintheorem}{Theorem}
\newtheorem{lemma}[theorem]{Lemma}

\newtheorem{proposition}[theorem]{Proposition}
\theoremstyle{myexample}
\newtheorem{example}[theorem]{Example}
\newtheorem{remark}[theorem]{Remark}
\newtheorem{caution}[theorem]{Caution}
\theoremstyle{remark}

\numberwithin{equation}{section}

\definecolor{lg}{rgb}{0.8,0.8,0.8}
\begin{document}
\title[ ]{Structure theorems for the heart of LCA}
\author[O. Braunling]{Oliver Braunling}
\address{Dortmund University of Applied Sciences, Germany}
\author[F. Ren]{Fei Ren}
\address{Bergische Universit\"{a}t Wuppertal, Germany}
\thanks{O.B. acknowledges support for this article as part of Grant CNS2023-145167
funded by MICIU/AEI/10.13039/501100011033.}

\begin{abstract}
Cohomology theories with values in LCA\ (locally compact abelian) groups
suffer from the problem that the latter do not form an abelian category.
However, the category LCA has a canonical abelian category envelope, the heart
of a suitable $t$-structure. It adds formal cokernel objects. We show the
surprising result that these abstract cokernels can also be interpreted as
Hausdorff topological abelian groups, at least up to lattice isogenies. These
need not be locally compact.

\end{abstract}
\maketitle

%

\tableofcontents

\section{Overview}

The category $\mathsf{LCA}$ of locally compact abelian groups is \textit{not
quite} an abelian category. For many applications, e.g.
\cite{MR4711130,MR4699875} or
\cite{MR4831262,artusa2025dualitycondensedweiletalerealisation}, one has to
work with its abelian envelope instead, the heart of a suitable $t$-structure
$\mathcal{LH}(\mathsf{LCA})$. The exotic \textit{non-}LCA objects which then
show up next to the classical LCA groups, and which effectively contribute new
cokernels to make the category abelian, are often regarded as pathological and
mysterious, which is what our paper aims to dispel. A priori these new
artificial cokernels have no reason to admit an interpretation as Hausdorff
topological groups. But they do, as we prove in this paper.

Write $\mathsf{PGA}$ for the category of groups which are isomorphic to a (not
necessarily closed) subgroup of a connected LCA group. Such a group can,
\textit{but need not}, be locally compact, e.g., take $\mathbb{Q}[\sqrt
{2}]\subset\mathbb{R}$ with the subspace topology from the real line or a
one-dimensional space-filling geodesic in a $2$-torus passing through the
neutral element $0$, again with its subspace topology. These are two Hausdorff
topological abelian groups, but neither is locally compact. A \emph{lattice
isogeny} is a continuous group morphism $\varphi\colon G^{\prime}\rightarrow
G$ such that both kernel and cokernel are discrete finitely generated abelian
groups. These morphisms form a multiplicative system of arrows which we denote
by $S_{\operatorname*{Lattice}}$.

\begin{maintheorem}
[Main Correspondence]\label{thma_main}Every object $X\in\mathcal{LH}%
(\mathsf{LCA})$ is a unique extension%
\[%
{
\begin{tikzcd}
	A & X & B
	\arrow[hook, from=1-1, to=1-2]
	\arrow[two heads, from=1-2, to=1-3]
\end{tikzcd}
}%
\]
with $B\in\mathsf{LCA}$ a classical LCA group and a non-classical part $A$.
All objects which show up as non-classical parts form a full subcategory
$\mathsf{Ghost}$ in the heart which is itself quasi-abelian and there is an
exact equivalence of categories%
\[
\mathsf{Ghost}\overset{\sim}{\longrightarrow}\mathsf{PGA}%
[S_{\operatorname*{Lattice}}^{-1}]\text{,}%
\]
where $\mathsf{PGA}[S_{\operatorname*{Lattice}}^{-1}]$ is the category
$\mathsf{PGA}$ up to lattice isogenies. This means that we have formed the
quotient category of $\mathsf{PGA}$ in which all morphisms whose kernel and
cokernel are discrete finitely generated groups become isomorphisms.
\end{maintheorem}

See Theorem \ref{thm_main_correspondence} below. Aside from L\"{u}ck's
equivalence \cite[Theorem 0.9]{MR1474192}, demystifying Farber's abelian
envelope of the category of Hilbert space representations of finite von
Neumann algebras, this theorem is one of the very few cases where the abelian
envelope of a category of topological abelian groups is fully understood. We
list a few other settings in \S \ref{sect_OtherSettings}.\medskip

We describe the groups in $\mathsf{PGA}$ a little: A discrete group lives in
$\mathsf{PGA}$ if and only if it is finitely generated. But in $\mathsf{PGA}%
[S_{\operatorname*{Lattice}}^{-1}]$ these become isomorphic to zero, so they
become irrelevant. All compact abelian groups and all of their subgroups are
in $\mathsf{PGA}$. For example, $\mathbb{Z}_{p}$ itself or $\mathbb{Z}%
\subset\mathbb{Z}_{p}$ with the subspace topology inside the $p$-adics are in
$\mathsf{PGA}$. On the other hand, $\mathbb{Q}_{p}$ or the ad\`{e}les are not
in $\mathsf{PGA}$. A Banach space is in $\mathsf{PGA}$ if and only if it is
finite-dimensional. The rational numbers $\mathbb{Q}$ with the discrete
topology are not in $\mathsf{PGA}$, but when equipped with the subspace
topology inside the real line, they are.

Suppose $A$ is a Hausdorff topological abelian group. We call a subset $T$
\emph{precompact} if for every non-empty open $U\subseteq A$, finitely many
translates $U+t$ cover all of $T$. All precompact groups live in
$\mathsf{PGA}$.

The objects in $\mathsf{PGA}$ admit an intrinsic characterization as the
precompactly generated and locally precompact Hausdorff groups, see the
Appendix \S \ref{Appendix_LocallyPrecompactGroups} for a review of the
necessary definitions.

The functor in Theorem \ref{thma_main} is defined as follows: The ghost
objects in $\mathcal{LH}(\mathsf{LCA})$ are formal cokernels for arrows%
\begin{equation}
x\colon X^{\prime}\longrightarrow X \label{l_intro_1}%
\end{equation}
of LCA\ groups $X^{\prime},X$ such that $x$ is injective with dense image. The
first intuition might be to carry out this quotient, but in the category
$\mathsf{LCA}$ that quotient is zero, and in the category of all topological
abelian groups it would be a non-Hausdorff group. This quotient loses too much
information: For example $\mathbb{R}_{d}\rightarrow\mathbb{R}$ (mapping the
real line with the discrete topology to the ordinary line) inevitably will
have quotient zero despite not being an isomorphism. This is the wrong path to pursue.

Instead, our functor does the following: We first show that up to
quasi-isomorphism in $\operatorname*{D}\nolimits^{b}(\mathsf{LCA})$, the
complex of Eq. \ref{l_intro_1} has a representative with $X^{\prime}$ discrete
and $X$ compact. Now equip $X^{\prime}$ instead with the subspace topology as
a subgroup in the compact $X$. We will prove that this always yields a group
in $\mathsf{PGA}$ (there is something to show since $X$ need not be
connected). The miracle is that this process can be reversed: Given a group
$Z$ in $\mathsf{PGA}$, we can complete it, giving a continuous map%
\begin{equation}
Z\longrightarrow cZ\text{,} \label{l_intro_1a}%
\end{equation}
with $cZ$ the completion (e.g., if the topology on $Z$ comes from a metric,
$cZ$ is really just the completion as a metric space. The metric completion of
a topological group is again a topological group in a canonical fashion). This
is not yet the original complex because $Z$ still carries a possibly
non-trivial topology, so we just force it: Replace the topology on $Z$ in Eq.
\ref{l_intro_1a} by the discrete topology. We claim that this restores the
original complex. In a sense, it is clear: We gave $Z$ the subspace topology
of $X^{\prime}$ inside $X$ and it was dense. Whenever you complete a dense
subspace in a complete space, this recovers the original complete space.

This is the correspondence $-$ roughly. Several critical complications will
cloud the simplicity of our basic construction:

\begin{enumerate}
\item The choice of a representative $x\colon X^{\prime}\longrightarrow X$
with $X^{\prime}$ discrete and $X$ compact is not unique. There will be
several such choices, but it turns out they can at worst differ by lattice
isogenies. We therefore replace $\mathsf{PGA}$ by $\mathsf{PGA}%
[S_{\operatorname*{Lattice}}^{-1}]$. In turn, we have to show that the reverse
process remains well-defined.

\item One would feel that one needs to show that for $Z\in\mathsf{PGA}$ the
completion $cZ$ is always compact. This is actually false, but whatever $cZ$
is, we show that $Z\longrightarrow cZ$ is always quasi-isomorphic to a complex
where the second group is indeed compact and that is good enough.

\item If the topology on $Z$ is not metrizable, one needs to use a more
complicated completion using nets or filters. This is due to Raikov and Weil.
We show that this type of completion is an exact functor on the groups of
interest to us, which appears to be a new result.

\item The most complicated part of the proof is to show that the morphisms in
the categories $\mathsf{Ghost}$ (certain roofs coming from the fact that the
heart lies in $\operatorname*{D}\nolimits^{b}(\mathsf{LCA})$, which is the
quotient category of complexes by acyclic complexes) and in $\mathsf{PGA}%
[S_{\operatorname*{Lattice}}^{-1}]$ agree (a different type of roof, now
coming from lattice isogenies).\medskip
\end{enumerate}

A possible point of confusion is that some groups are in both $\mathsf{LCA}$
\textit{and} $\mathsf{PGA}$, e.g., the circle $\mathbb{T}$ or the $p$-adic
integers $\mathbb{Z}_{p}$. These live a double life in $\mathcal{LH}%
(\mathsf{LCA})$ since they can now be regarded as objects in the heart in two
distinct ways, either as an LCA group \textit{or} as a ghost
group\footnote{Oren Ben-Bassat used the term `\textit{ghost}' during a video
call and we like the descriptive nature of the term very much.}. And these
objects admit non-split extensions. There exists, for example, a group $G$ in
the heart such that%
\[%
{
\begin{tikzcd}
	{\widetilde{\mathbb{T}}} && G && {\mathbb{T}}
	\arrow[hook, from=1-1, to=1-3]
	\arrow[two heads, from=1-3, to=1-5]
\end{tikzcd}
}%
\]
with $\mathbb{T}$ the ordinary circle in $\mathsf{LCA}$ and
$\mathbb{\widetilde{T}}$ \textit{also} the circle, but regarded as coming from
$\mathsf{PGA}$ (i.e., as a ghost group living in the non-LCA piece of the
heart), and this extension only lives in the heart, it is neither a genuine
topological group in $\mathsf{LCA}$ nor $\mathsf{PGA}$ alone.\medskip

Our hope is that our correspondence might make the non-LCA objects that can
show up in cohomology groups as they could occur in \cite{MR4699875,MR4711130}
or \cite{MR4831262} more accessible to a concrete intepretation. Certainly,
rendering them Hausdorff topological abelian groups makes them a lot more
material and touchable than as quasi-isomorphism classes of objects of a
certain type in the derived category $\operatorname*{D}\nolimits^{b}%
(\mathsf{LCA})$, or categories of monic-epic morphisms modulo bicartesian
squares (this is the interpretation that the general theory of abelian
envelopes of Schneiders \cite{MR1779315} offers).

Instead of working in the left heart $\mathcal{LH}(\mathsf{LCA})$, a reader
might prefer to embed $\mathsf{LCA}$ into the category of Clausen--Scholze
condensed abelian groups $\mathsf{Cond}(\mathsf{Ab})$ \cite{condensedmath}.
Both approaches are compatible (see \S \ref{sect_CondensedPicture}), and one
may regard these as full subcategories%
\[
\mathsf{LCA}\subset\mathcal{LH}(\mathsf{LCA})\subset\mathsf{Cond}%
(\mathsf{Ab})\text{.}%
\]
In particular, $\mathsf{Ghost}\cong\mathsf{PGA}[S_{\operatorname*{Lattice}%
}^{-1}]$ is a full subcategory of $\mathsf{Cond}(\mathsf{Ab})$ (Theorem
\ref{thm_p1}). We believe that this might only be the first example of
possibly a whole series of exotic embeddings of topological abelian groups
into $\mathsf{Cond}(\mathsf{Ab})$. There is no classical embedding of
$\mathsf{PGA}$ into $\mathsf{Cond}(\mathsf{Ab})$ since the underlying spaces
of precompact groups need not be $k$-spaces, so the usual embedding of
Clausen--Scholze collapses for them (Examples \ref{example_p1},
\ref{example_p2} by Gabriyelyan and Trigos-Arrieta).

\subsection{Variations of the correspondence}

Our main result comes in various shapes and forms. First, the picture is
simpler if we consider the full subcategory $\mathsf{LCA}_{\operatorname*{vf}%
}\subset\mathsf{LCA}$ of those LCA groups which do not have a real vector
space as a direct summand. These groups are known as \emph{vector-free}. Then
it is enough to use precompact groups to model ghost objects, and moreover we
also only need to quotient out classical isogenies. We write
$S_{\operatorname*{Isogeny}}$ for the multiplicative system of arrows which
have finite kernel and cokernel:

\begin{maintheorem}
[Correspondence, Vector-free Variant]\label{thma_vfversion}Every object
$X\in\mathcal{LH}(\mathsf{LCA}_{\operatorname*{vf}})$ is a unique extension%
\[%
{
\begin{tikzcd}
	A & X & B
	\arrow[hook, from=1-1, to=1-2]
	\arrow[two heads, from=1-2, to=1-3]
\end{tikzcd}
}%
\]
with $B\in\mathsf{LCA}_{\operatorname*{vf}}$ a classical LCA group and a
non-classical part $A$. All objects which show up as non-classical parts form
a full subcategory $\mathsf{Ghost}_{\operatorname*{vf}}$ in the heart which is
itself quasi-abelian and there is an exact equivalence of categories%
\[
\mathsf{Ghost}_{\operatorname*{vf}}\overset{\sim}{\longrightarrow}%
\mathsf{PCA}[S_{\operatorname*{Isogeny}}^{-1}]\text{,}%
\]
where $\mathsf{PCA}[S_{\operatorname*{Isogeny}}^{-1}]$ is the category of
precompact Hausdorff abelian groups up to isogenies. This means that we have
formed the quotient category of $\mathsf{PCA}$ in which all morphisms with
finite kernel and finite cokernel become isomorphisms.
\end{maintheorem}

This is also part of Theorem \ref{thm_main_correspondence} below. Besides the
left $t$-structure and its left heart $\mathcal{LH}(\mathsf{LCA})$, Schneiders
also defines a mirrored right $t$-structure whose heart is known as the right
heart $\mathcal{RH}(\mathsf{LCA})$. We will not discuss this variant much in
this text, but it turns out that essentially the same structural results also
apply to the right heart. For future reference, we spell them out explicitly.

\begin{maintheorem}
[Correspondence, Right Heart Versions]\label{thma_main_rightheartversions}%
Every object $X\in\mathcal{RH}(\mathsf{LCA})$ [resp. $\mathcal{RH}%
(\mathsf{LCA}_{\operatorname*{vf}})$] is a unique extension%
\[%
{
\begin{tikzcd}
	B & X & A
	\arrow[hook, from=1-1, to=1-2]
	\arrow[two heads, from=1-2, to=1-3]
\end{tikzcd}
}%
\]
with $B\in\mathsf{LCA}$ [resp. $\mathsf{LCA}_{\operatorname*{vf}}$] a
classical LCA group [resp. without real line summand] and a non-classical part
$A$. All objects which show up as non-classical parts form a full subcategory
$\mathsf{Ghost}$ [resp. $\mathsf{Ghost}_{\operatorname*{vf}}$] in the heart
which is itself quasi-abelian and there is an exact equivalence of categories%
\[
\mathsf{Ghost}\overset{\sim}{\longrightarrow}\mathsf{PGA}%
[S_{\operatorname*{Lattice}}^{-1}]
\]
with $S_{\operatorname*{Lattice}}$ as in Theorem \ref{thma_main} [resp.%
\[
\mathsf{Ghost}_{\operatorname*{vf}}\overset{\sim}{\longrightarrow}%
\mathsf{PCA}[S_{\operatorname*{Isogeny}}^{-1}]
\]
with $S_{\operatorname*{Isogeny}}$ as in Theorem \ref{thma_vfversion}].
\end{maintheorem}

See \S \ref{sect_Proofs} for the proof. Another interesting story is that
Pontryagin duality on $\mathsf{LCA}$ extends to a categorical equivalence
$\mathcal{LH}(\mathsf{LCA})\overset{\sim}{\longrightarrow}\mathcal{RH}%
(\mathsf{LCA})^{op}$ with the right heart, and this induces a duality on ghost
groups%
\[
\mathsf{Ghost}^{op}\longrightarrow\mathsf{Ghost}\text{.}%
\]
We give a partial description of this, and show that for a certain full
subcategory, the duality agrees with weak duality, i.e., $X\mapsto
\operatorname*{Hom}(X,\mathbb{T})$, but equipped with the weak topology
instead of the usual compact-open topology (Theorem \ref{thm_dual_1}).

\subsection{Abelian envelopes of other categories of topological
modules\label{sect_OtherSettings}}

It is possible that our techniques extend to hearts of other categories of
topological abelian groups or modules. We do not explore this potential for
generalization in the present text, but closely related investigations in
various topological settings can be found in the literature. We list a few.

\begin{itemize}
\item In the context of $L^{2}$-cohomology of \textit{real manifolds}, the
category of Hilbert space representations of von Neumann algebras is not
abelian, and Farber studied an abelian envelope under the name of
\textquotedblleft extended categories\textquotedblright%
\ \cite{MR1406667,MR1656223}. A correspondence giving a structural
understanding of this abelian envelope was found by L\"{u}ck \cite{MR1474192}.
These papers do not use $t$-structures or quasi-abelian categories as in
Schneiders, but could easily be rephrased in this terminology.

\item In the context of $L^{2}$-cohomology of \textit{complex manifolds} and
complex analytic spaces, Eyssidieux \cite{MR1776117} and Dingoyan
\cite{MR3118627} have developed elaborate complex geometry analogues which
again output objects living in such an abelian category envelope in the style
of Farber. Recent work of Eyssidieux \cite{AHL25} extends the latter picture
to categories of mixed Hodge modules.

\item In non-archimedean analytic geometry, Banach spaces fail to form an
abelian category (both in the real and non-archimedean setting) and the left
heart abelian envelope is used by Ben-Bassat and Kremnizer in \cite{MR3626003}.

\item The articles \cite{MR4699875,MR4711130} or
\cite{MR4831262,artusa2025dualitycondensedweiletalerealisation} are in the
setting of LCA groups, just as the present paper.
\end{itemize}

Articles of a more functional analytic nature are:

\begin{itemize}
\item Work of Lupini studies the heart of categories of Polish abelian groups
\cite{MR4779743}.

\item Work of Wegner studies the heart of the category of Banach and
Fr\'{e}chet spaces and shows that various categories of topological spaces
which fail to be quasi-abelian (e.g., $LB$-spaces), still permit categories
similar to Schneiders' left heart \cite{MR3655709}. This has its origins in
older works of Waelbroeck to find abelian envelopes of categories in
functional analysis. Recent developments in this direction are, among others,
for example \cite{MR4441467,MR4575371,MR4590331}.
\end{itemize}

This list is incomplete and we apologize to the authors of the works we have
not mentioned.

\section{Quasi-abelian categories}

\subsection{Definitions}

Suppose $\mathsf{C}$ is an additive category such that for every morphism
kernel and cokernel exist. Suppose $f\colon A\rightarrow B$ is any morphism.
Then the \emph{co-image} $\operatorname*{coim}f$ is defined as
$\operatorname*{coker}(\ker f\rightarrow A)$ and the \emph{image} is defined
as $\ker(B\rightarrow\operatorname*{coker}f)$.\ The universal properties of
kernels and cokernels define a natural diagram%
\[
A\longrightarrow\operatorname*{coim}f\overset{\widehat{f}}{\longrightarrow
}\operatorname*{im}f\longrightarrow B\text{.}%
\]

\begin{definition}
\label{def_AdmissibleMorphism}A morphism $f\colon A\rightarrow B$ is called
\emph{admissible} if $\widehat{f}$ is an isomorphism.\footnote{Other authors
call this a \emph{strict} morphism. Our terminology is in line with B\"{u}hler
\cite{MR2606234}.}
\end{definition}

It is important to stress that the class of admissible morphisms need not be
closed under composition.

\begin{example}
The category $\mathsf{C}$ is called \emph{abelian} if and only if every
morphism is admissible \cite[Exercise 8.6]{MR2606234}.
\end{example}

The category $\mathsf{C}$ is called \emph{quasi-abelian} if

\begin{enumerate}
\item the pushout of an admissible monic $f$ along an arbitrary arrow always
exists and is an admissible monic as well:%
\[%
{
\begin{tikzcd}
	A && B \\
	\\
	X && {X \cup_A B}
	\arrow["f", hook, from=1-1, to=1-3]
	\arrow[from=1-1, to=3-1]
	\arrow[from=1-3, to=3-3]
	\arrow["{f'}"', hook, from=3-1, to=3-3]
\end{tikzcd}
}%
\]
and dually

\item the pullback of an admissible monic always exists and is also an
admissible monic.
\end{enumerate}

Equivalently, $\mathsf{C}$ is quasi-abelian if the admissible monics and
admissible epics (with admissibility in the sense of Definition
\ref{def_AdmissibleMorphism}) satisfy the axioms of an exact structure
\cite[\S 2]{MR2606234}.

For $\mathsf{C}$ quasi-abelian, the morphism $\widehat{f}$ is always both
monic and epic (\cite[Cor. 1.1.5]{MR1779315}). We will use the notation%
\[%
{
\begin{tikzcd}
	{X'} && X
	\arrow["x"{inner sep=.8ex}, "\bullet"{marking}, from=1-1, to=1-3]
\end{tikzcd}
}%
\]
throughout this text to stress when a morphism $x$ is both monic and epic.

\begin{lemma}
\label{lemma_S0}A morphism $f\colon A\rightarrow B$ is admissible in the sense
of Definition \ref{def_AdmissibleMorphism} if and only if a factorization
$A\overset{e}{\twoheadrightarrow}X\overset{m}{\hookrightarrow}B$ exists such
that $e$ is an admissible epic and $m$ an admissible monic.
\end{lemma}

\begin{proof}
Fairly clear. If a factorization $A\overset{e}{\twoheadrightarrow
}X\overset{m}{\hookrightarrow}B$ of $f$ exists, it forces $X\cong%
\operatorname*{coim}f$ $\cong\operatorname*{im}f$, showing that $f$ is
admissible in the sense of Definition \ref{def_AdmissibleMorphism}.
Conversely, take $X$ to be the (co)image of $f$.
\end{proof}

\subsection{Quotients}

Suppose $\mathsf{Q}$ is a quasi-abelian category and $\mathsf{C}%
\subseteq\mathsf{Q}$ is a full subcategory which is closed under subobjects
and quotients inside $\mathsf{Q}$, i.e., if%
\[
X\hookrightarrow C
\]
is an admissible monic in $\mathsf{Q}$ with $C\in\mathsf{C}$, then $X$ must
also be in $\mathsf{C}$, and symmetrically, if%
\[
C\twoheadrightarrow Y
\]
is an admissible epic in $\mathsf{Q}$ with $C\in\mathsf{C}$, then $Y$ must
also lie in $\mathsf{C}$. Suppose moreover that $\mathsf{C}$ is closed under
extensions in $\mathsf{Q}$. Jointly with the above properties, this implies
that $\mathsf{C}$ is a Serre subcategory of $\mathsf{Q}$, i.e., whenever%
\[
X^{\prime}\hookrightarrow X\twoheadrightarrow X^{\prime\prime}%
\]
is exact in $\mathsf{Q}$, we have $X\in\mathsf{C}$ if and only if both
$X^{\prime},X^{\prime\prime}\in\mathsf{C}$. In this case, we may form the
quotient category $\mathsf{Q}/\mathsf{C}$, which is defined as the
Gabriel--Zisman localization $\mathsf{Q}[S_{\mathsf{C}}^{-1}]$, where
$S_{\mathsf{C}}$ refers to the multiplicative system of finite chains of
compositions of admissible epics with $\ker(f)\in\mathsf{C}$ and admissible
monics with $\operatorname*{coker}(f)\in\mathsf{C}$. This is a left and right
multiplicative system.

\begin{definition}
Let $\mathsf{C}\subseteq\mathsf{Q}$ be a Serre subcategory closed under
subobjects and quotients. Call $\mathsf{Q}/\mathsf{C}:=\mathsf{Q}%
[S_{\mathsf{C}}^{-1}]$ the \emph{quotient category} of $\mathsf{Q}$ by
$\mathsf{C}$.

\begin{enumerate}
\item The functor%
\[
q\colon\mathsf{Q}\longrightarrow\mathsf{Q}/\mathsf{C}%
\]
is additive and commutes with finite limits and colimits.

\item Every additive functor $F\colon\mathsf{Q}\longrightarrow\mathsf{Q}%
^{\prime}$ to an additive category $\mathsf{Q}^{\prime}$ such that $F(f)$ is
an isomorphism for all $f\in S_{\mathsf{C}}$ has a $2$-universal factorization%
\begin{equation}%
{
\begin{tikzcd}
	{\mathsf{Q}} & {\mathsf{Q}'} \\
	{\mathsf{Q}/\mathsf{C}.}
	\arrow["F", from=1-1, to=1-2]
	\arrow["q"', from=1-1, to=2-1]
	\arrow["{\overline{F}}"', dashed, from=2-1, to=1-2]
\end{tikzcd}
}
\label{dg4}%
\end{equation}

\end{enumerate}
\end{definition}

For the abstract existence and properties of the quotient we rely on the
standard techniques of Gabriel--Zisman \cite[Ch. I, \S 3]{MR0210125}.

\begin{remark}
[{\cite[Prop. 1.4, Prop. 8.29]{hr}}]The quotient $\mathsf{Q}/\mathsf{C}$ is
also quasi-abelian, $q$ is an exact functor and if $\mathsf{Q}^{\prime}$ is an
exact category and $F$ an exact functor, so is $\overline{F}$ in Diagram
\ref{dg4}.
\end{remark}

\section{Cotilting torsion pairs}

If $\mathsf{E}$ is any exact category, there is a category of exact sequences
$\mathcal{E}\mathsf{E}$ whose objects are exact sequences $E^{\prime
}\hookrightarrow E\twoheadrightarrow E^{\prime\prime}$ in $\mathsf{E}$ and
morphisms are compatible diagrams of exact sequences. This is itself an exact
category with respect to levelwise admissible monics and epics (\cite[Exercise
3.9, Rmk. 3.10]{MR2606234}).

\begin{definition}
[{\cite[\S 5.4]{MR1996800}, \cite[Ch. II, \S 2]{MR1327209}}]%
\label{def_TorsionPair}Suppose $\mathsf{A}$ is an abelian category. A
\emph{torsion pair} $(\mathsf{T},\mathsf{F})$ in $\mathsf{A}$ are full
subcategories $\mathsf{T},\mathsf{F}\subseteq\mathsf{A}$ and a functor
$e\colon\mathsf{A}\rightarrow\mathcal{E}\mathsf{A}$ such that

\begin{enumerate}
\item $\operatorname*{Hom}\nolimits_{\mathsf{A}}(T,F)=0$ for all
$T\in\mathsf{T}$ and $F\in\mathsf{F}$,

\item the functor $e$ sends any object $A\in\mathsf{A}$ to an exact sequence%
\[
0\longrightarrow T\longrightarrow A\longrightarrow F\longrightarrow0
\]
with $T\in\mathsf{T}$ (called the \emph{torsion part} of $A$) and
$F\in\mathsf{F}$ (called the \emph{torsion-free part} of $A$).
\end{enumerate}

The torsion pair is \emph{tilting} if every object in $\mathsf{A}$ can be
realized as a subobject of an object in $\mathsf{T}$. It is \emph{cotilting}
if every object in $\mathsf{A}$ can be realized as a quotient of an object in
$\mathsf{F}$.
\end{definition}

\begin{example}
Instead of (2) it is sufficient to demand that one such sequence exists for
any object. Axiom (1) makes it possible to lift any morphism $A^{\prime
}\rightarrow A$ to a unique morphism in $\mathcal{E}\mathsf{E}$.
\end{example}

\begin{lemma}
\label{lemma8}If $(\mathsf{T},\mathsf{F})$ is a torsion pair in an abelian
category $\mathsf{A}$, then both $\mathsf{T}$ and $\mathsf{F}$ are
extension-closed subcategories in $\mathsf{A}$, i.e., they carry natural
induced exact structures. Both $\mathsf{T}$ and $\mathsf{F}$ are also
quasi-abelian and this pins down the same exact structure.
\end{lemma}

\begin{proof}
This can be found in Rump \cite[Theorem 2]{MR1856638} or Bondal--van den Bergh
\cite{MR1996800}. The comparison of exact structures as stated here is also
explicitly proven in Bodzenta--Bondal \cite[Prop. A.4]{MR4650627}. A more
general result than what we need here was recently developed in Tattar
\cite[Prop. 5.12]{MR4340852}.
\end{proof}

\begin{definition}
Suppose $\mathsf{Q}$ is a quasi-abelian category. Then the \emph{left heart}
$\mathcal{LH}(\mathsf{Q})$ is the full subcategory of the derived category
$\operatorname*{D}\nolimits^{b}(\mathsf{Q})$ of complexes%
\begin{equation}
\left[
{
\begin{tikzcd}
	{X'} && X
	\arrow["x", from=1-1, to=1-3]
\end{tikzcd}
}%
\right]  \label{l_3c}%
\end{equation}
concentrated in homological degrees $[1,0]$, $X^{\prime},X\in\mathsf{Q}$ and
$x$ a monic\footnote{the monic need \textit{not} be an admissible monic} in
$\mathsf{Q}$.
\end{definition}

As the derived category has been formed from complexes up to
quasi-isomorphism, this unravels to the following description:

\begin{remark}
\label{rmk_SchneidersDescriptionOfLH}$\mathcal{LH}(\mathsf{Q})$ is the
localization of the category of complexes as in Eq. \ref{l_3c} by the
multiplicative system formed from morphisms $u$ of two-term complexes such
that%
\begin{equation}%
{
\begin{tikzcd}
	{X'} && X \\
	\\
	{Y'} && Y
	\arrow["f", from=1-1, to=1-3]
	\arrow["{u'}"', from=1-1, to=3-1]
	\arrow["u", from=1-3, to=3-3]
	\arrow["g"', from=3-1, to=3-3]
\end{tikzcd}
}
\label{l_3d}%
\end{equation}
is a bicartesian square in the category $\mathsf{Q}$ \cite[Cor. 1.2.21]%
{MR1779315}.
\end{remark}

Following Schneiders and later independent work of Rump and Bondal--van den
Bergh, there is a standard relation between quasi-abelian categories and
abelian categories:

\begin{proposition}
\label{prop_QAbTorsionPairInLeftHeart}Suppose $\mathsf{Q}$ is a quasi-abelian
category. Take

\begin{enumerate}
\item $\mathsf{T}$ to be complexes $\left[
{
\begin{tikzcd}
	{X'} & X
	\arrow["\bullet"{marking}, from=1-1, to=1-2]
\end{tikzcd}
}%
\right]  $ whose arrow is both monic and epic, and

\item $\mathsf{F}$ ($\cong\mathsf{Q}$) to be complexes $\left[
{
\begin{tikzcd}
	{0} & X
	\arrow[, from=1-1, to=1-2]
\end{tikzcd}
}%
\right]  $ whose first object is zero.
\end{enumerate}

Then $(\mathsf{T},\mathsf{F})$ is a cotilting torsion pair in the left heart
$\mathcal{LH}(\mathsf{Q})$. The natural inclusion $\mathsf{Q}\longrightarrow
\mathcal{LH}(\mathsf{Q})$ sending $X$ to $\left[
{
\begin{tikzcd}
	{0} & X
	\arrow[, from=1-1, to=1-2]
\end{tikzcd}
}%
\right]  $ induces a derived equivalence $\operatorname*{D}\nolimits^{b}%
(\mathsf{Q})\overset{\sim}{\longrightarrow}\operatorname*{D}\nolimits^{b}%
(\mathcal{LH}(\mathsf{Q}))$.
\end{proposition}

\begin{proof}
This formulation follows Bondal--van den Bergh \cite[Prop. B.3 (1)$\Rightarrow
$(2), then Prop. B.2]{MR1996800}, but most of this is proven already by
Schneiders \cite[\S 1.2.3, Prop. 1.2.32]{MR1779315}, without using the
terminology of torsion pairs. Alternatively, one can also follow Rump's
treatise \cite[Theorem 2]{MR1856638}.
\end{proof}

The characterization of the torsion part $\mathsf{T}$ is stable under quasi-isomorphisms:

\begin{lemma}
\label{lemma9}Suppose $\mathsf{Q}$ is a quasi-abelian category. Suppose%
\[%
{
\begin{tikzcd}
	{X'} && X \\
	\\
	{Y'} && Y
	\arrow["f", from=1-1, to=1-3]
	\arrow[from=1-1, to=3-1]
	\arrow[from=1-3, to=3-3]
	\arrow["g"', from=3-1, to=3-3]
\end{tikzcd}
}%
\]
is a bicartesian square in $\mathsf{Q}$ with both horizontal arrows monic.
Then $f$ is epic if and only if $g$ is epic.
\end{lemma}

Said differently: If the rows represent objects in $\mathcal{LH}(\mathsf{Q})$,
then if one row lies in $\mathsf{T}$ (as defined in Prop.
\ref{prop_QAbTorsionPairInLeftHeart}), then so does the other.

\begin{proof}
Direct verification, e.g.: If $g$ is an epic, let us verify that $f$ is an
epic. Let $h\colon X\rightarrow Z$ be any arrow such that $h\circ f=0$. Then
there is a unique arrow $h^{\prime}\colon Y\rightarrow Z$ defined as the
pushout of $h$ and the zero map from $Y^{\prime}$ to $Z$. Since $h^{\prime
}\circ g=0$ by construction and $g$ is epic, $h^{\prime}=0$, but then it
follows that $h=0$. This shows that $f$ is epic. The converse direction is
done similarly. Alternatively, use \cite[Lemma 1.2.25]{MR1779315}: The
cokernel is zero if and only if the horizontal arrow is an epic.
\end{proof}

\begin{remark}
The analogous statement for $\mathsf{F}$ does not hold: Complexes whose first
object is non-zero can be quasi-isomorphic to objects in $\mathsf{F}$.
\end{remark}

\begin{example}
The category $\mathsf{Ghost}$, as it appears in the introduction, will
correspond precisely to the torsion part $\mathsf{T}$ of $\mathcal{LH}%
(\mathsf{LCA})$.
\end{example}

\section{\label{sect_ExactStructuresOnCatsOfTopGroups}Categories of
topological groups}

We shall freely use topological terms as recalled in
\S \ref{Appendix_LocallyPrecompactGroups}.

\subsection{Definitions}

\begin{definition}
\label{def_a1}Let $\mathsf{C}$ be any of the categories

\begin{itemize}
\item $\mathsf{HA}$ of Hausdorff topological abelian groups,

\item $\mathsf{LCA}$ of locally compact Hausdorff abelian groups, i.e., $0$
has a compact neighbourhood,

\item $\mathsf{LCA}_{\operatorname*{vf}}\subset\mathsf{LCA}$ of those groups
which do not have a real line as a direct summand,

\item $\mathsf{LPA}$ of locally precompact Hausdorff abelian groups, i.e., $0$
has a precompact neighbourhood (see Definition \ref{def_Precompact}),
\end{itemize}

or its full subcategories

\begin{itemize}
\item $\mathsf{PCA}\subset\mathsf{LPA}$ of precompact Hausdorff abelian groups
(see Definition \ref{def_Precompact}),

\item $\mathsf{PGA}\subset\mathsf{LPA}$ of locally precompact, precompactly
generated Hausdorff abelian groups (see Definition \ref{def_Precompact} or the
alternative characterization of Prop. \ref{prop_ComfortLukacs} by Comfort--Luk\'{a}cs).
\end{itemize}
\end{definition}

In each respective case, the objects of the category are pairs $(X,\tau)$
consisting of an abelian group $X$ and a topology $\tau$ on the underlying set
$X$. Morphisms are continuous group homomorphisms. The following arrows
indicate inclusion of a full subcategory:%
\[%
{
\begin{tikzcd}
	{\mathsf{PCA}} & {\mathsf{PGA}} \\
	&& {\mathsf{LPA}} && {\mathsf{HA}} \\
	& {\mathsf{LCA}}
	\arrow[from=1-1, to=1-2]
	\arrow[from=1-2, to=2-3]
	\arrow[from=2-3, to=2-5]
	\arrow[from=3-2, to=2-3]
\end{tikzcd}
}%
\]
In each of the categories of Definition \ref{def_a1} we use the symbol

\begin{enumerate}
\item
{
\begin{tikzcd}
	{X'} & X
	\arrow[hook, from=1-1, to=1-2]
\end{tikzcd}
}
for an injective closed continuous group morphism,

\item
{
\begin{tikzcd}
	{X'} & X
	\arrow[two heads, from=1-1, to=1-2]
\end{tikzcd}
}
for a surjective open continuous group morphism,

\item
{
\begin{tikzcd}
	{X'} & X
	\arrow["\bullet"{marking}, from=1-1, to=1-2]
\end{tikzcd}
}
for a morphism which is both monic and epic.\footnote{As we shall see below in
Lemma \ref{lemma_DescribeKernelsAndCokernelsInHA}, in all listed categories
$\mathsf{C}$ (i.e., $\mathsf{HA}$, $\mathsf{LCA}$, $\mathsf{LPA}$, etc.) this
corresponds to injective morphisms with dense image. Contrastingly, in the
category of all (not necessarily Hausdorff) topological abelian groups, it
would correspond to continuous bijective morphisms.}
\end{enumerate}

An \emph{exact sequence} in $\mathsf{C}$ (with $\mathsf{C}$ again any of the
categories of Definition \ref{def_a1}) is%
\begin{equation}
X^{\prime}\overset{i}{\hookrightarrow}X\overset{j}{\twoheadrightarrow
}X^{\prime\prime} \label{l_4}%
\end{equation}
which is (1) an exact sequence of the underlying abelian groups without
topology, (2) $i$ is a closed continuous map, (3) $j$ is an open continuous map.

\subsection{Basic properties}

\begin{lemma}
\label{lemma_LCAPermanence}Suppose $X\in\mathsf{LCA}$.

\begin{enumerate}
\item For any closed subgroup $Y\subseteq X$ with its induced subspace
topology, we have $Y,X/Y\in\mathsf{LCA}$ and%
\[
Y\hookrightarrow X\twoheadrightarrow X/Y
\]
is an exact sequence in $\mathsf{LCA}$ as described in Eq. \ref{l_4}.

\item If the induced subspace topology of some subgroup $Y\subseteq X$ is
locally compact, then $Y$ must be a closed subset of $X$, and we are in the
situation of (1).
\end{enumerate}
\end{lemma}

\begin{proof}
(1) As is true for all subspaces, $Y$ is also Hausdorff and since $Y$ is
closed, $X/Y$ is Hausdorff. If $C\ni0$ is a compact neighbourhood of $0$ in
$X$, it follows that $C\cap Y$ is both a neighbourhood of $0$ in $Y$, and a
compact set. Hence, $Y$ is also a locally compact group, i.e., $Y\in
\mathsf{LCA}$ and the inclusion is a closed injective map. Moreover, $X/Y$ is
Hausdorff since $Y$ is closed, and locally compact by \cite[Prop.
3.1.23]{MR2433295}\ (2) \cite[Prop. 1.4.19]{MR2433295} shows that $Y$ must be
a closed subset, so (1) applies.
\end{proof}

\begin{lemma}
\label{lemma_LCAvfProps}Suppose $X\in\mathsf{LCA}_{\operatorname*{vf}}$. For
any closed subgroup $Y\subseteq X$ with its induced subspace topology, we have
$Y,X/Y\in\mathsf{LCA}_{\operatorname*{vf}}$ and%
\[
Y\hookrightarrow X\twoheadrightarrow X/Y
\]
is an exact sequence in $\mathsf{LCA}_{\operatorname*{vf}}$ as described in
Eq. \ref{l_4}.
\end{lemma}

\begin{proof}
The real line $\mathbb{R}$ is both an injective and projective object in
$\mathsf{LCA}$ \cite[\S III]{MR0215016}, so if it is a direct summand of $Y$
or $X/Y$, it will lift to being a direct summand of $X$.
\end{proof}

\begin{example}
The full subcategory $\mathsf{LCA}_{\operatorname*{vf}}$ is \emph{not}
extension-closed in $\mathsf{LCA}$. The extension $\mathbb{Z}\hookrightarrow
\mathbb{R}\twoheadrightarrow\mathbb{T}$ is a counter-example. In particular,
the exact structure on $\mathsf{LCA}_{\operatorname*{vf}}$ is not the induced
one from $\mathsf{LCA}$. See \cite{braunling2025realvectorspacesregulators}
for more peculiar properties of $\mathsf{LCA}_{\operatorname*{vf}}$.
\end{example}

\begin{lemma}
\label{lemma_TBAPermanence}Suppose $X\in\mathsf{PCA}$. For every subgroup
$Y\subseteq X$ with its induced subspace topology, we have $Y,\overline
{Y},X/\overline{Y}\in\mathsf{PCA}$.
\end{lemma}

\begin{proof}
As is true for all subspaces, $Y,\overline{Y}$ are also Hausdorff, and since
$\overline{Y}$ is closed, $X/\overline{Y}$ is Hausdorff. \cite[Prop.
3.7.4]{MR2433295} shows that $Y$ is precompact, and since $\overline{Y}$ is
also a subgroup (\cite[Cor. 1.4.14]{MR2433295}), this settles the same for
$\overline{Y}$. For $X/\overline{Y}$ use \cite[Prop. 3.7.1]{MR2433295}.
\end{proof}

\begin{lemma}
\label{lemma_DescribeKernelsAndCokernelsInHA}Suppose $\mathsf{C}%
\in\{\mathsf{HA},\mathsf{LCA},\mathsf{LCA}_{\operatorname*{vf}},\mathsf{PCA}%
,\mathsf{LPA},\mathsf{PGA}\}$. Suppose $X,Y\in\mathsf{C}$. Suppose $f\colon
X\rightarrow Y$ is a morphism.

\begin{enumerate}
\item Then $\ker(f)=\{x\in X\mid f(x)=0\}\in\mathsf{C}$ is a concrete choice
for its kernel in the sense of category-theory. The inclusion%
\[
\ker(f)\longrightarrow X
\]
is a closed injective morphism in $\mathsf{C}$.

\item Moreover, $\operatorname*{coker}(f):=Y/\overline{f(X)}\in\mathsf{C}$ is
a concrete choice for its cokernel, and%
\[
Y\longrightarrow\operatorname*{coker}(f)
\]
is an open surjective morphism in $\mathsf{C}$.

\item The monomorphisms (in the sense of category theory) in $\mathsf{C}$ are
injective continuous group homomorphisms, i.e. those maps with kernel being
zero. The category-theoretic image $\ker\left(  Y\longrightarrow
\operatorname*{coker}(f)\right)  $ in $\mathsf{C}$ agrees with $\overline
{f(X)}$.

\item The epimorphisms (in the sense of category theory) in $\mathsf{C}$ are
continuous group homomorphisms with dense set-theoretic image, i.e. those maps
with cokernel being zero. The category-theoretic co-image
$\operatorname*{coker}(\ker(f)\longrightarrow X)$ in $\mathsf{C}$ is
$X/\ker(f)$.

\item The category $\mathsf{C}$ is quasi-abelian.
\end{enumerate}
\end{lemma}

\begin{remark}
One has to be very careful with the category-theoretic concept of
\textquotedblleft image\textquotedblright\ in (3), contrasting with the
convention in topology, where we would call $f(X)$ the image. But in most
categories of interest for this paper the true image object would be the
closure $\overline{f(X)}$. We call $f(X)$ the \emph{set-theoretic image} in
cases where confusion is possible, as this is the image object in the category
of sets.
\end{remark}

\begin{proof}
We only discuss this for $\mathsf{C}\in\{\mathsf{HA},\mathsf{LCA}%
,\mathsf{LCA}_{\operatorname*{vf}},\mathsf{PCA}\}$ here and the categories
$\{\mathsf{LPA},\mathsf{PGA}\}$ will be treated later in this text, see Lemma
\ref{lemma_DescribeKernelsAndCokernelsInHA_2}. Among $\mathsf{LCA}$ and
$\mathsf{LCA}_{\operatorname*{vf}}$, it suffices to treat $\mathsf{LCA}$ and
combine it with Lemma \ref{lemma_LCAvfProps}. (1) Since $Y$ is Hausdorff,
$\ker(f)=f^{-1}(\{0\})$ is a closed subgroup of $X$ and itself Hausdorff, so
we have $\ker(f)\in\mathsf{HA}$, and by Lemma \ref{lemma_LCAPermanence} (resp.
\ref{lemma_TBAPermanence}) it follows that this subgroup is also locally
compact (resp. precompact). The inclusion map of a closed subset is a closed
and continuous map. Finally, the universal property of kernels is evidently
satisfied. (2) Since $\overline{f(X)}$ is closed, Lemma
\ref{lemma_LCAPermanence} (resp. \ref{lemma_TBAPermanence}) shows that
$\operatorname*{coker}(f)$ is locally compact (resp. precompact) and
Hausdorff. The quotient map is always an open continuous surjection
(\cite[Lemma 3.2.1]{MR4510389}). (3), (4) follow. (5) The same proof as in
\cite[Prop. 1.2]{MR2329311} applies.
\end{proof}

The concept of exact sequences in Eq. \ref{l_4} agrees with the one induced
from the quasi-abelian structure as in \cite[Prop. 4.4]{MR2606234}. Both
define the same exact structure and render each category $\mathsf{C}$ an exact category.

\begin{example}
In each of the categories of Definition \ref{def_a1} suppose $f$ is a
continuous group morphism. Then

\begin{itemize}
\item $f$ is a monic iff it is injective,

\item $f$ is an admissible monic iff it is injective and a closed map,

\item $f$ is an epic iff it has dense image,

\item $f$ is an admissible epic if it is surjective and an open map.
\end{itemize}

In each choice for $\mathsf{C}$, a morphism can be both monic and epic without
being an isomorphism.
\end{example}

\subsection{Examples}

\begin{definition}
\label{def_SubspaceTopologyNotation}If $X$ is a subset of a topological space
$Y$, we shall write $X_{\subseteq Y}$ for $X$, equipped with the subspace
topology of $Y$.
\end{definition}

\begin{example}
\label{ex_image_coimage}In $\mathsf{HA}$ (or analogously in $\mathsf{LCA}$ or
$\mathsf{LPA}$) the natural map%
\[
\mathbb{Q}\overset{f}{\longrightarrow}\mathbb{T}%
\]
for $\mathbb{T}:=\mathbb{R}/\mathbb{Z}$ is continuous.\ Its kernel is the
inclusion of $\mathbb{Z}$ into $\mathbb{Q}$, its image is (the inclusion of
$\mathbb{T}$ into) $\mathbb{T}$, the co-image is the quotient $\mathbb{Q}%
/\mathbb{Z}$ of $\mathbb{Q}$, the cokernel is the unique map from $\mathbb{T}$
to the zero object $0$. We never consider non-Hausdorff spaces in this text.
However, if we did, then in the category of all (possibly non-Hausdorff)
topological abelian groups, the image of $f$ would be $\mathbb{Q}$ (with the
subspace topology of $\mathbb{T}$) and the cokernel would be $\mathbb{T}%
/(\mathbb{Q}/\mathbb{Z})$ with the indiscrete topology.
\end{example}

\begin{example}
\label{example_4}We review the cotilting torsion decomposition of Definition
\ref{def_TorsionPair}. Let $\alpha\in\mathbb{R}$ be given. As before,
$\mathbb{T}:=\mathbb{R}/\mathbb{Z}$ is the circle and we regard real numbers
as representatives of elements in it so that $1$ is the period. Define an
object%
\[
X_{\alpha}:=\left[
{
\begin{tikzcd}
	{\mathbb{Z}} && {\mathbb{T}}
	\arrow["\alpha"{inner sep=.8ex}, "\bullet"{marking}, from=1-1, to=1-3]
\end{tikzcd}
}%
\right]
\]
in $\mathcal{LH}(\mathsf{LCA})$, sending $1$ in $\mathbb{Z}$ to $\alpha$ in
$\mathbb{T}$. We see that $X_{\alpha}$ lies in the torsion class of
$\mathcal{LH}(\mathsf{LCA})$ precisely iff $\alpha$ is irrational. Whenever
$\alpha$ is rational, $X_{\alpha}\simeq\mathbb{Z}/n\mathbb{Z}$ for some $n$
and this lives in the torsion-free part of $\mathcal{LH}(\mathsf{LCA})$.
\end{example}

\begin{example}
\label{example_1}For any choice of $\mathsf{C}$ in Definition \ref{def_a1},
the axiom $\operatorname*{Hom}(T,F)=0$ of Definition \ref{def_TorsionPair}
arises as follows: If%
\[%
{
\begin{tikzcd}
	{X'} && X \\
	\\
	0 && Y
	\arrow["f"{inner sep=.8ex}, "\bullet"{marking}, from=1-1, to=1-3]
	\arrow[from=1-1, to=3-1]
	\arrow[from=1-3, to=3-3]
	\arrow[from=3-1, to=3-3]
\end{tikzcd}
}%
\]
is a morphism of complexes, then since the top horizontal arrow is injective
and dense, we see that the right downward arrow is zero on a dense subset,
hence must be the zero map by continuity.
\end{example}

Cokernels change, depending on whether we regard them intrinsic to a
quasi-abelian category or in its left heart:

\begin{example}
\label{example_z_in_Cmult}Suppose $z$ is a non-zero complex number. Then%
\begin{equation}
\mathbb{Z}\overset{f}{\longrightarrow}\mathbb{C}^{\times},\qquad1\mapsto z
\label{lm_1}%
\end{equation}
is a morphism in $\mathsf{LCA}$, but behaves differently depending on whether
$z$ is a root of unity or not. If it is, the morphism is admissible, but
neither monic nor epic, and both its image and co-image are $\mathbb{Z}%
/n\mathbb{Z}\simeq\mu_{n}(\mathbb{C})$ for some $n$. On the other hand, if $z$
is not a root of unity but still has $\left\vert z\right\vert =1$, then the
morphism is monic with image $U(1)\simeq S^{1}$, but co-image $\mathbb{Z}$.
The cokernel of $f$ intrinsic to $\mathsf{LCA}$ is isomorphic to the real
line. Only if $\left\vert z\right\vert \neq1$, the morphism is an admissible
monic. Then the cokernel is a $2$-torus (this is essentially the complex
version of Tate's uniformization of an elliptic curve).
\end{example}

\begin{example}
Extending Example \ref{example_z_in_Cmult}, precisely for $\left\vert
z\right\vert =1$ but $z$ not a root of unity, the cokernel of Eq. \ref{lm_1}
as a morphism in $\mathcal{LH}(\mathsf{LCA})$ happens to possess a non-trivial
contribution from $\mathsf{T}$ of $\mathcal{LH}(\mathsf{LCA})$, namely
\[
\left[
{
\begin{tikzcd}
	{\mathbb{Z}} && {S^1}
	\arrow["{1\mapsto z}"{inner sep=.8ex}, "\bullet"{marking}, from=1-1, to=1-3]
\end{tikzcd}
}%
\right]  \longrightarrow\operatorname*{coker}f\overset{\left\vert
\cdot\right\vert }{\longrightarrow}\mathbb{R}_{>0}^{\times}\text{,}%
\]
where the latter arrow is the absolute value morphism. These two examples show
how the cokernel of $f$ changes, depending on whether we regard it intrinsic
to $\mathsf{LCA}$ or in the left heart.
\end{example}

\section{Group-theoretical compact generation}

\begin{definition}
\label{def_CompactlyGeneratedTopologicalGroup}We will call a locally compact
group $X$ \emph{group-theoretically compactly generated} if there exists a
compact subset $C\subseteq X$ which generates $X$ as an abstract group.
\end{definition}

Said differently:\ Every subgroup $Z\subseteq X$ containing $C$ must be equal
to all of $X$. Yet differently: Every element $x\in X$ lies in the
set-theoretic image of%
\begin{align}
C^{n}  &  \longrightarrow X\nonumber\\
(x_{1},\ldots,x_{n})  &  \longmapsto x_{1}-x_{2}+x_{3}-x_{4}+\ldots
\label{l_ppr_1}%
\end{align}
for $n$ large enough (possibly depending on $x$).

\begin{caution}
The term \textquotedblleft compactly generated\textquotedblright\ has a long
history in topological group theory \cite{MR2433295, MR4510389, MR637201,
MR551496}, despite its evident conflicts with at least two other notions of
exactly the same name. In general point-set topology, a topological space $X$
is called `compactly generated' if a set-theoretic map
$X\overset{f}{\longrightarrow}Y$ is continuous if and only if the composition
$C\overset{s}{\longrightarrow}X\overset{f}{\longrightarrow}Y$ is continuous
for all possible choices of continuous maps $s$ from a compact $C$. This
meaning of the term is in active use for example in \cite{condensedmath},
\cite{MR2522659} and all of homotopy theory. All locally compact groups are
compactly generated in this sense, but most precompact groups are \emph{not}!
Moreover, there is the concept of compact generation of an abelian or stable
$\infty$-category, also in active use in the same texts.
\end{caution}

\begin{example}
The real line $\mathbb{R}$ is group-theoretically compactly generated, take
$C:=[-1,1]$.
\end{example}

\begin{example}
Every compact group is group-theoretically compactly generated, take $C$ to be
the entire group.
\end{example}

\begin{example}
\label{example_f0}The $p$-adics $\mathbb{Q}_{p}$ are locally compact, but not
group-theoretically compactly generated:\ If $C$ is any compact subset, its
image under $\mathbb{Q}_{p}\twoheadrightarrow\mathbb{Q}_{p}/\mathbb{Z}_{p}$
must be finite since $\mathbb{Q}_{p}/\mathbb{Z}_{p}$ is discrete as a
topological space. It follows that $C\subseteq\frac{1}{p^{n}}\mathbb{Z}_{p}$
for some large enough integer $n$. However, each $\frac{1}{p^{n}}%
\mathbb{Z}_{p}$ is an abelian subgroup of $\mathbb{Q}_{p}$, so the image under
the map in Eq. \ref{l_ppr_1} will never be able to leave this proper subgroup.
\end{example}

\begin{example}
The Laurent series $\mathbb{F}_{q}((t))$ with the $t$-adic topology are
locally compact, but not group-theoretically compactly generated. Imitate the
same argument as in Example \ref{example_f0}.
\end{example}

\begin{lemma}
[{\cite[Theorem 2.5, Theorem 2.6 (2)]{MR0215016}}]\label{lemma11}Suppose
$X\in\mathsf{LCA}$ is a group-theoretically compactly generated group in the
sense of Definition \ref{def_CompactlyGeneratedTopologicalGroup}.

\begin{enumerate}
\item Closed subgroups of $X$ are again group-theoretically compactly generated.

\item Quotients of $X$ by closed subgroups are again group-theoretically
compactly generated.

\item A group $Y\in\mathsf{LCA}$ is group-theoretically compactly generated if
and only if $Y\simeq\mathbb{R}^{n}\oplus\mathbb{Z}^{m}\oplus C$ for some
(finite) $n,m\geq0$ and $C$ a compact group.

\item The full subcategory of group-theoretically compactly generated groups
is extension-closed inside $\mathsf{LCA}$.
\end{enumerate}
\end{lemma}

\begin{lemma}
\label{lemma12}A discrete and group-theoretically compactly generated group
must be a finitely generated discrete abelian group.
\end{lemma}

\begin{proof}
Any possible choice of $C$ in Definition
\ref{def_CompactlyGeneratedTopologicalGroup} is necessarily a finite set.
\end{proof}

\section{Exactness of the Completion Functor\label{sect_Completion}}

\subsection{Completion}

We need to show that completion is an exact functor on $\mathsf{LPA}$. We are
not aware of this being widely known nor recorded in the literature. A
Hausdorff topological abelian group $X$ is called (\emph{Raikov-} or
\emph{Weil-}) \emph{complete} if every Cauchy net/filter converges, and we
refer to \cite[Def. 7.1.5]{MR4510389} or equivalently \cite[\S 3.6]{MR2433295}
for details. There is no difference whether one describes this theory via
nets, filters or uniform structures. If the topology on $X$ stems from a
metric, $X$ is complete iff it is complete in the sense of the metric (all
Cauchy sequences converge). We write $\mathsf{HA}^{\operatorname*{compl}%
}\subseteq\mathsf{HA}$ for the full subcategory of complete groups. The
inclusion functor has a left adjoint%
\[
c(-)\colon\mathsf{HA}\longrightarrow\mathsf{HA}^{\operatorname*{compl}%
}\text{,}%
\]
known as the \emph{completion} functor \cite[Def. 7.1.18]{MR4510389}. We write
$\iota_{X}$ for the unit of this adjunction, a natural morphism%
\[
\iota_{X}\colon X\longrightarrow cX
\]
in $\mathsf{HA}$. If the topology on $X$ comes from a metric, we may
concretely realize this functor by taking the metric completion. The metric
completion of a topological group is again a topological group in a canonical
fashion. If the topology is not metrizable, both cited sources provide a
construction of the completion using filters.

\begin{lemma}
For every $X\in\mathsf{HA}$,

\begin{enumerate}
\item the unit%
\[%
{
\begin{tikzcd}
	X && cX
	\arrow["{\iota_{X }}"{inner sep=.8ex}, "\bullet"{marking}, from=1-1, to=1-3]
\end{tikzcd}
}%
\]
is a monic and epic in $\mathsf{HA}$ (i.e., it is injective and has dense
set-theoretic image).

\item $ccX=cX$.
\end{enumerate}
\end{lemma}

\begin{proof}
See \cite[Thm. 7.1.10]{MR4510389}, showing that $X\rightarrow cX$ is injective
with dense set-theoretic image, and the categorical characterization follows
from Lemma \ref{lemma_DescribeKernelsAndCokernelsInHA}.
\end{proof}

\begin{lemma}
\label{lemma1}Let $f\colon X\rightarrow Y$ be a continuous map of topological
spaces and $\Sigma\subseteq X$ a dense subset.

\begin{enumerate}
\item Then $f(\Sigma)$ is dense in $f(X)$.

\item If $\Sigma$ is connected, then $X$ is connected.
\end{enumerate}
\end{lemma}

\begin{proof}
(1) A map $f\colon X\rightarrow Y$ is continuous if and only if $f(\overline
{A})\subseteq\overline{f(A)}$ holds for all subsets $A\subseteq X$. Using this
for our $f$, we obtain%
\[
f(X)=f(\overline{\Sigma})\subseteq\overline{f(\Sigma)}\subseteq\overline
{f(X)}\text{.}%
\]
Taking the closure, this yields the sandwich $\overline{f(X)}\subseteq
\overline{f\left(  \Sigma\right)  }\subseteq\overline{f\left(  X\right)  }$,
forcing equality. (2) \cite[Lemma B.6.4 (b)]{MR4510389}.
\end{proof}

Various properties of topological groups can be rephrased in terms of their completion:

\begin{lemma}
\label{lemma_PrecompactAreThoseWithCompactClosureInCompletion}Suppose
$X\in\mathsf{HA}$.

\begin{enumerate}
\item A subset $W\subseteq X$ is precompact iff $\overline{W}$ is compact in
$cX$, where $\overline{W}$ is the closure as a subset of $cX$.

\item $X$ is precompact iff $cX$ is compact.

\item $X$ is locally precompact iff $cX$ is locally compact.
\end{enumerate}
\end{lemma}

\begin{proof}
(1) \cite[Theorem 3.7.10]{MR2433295} for a subset $W\subseteq X$. (2) Since
$\overline{X}=cX$ in the special case of $W$ being the entire space $X$, it
follows that $X$ is precompact iff $cX$ is compact. (3) If $X$ is locally
precompact, $0$ has a precompact neighbourhood $N$, so there is an open $U$ in
$X$ such that $0\in U\subseteq N$ with $N$ precompact. In the completion,
$U=X\cap V$ for some open $V\subseteq cX$ (since $X$ carries the subspace
topology inside its completion \cite[Theorem 3.6.10]{MR2433295}), so $0\in
V\cap\overline{N}\subseteq\overline{N}$ and $\overline{N}$ is compact by (1).
This shows that $0$ has a compact neighbourhood in $cX$. Conversely, if $cX$
is locally compact, let $C$ be a compact neighbourhood of $0$ in $cX$. Then
$C\cap X$ satisfies $\overline{C\cap X}=C$ inside $cX$, so $C\cap X$ is a
precompact neighbourhood of $0$ in $X$ by (1).
\end{proof}

We are not aware that the following fact has been observed in literature.

\begin{lemma}
\label{lem_w1}Suppose%
\[
X^{\prime}\overset{f}{\hookrightarrow}X\overset{g}{\twoheadrightarrow
}X^{\prime\prime}%
\]
is an exact sequence in $\mathsf{HA}$ such that $X$ is locally precompact.
Then $cX^{\prime}\overset{cf}{\hookrightarrow}%
cX\overset{cg}{\twoheadrightarrow}cX^{\prime\prime}$ is exact in $\mathsf{HA}$.
\end{lemma}

\begin{proof}
Induced from $g$, we get the commutative square%
\begin{equation}%
{
\begin{tikzcd}
	X && {X''} \\
	\\
	cX && {cX'',}
	\arrow["g", two heads, from=1-1, to=1-3]
	\arrow["{{{\iota_X}}}"'{inner sep=.8ex}, "\bullet"{marking}%
, from=1-1, to=3-1]
	\arrow["{{{\iota_{X''}}}}"{inner sep=.8ex}, "\bullet"{marking}%
, from=1-3, to=3-3]
	\arrow["cg"', from=3-1, to=3-3]
\end{tikzcd}
}
\label{l_DF1}%
\end{equation}
where $g$ is a surjective open morphism. We write $cg$ for the induced
morphism. Now denote by $(cg)(cX)$ the set-theoretic image of $cX$ inside
$cX^{\prime\prime}$. This is a subgroup of $cX^{\prime\prime}$ and comes with
its subspace topology inside $cX^{\prime\prime}$. Writing $a$ for $cg$ with
this restricted codomain, makes $a$ a surjective map and it is known to be
open by Arhangel'skii--Tkachenko \cite[Theorem 3.6.19]{MR2433295} (this is
non-trivial!).%
\[%
{
\begin{tikzcd}
	cX &&&& {cX''} \\
	\\
	&& {(cg)(cX)}
	\arrow["cg", from=1-1, to=1-5]
	\arrow["a"', from=1-1, to=3-3]
	\arrow["b"', from=3-3, to=1-5]
\end{tikzcd}
}%
\]
It follows that $a$ is an admissible epic in $\mathsf{HA}$. Since $cX$ is
locally compact (Lemma
\ref{lemma_PrecompactAreThoseWithCompactClosureInCompletion}), it follows that
$(cg)(cX)$ is also locally compact (all quotient spaces of locally compact
spaces are themselves locally compact \cite[Theorem 3.3.15]{MR1039321}).
However, a locally compact subgroup inside a Hausdorff topological group must
be closed (\cite[Prop. 7]{MR0442141} or \cite[Prop. 1.4.19]{MR2433295}). Thus,
$b$ must be an admissible monic. Summarizing this, we deduce that $cg$ is an
admissible morphism as described in Lemma \ref{lemma_S0}. However, $(cg)(cX)$
must also be dense in $cX^{\prime\prime}$. This can be seen in several ways:
For example, since completion is a left adjoint, it sends epics to epics. But
we can also argue directly: As each $\iota_{(-)}$ is injective, we may regard
$X$ (resp. $X^{\prime\prime}$) as subspaces of $cX$ (resp. $cX^{\prime\prime}%
$). Since $X$ is dense in $cX$, $g(X)$ is dense in $(cg)(cX)$ by Lemma
\ref{lemma1}, but $g(X)=X^{\prime\prime}$ since $g$ is surjective. Thus,%
\begin{equation}
X^{\prime\prime}\underset{\operatorname*{dense}}{\subseteq}(cg)(cX)\subseteq
cX^{\prime\prime}\text{.} \label{l_7}%
\end{equation}
Taking closures in $cX^{\prime\prime}$, and using that $\overline
{X^{\prime\prime}}=cX^{\prime\prime}$, the sandwich argument of Eq. \ref{l_7}
implies $\overline{(cg)(cX)}=cX^{\prime\prime}$. Hence, the closed embedding
$b$ is \textit{surjective}, and therefore an isomorphism\footnote{one can also
argue by category theory: $b$ is an epic (it has dense image) and an
admissible monic and hence an isomorphism \cite[Exercise 2.6 (dual)]%
{MR2606234}.}. We conclude that $cg$ is an admissible epic. We now add the
kernel to Diagram \ref{l_DF1}, rendering the lower row an exact sequence:%
\begin{equation}%
{
\begin{tikzcd}
	& {X'} && X && {X''} \\
	{cX'} \\
	& {\operatorname{ker}(cg)} && cX && {cX''}
	\arrow["f", hook, from=1-2, to=1-4]
	\arrow["{{\iota_{X'}}}"', from=1-2, to=2-1]
	\arrow["\eta", dashed, from=1-2, to=3-2]
	\arrow["g", two heads, from=1-4, to=1-6]
	\arrow["{{\iota_X}}"', from=1-4, to=3-4]
	\arrow["{{\iota_{X''}}}", from=1-6, to=3-6]
	\arrow["\phi"', dashed, from=2-1, to=3-2]
	\arrow[hook, from=3-2, to=3-4]
	\arrow["cg"', two heads, from=3-4, to=3-6]
\end{tikzcd}
}
\label{lh_C2}%
\end{equation}
The arrow $\eta$ stems from the universal property of kernels, it must be
injective (as $f$ and $\iota_{X}$ are injective) and has dense image by
\cite[Thm. 3.6.19]{MR2433295}. Since $\ker(cg)$ is locally compact (since it
is a closed subgroup of the locally compact group $cX$, Lemma
\ref{lemma_LCAPermanence}), it is complete \cite[Prop. 8.2.6]{MR4510389}, so
by \cite[Def. 7.1.18]{MR4510389} (applied to $\eta$) a unique arrow $\phi$
exists, making the left triangle commute. Now apply \cite[Cor. 3.6.18]%
{MR2433295} to $\phi$, using $X^{\prime}$ as the dense subgroup in
$cX^{\prime}$ which $\phi$ isomorphically maps onto its set-theoretic image.
It follows that $\phi$ is an isomorphism in $\mathsf{HA}$. Thus, the lower row
in Diagram \ref{lh_C2} is canonically isomorphic to%
\[
cX^{\prime}\hookrightarrow cX\twoheadrightarrow cX^{\prime\prime}%
\]
and thus is also exact.
\end{proof}

Although the previous lemma is broader, the following appears to be a much
more natural formulation.

\begin{proposition}
\label{prop_functor_completion_exact}Restricted to $\mathsf{LPA}$, completion
is an exact functor%
\begin{align*}
\mathsf{LPA}  &  \longrightarrow\mathsf{LCA}\\
X  &  \longmapsto cX\text{.}%
\end{align*}

\end{proposition}

\begin{proof}
The heavy lifting has been done in Lemma \ref{lem_w1}: If%
\[
X^{\prime}\overset{f}{\hookrightarrow}X\overset{g}{\twoheadrightarrow
}X^{\prime\prime}%
\]
is an exact sequence in $\mathsf{LPA}$, it is exact in $\mathsf{HA}$. By Lemma
\ref{lem_w1} the induced sequence $cX^{\prime}\overset{cf}{\hookrightarrow
}cX\overset{cg}{\twoheadrightarrow}cX^{\prime\prime}$ of completions is exact
in $\mathsf{HA}$, but by Lemma
\ref{lemma_PrecompactAreThoseWithCompactClosureInCompletion} each $cX^{\prime
},cX,cX^{\prime\prime}$ is locally compact, and since $\mathsf{LCA}%
\subseteq\mathsf{HA}$ is a fully exact subcategory, the claim follows.
\end{proof}

\begin{example}
\label{example_CompletionNeedNotBeExact}In general, completion fails to be
right exact. Suppose $X\in\mathsf{HA}$. Then there exists a group
$TX\in\mathsf{HA}^{\operatorname*{compl}}$ along with an admissible epic
$q\colon TX\twoheadrightarrow X$ \cite[Prop. 11.1]{MR644485}, \cite[Theorem
4.1.48]{MR1368852}. Completion sends this to the epic\footnote{Since
completion is a left adjoint, it necessarily always preserves epics.}
$cq\colon TX\rightarrow cX$ since $TX$ is already complete. Hence, whenever
$X$ is not already complete itself, this shows that completion can send an
admissible epic to a \emph{non-}admissible epic. We get%
\[%
{
\begin{tikzcd}
	Z && TX && X \\
	\\
	Z && TX && cX,
	\arrow[hook, from=1-1, to=1-3]
	\arrow["1"', from=1-1, to=3-1]
	\arrow["q", two heads, from=1-3, to=1-5]
	\arrow["1"', from=1-3, to=3-3]
	\arrow["cq"', from=1-3, to=3-5]
	\arrow["\bullet"{marking}, from=1-5, to=3-5]
	\arrow[hook, from=3-1, to=3-3]
	\arrow[from=3-3, to=3-5]
\end{tikzcd}
}%
\]
where $Z:=\ker q$ is also complete (closed subgroups of complete groups are
complete). For example, take $X:=\mathbb{Q}_{\subset\mathbb{R}}$ ($\mathbb{Q}$
with the subspace topology of the real line).
\end{example}

\begin{remark}
Wigner \cite{MR387476} proves another exactness result for completion, but he
works in the setting of not necessarily Hausdorff abelian groups. The
admissible epics in this setting are genuinely different from our setting
(e.g., $\mathbb{Q}\longrightarrow\mathbb{R}$ is an epic in $\mathsf{HA}$, but
not in all topological abelian groups).
\end{remark}

\begin{lemma}
\label{lemma_OnPCACompletionAgreesWithBohrCompactification}Restricted to the
subcategory $\mathsf{PCA}\subset\mathsf{LPA}$, the completion functor agrees
with Bohr compactification.
\end{lemma}

\begin{proof}
Restricted to $\mathsf{PCA}\subset\mathsf{LPA}$, completion agrees with Bohr
compactification. This is \cite[Thm. 10.2.15]{MR4510389}: $bX\cong%
\widetilde{X^{+}}$ in the notation of loc. cit. (but $X^{+}$ is defined in
\cite[Prop. 10.2.13]{MR4510389} as the group $X$, but with the topology
replaced by the finest precompact group topology coarser than the original
topology of $X$. Since $X$ is already precompact, $X^{+}=X$).
\end{proof}

\subsection{Consequences of the exactness of completion}

The time has come where we can collect the fruits of our labour.

\begin{proposition}
\label{prop_LPAClosureProps}Suppose $X\in\mathsf{LPA}$.

\begin{enumerate}
\item Closed subgroups of $X$ are again locally precompact.

\item Quotients of $X$ by closed subgroups are again locally precompact.
\end{enumerate}
\end{proposition}

\begin{proof}
(1) If $X^{\prime}\hookrightarrow X$ is a closed subgroup, $cX^{\prime
}\hookrightarrow cX$ is an admissible monic by Lemma \ref{lem_w1}. Since $cX$
is locally compact by Lemma
\ref{lemma_PrecompactAreThoseWithCompactClosureInCompletion} and closed
subgroups of locally compact groups are locally compact, it follows that
$cX^{\prime}$ is locally compact. Using Lemma
\ref{lemma_PrecompactAreThoseWithCompactClosureInCompletion} in the reverse
direction as before, it follows that $X^{\prime}$ is locally precompact.
(2)\ Similar. If $X\twoheadrightarrow X/X^{\prime}$ is the quotient by the
closed subgroup $X^{\prime}$, Lemma \ref{lem_w1} shows that
$cX\twoheadrightarrow c(X/X^{\prime})$ is an admissible epic, and since $cX$
is locally compact, so is the quotient. Lemma
\ref{lemma_PrecompactAreThoseWithCompactClosureInCompletion} shows that
$X/X^{\prime}$ is locally precompact.
\end{proof}

\begin{lemma}
[{\cite[Lemma 5.2]{MR2681374}}]\label{lemma14}Let $X\in\mathsf{LPA}$ be a
locally precompact group. Then $X$ is precompactly generated if and only if
its completion $cX$ is group-theoretically compactly generated in the sense of
Definition \ref{def_CompactlyGeneratedTopologicalGroup}.
\end{lemma}

\begin{proposition}
\label{prop_PGAClosureProps}Suppose $X\in\mathsf{PGA}$, i.e., $X$ is locally
precompact and precompactly generated.

\begin{enumerate}
\item Closed subgroups of $X$ are again in $\mathsf{PGA}$.

\item Quotients of $X$ by closed subgroups are again in $\mathsf{PGA}$.

\item The full subcategory $\mathsf{PGA}$ is extension-closed in
$\mathsf{LPA}$.
\end{enumerate}
\end{proposition}

\begin{proof}
(1) If $X^{\prime}\hookrightarrow X$ is a closed subgroup, $X^{\prime}%
\in\mathsf{LPA}$ by Prop. \ref{prop_LPAClosureProps} and $cX^{\prime
}\hookrightarrow cX$ is an admissible monic by Lemma \ref{lem_w1}. Since
$X\in\mathsf{PGA}$, $cX$ is group-theoretically compactly generated by Lemma
\ref{lemma14}, and then its closed subgroup $cX^{\prime}$ is also
group-theoretically compactly generated by Lemma \ref{lemma11}, so $X^{\prime
}\in\mathsf{PGA}$ by using Lemma \ref{lemma14} in the reverse direction. (2)
This is analogous. (3) Suppose%
\[
X^{\prime}\overset{f}{\hookrightarrow}X\overset{g}{\twoheadrightarrow
}X^{\prime\prime}%
\]
is an exact sequence in $\mathsf{LPA}$ with $X^{\prime},X^{\prime\prime}$ in
$\mathsf{PGA}$. Then by Lemma \ref{lem_w1} the sequence $cX^{\prime
}\overset{cf}{\hookrightarrow}cX\overset{cg}{\twoheadrightarrow}%
cX^{\prime\prime}$ is exact and $cX^{\prime},cX^{\prime\prime}$ are
group-theoretically compactly generated by Lemma \ref{lemma14} and all three
groups are completions of locally precompact groups, and thus in
$\mathsf{LCA}$ by Lemma
\ref{lemma_PrecompactAreThoseWithCompactClosureInCompletion}. Now the
extension-closedness from Lemma \ref{lemma11} implies that $cX$ is compactly
generated, hence $X\in\mathsf{PGA}$.
\end{proof}

\begin{example}
\label{example_3}Dense subgroups $Y\subseteq X$ of group-theoretically
compactly generated groups $X$ are precompactly generated. To see this, note
that the completion of the inclusion $Y\rightarrow X$ becomes the identity
$cY=cX$ by the assumption that $Y$ is dense in $X$, but $cX=X$ since $X$ was
locally compact by assumption and therefore already complete \cite[Prop.
8.2.6]{MR4510389}. Hence, $cY=X$, so $cY$ is group-theoretically compactly
generated and by Lemma \ref{lemma14}, $Y$ is precompactly generated.
\end{example}

\begin{lemma}
\label{lemma_DescribeKernelsAndCokernelsInHA_2}Lemma
\ref{lemma_DescribeKernelsAndCokernelsInHA} is also true for $\mathsf{C}%
\in\{\mathsf{LPA},\mathsf{PGA}\}$. In particular, $\mathsf{LPA}$ and
$\mathsf{PGA}$ are quasi-abelian categories.
\end{lemma}

\begin{proof}
Follow the proof of Lemma \ref{lemma_DescribeKernelsAndCokernelsInHA} earlier
in this text, and combine it with Prop. \ref{prop_LPAClosureProps} for
$\mathsf{LPA}$ (resp. Prop. \ref{prop_PGAClosureProps} for $\mathsf{PGA}$) to
see that the kernels and cokernels, computed in $\mathsf{HA}$, happen to lie
in $\mathsf{LPA}$ (resp. $\mathsf{PGA}$) and agree with the kernel and
cokernel intrinsic to these full subcategories.
\end{proof}

\section{The Functor $\Theta$}

\begin{definition}
We write

\begin{itemize}
\item $\mathsf{Ab}_{\operatorname*{fin}}$ for the abelian category of finite
abelian groups, and

\item $\mathsf{Ab}_{\operatorname*{fg}}$ for the abelian category of finitely
generated abelian groups.
\end{itemize}
\end{definition}

\begin{definition}
\label{def_ObjectsOfTypeDC_resp_DCG}We call an object%
\[
\left[
{
\begin{tikzcd}
	D && C
	\arrow["f"{inner sep=.8ex}, "\bullet"{marking}, from=1-1, to=1-3]
\end{tikzcd}
}%
\right]
\]
in $\mathcal{LH}(\mathsf{LCA})$ with $f$ monic and epic (intrinsic to
$\mathsf{LCA}$), of type

\begin{enumerate}
\item \emph{d-c} if $D$ is discrete and $C$ compact,

\item \emph{d-cg} if $D$ is discrete and $C$ group-theoretically compactly generated.
\end{enumerate}
\end{definition}

\begin{definition}
We define functors%
\begin{equation}
\Theta\colon\mathsf{PCA}\longrightarrow\mathcal{LH}(\mathsf{LCA}%
_{\operatorname*{vf}})\qquad\text{resp.}\qquad\Theta\colon\mathsf{PGA}%
\longrightarrow\mathcal{LH}(\mathsf{LCA})\text{,} \label{l_h1}%
\end{equation}
both sending an object $X$ to the complex%
\begin{equation}
\left[
{
\begin{tikzcd}
	{X_d} && {cX},
	\arrow["x"{inner sep=.8ex}, "\bullet"{marking}, from=1-1, to=1-3]
\end{tikzcd}
}%
\right]  \label{l_h2}%
\end{equation}
where $X_{d}$ is $X$, but equipped with the discrete topology, and $cX$ is the
Weil completion of $X$. The arrow $x$ is the natural inclusion of
sets.\footnote{This is a continuous map, but not an embedding of topological
groups, i.e. $x$ usually will not induce a topological isomorphism onto its
set-theoretic image in $cX$.}
\end{definition}

\begin{lemma}
\label{lemma15}The functor~$\Theta$ is well-defined and induces functors%
\[
\Theta\colon\mathsf{PCA}/\mathsf{Ab}_{\operatorname*{fin}}\longrightarrow
\mathcal{LH}(\mathsf{LCA}_{\operatorname*{vf}})\qquad\text{resp.}\qquad
\Theta\colon\mathsf{PGA}/\mathsf{Ab}_{\operatorname*{fg}}\longrightarrow
\mathcal{LH}(\mathsf{LCA})\text{.}%
\]
In the situation of $\mathsf{PCA}$, this functor outputs objects of type d-c,
and in the $\mathsf{PGA}$ variant, it outputs objects of type d-cg. The
essential image is contained in the torsion class $\mathsf{T}$ in the sense of
Prop. \ref{prop_QAbTorsionPairInLeftHeart}.
\end{lemma}

\begin{proof}
We first show that the functors in Eq. \ref{l_h1} are well-defined and land in
$\mathsf{T}$. For any object $X$, we may factor the arrow $x$ in Eq.
\ref{l_h2} as%
\[
x\colon X_{d}\overset{a}{\longrightarrow}X\overset{b}{\longrightarrow
}cX\text{.}%
\]
Here $a$ is the natural arrow from $X$ to itself, but with the source given
the discrete topology. This map is tautologically bijective and continuous
since every subset of $X_{d}$ is open. The arrow $b$ is the natural arrow to
the completion, it is always injective and has dense image \cite[Theorem
7.1.18]{MR4510389}. Hence, the composition is injective with dense image and
therefore a monic and epic (Lemma \ref{lemma_DescribeKernelsAndCokernelsInHA}%
). By Lemma \ref{lemma14}, for $X\in\mathsf{PCA}$ [resp. $\mathsf{PGA}$], $cX$
is compact and therefore $\Theta(X)$ is of type d-c [resp. $cX$
group-theoretically compactly generated and therefore $\Theta(X)$ is of type
d-cg]. In order to show that the functor factors over the quotient category by
$\mathsf{Ab}_{\operatorname*{fin}}$ [resp. $\mathsf{Ab}_{\operatorname*{fg}}%
$], it suffices to note that for a discrete group $X$, we have $X=X_{d}=cX$,
since discrete groups are tautologically complete, so $\Theta$ sends $X$ to
the complex
\begin{equation}
\left[
{
\begin{tikzcd}
	X && X
	\arrow["1"{inner sep=.8ex}, "\bullet"{marking}, from=1-1, to=1-3]
\end{tikzcd}
}%
\right]  \text{,} \label{l_8}%
\end{equation}
which is quasi-isomorphic to zero. Thus, by the universal property of
localizations, the functor $\Theta$ factors over the localization
$\mathsf{PCA}[S_{\mathsf{Ab}_{\operatorname*{fin}}}^{-1}]$, where
$S_{\mathsf{Ab}_{\operatorname*{fin}}}$ is the multiplicative system of
morphisms in $\mathsf{LCA}$ whose kernel and cokernel lies in $\mathsf{Ab}%
_{\operatorname*{fin}}$ [resp. the analogue for $\mathsf{Ab}%
_{\operatorname*{fg}}$].
\end{proof}

\begin{example}
Although we have already constructed $\Theta$, it is worth spelling out what
this functor does on left roofs. Suppose $X\rightarrow Y$ is an arrow in
$\mathsf{PCA}/\mathsf{Ab}_{\operatorname*{fin}}$. Then it can be represented
by a left roof%
\begin{equation}
X\overset{s}{\longleftarrow}Z\longrightarrow Y \label{diag_7}%
\end{equation}
with $s$ having kernel $A:=\ker(s)$ and cokernel $B:=\operatorname*{coker}(s)$
in $\mathsf{Ab}_{\operatorname*{fin}}$. Since $(-)_{d}$ is exact, the arrow
$Z_{d}\rightarrow X_{d}$ has kernel $A_{d}$, but this agrees with $A$ since
$A$ is finite. Analogously, since $c(-)$ is exact, the arrow $cZ\rightarrow
cX$ has kernel $cA$, which again agrees with $A$ since $A$ is finite.
Analogously for the cokernels. The roof in Diagram \ref{diag_7} gets mapped to
the left roof%
\[
\left[
{
\begin{tikzcd}
	{X_d} & {cX}
	\arrow["\bullet"{marking}, from=1-1, to=1-2]
\end{tikzcd}
}%
\right]  \overset{\Theta(s)}{\longleftarrow}\left[
{
\begin{tikzcd}
	{Z_d} & {cZ}
	\arrow["\bullet"{marking}, from=1-1, to=1-2]
\end{tikzcd}
}%
\right]  \longrightarrow\left[
{
\begin{tikzcd}
	{Y_d} & {cY}
	\arrow["\bullet"{marking}, from=1-1, to=1-2]
\end{tikzcd}
}%
\right]  \text{,}%
\]
and $s$ unravels as a pair of arrows $(u^{\prime},u)$ in%
\[%
{
\begin{tikzcd}
	A && A \\
	{Z_d} && cZ \\
	\\
	{Y_d} && cY \\
	B && {B,}
	\arrow["1", from=1-1, to=1-3]
	\arrow[hook, from=1-1, to=2-1]
	\arrow[hook, from=1-3, to=2-3]
	\arrow["\bullet"{marking}, from=2-1, to=2-3]
	\arrow["{{u'}}"', from=2-1, to=4-1]
	\arrow["u", from=2-3, to=4-3]
	\arrow["\bullet"{marking}, from=4-1, to=4-3]
	\arrow[two heads, from=4-1, to=5-1]
	\arrow[two heads, from=4-3, to=5-3]
	\arrow["1"', from=5-1, to=5-3]
\end{tikzcd}
}%
\]
and it follows that the middle square must be bicartesian in $\mathsf{LCA}$.
Thus, $\Theta(s)$ determines a quasi-isomorphism in $\operatorname*{D}%
\nolimits^{b}(\mathsf{LCA})$, as explained in Eq. \ref{l_3d}. For
$\mathsf{PCA}/\mathsf{Ab}_{\operatorname*{fg}}$, the story carries over verbatim.
\end{example}

We make a trivial observation:

\begin{lemma}
\label{lemma_functor_d_exact}The functor%
\begin{align*}
\left(  -\right)  _{d}\colon\mathsf{PGA}  &  \longrightarrow\mathsf{LCA}\\
X  &  \longmapsto X_{d}\text{,}%
\end{align*}
sending a precompactly generated group to itself, but equipped with the
discrete topology, is exact.
\end{lemma}

\begin{proof}
This borders on being a triviality: Since the output of the functor consists
entirely of discrete topological abelian groups, any morphism $f\colon
X_{d}\rightarrow Y_{d}$ is trivially continuous once $X\rightarrow Y$ is a
group homomorphism. Moreover, checking exactness is equivalent to checking
exactness on the underlying abelian groups. Hence, by Eq. \ref{l_4}, it is
clear that $\left(  -\right)  _{d}$ is an exact functor.
\end{proof}

\begin{lemma}
\label{lemma_f_exact}The functor $\Theta$ is exact (with respect to the exact
structure on $\mathsf{PGA}$ resp. $\mathsf{PCA}$ as explained in
\S \ref{sect_ExactStructuresOnCatsOfTopGroups}).
\end{lemma}

\begin{proof}
We first prove exactness for $\Theta$ on $\mathsf{PGA}$. Suppose%
\[
X^{\prime}\hookrightarrow X\twoheadrightarrow X^{\prime\prime}%
\]
is an exact sequence in $\mathsf{PGA}$. We shall now use that a sequence in
$\mathsf{LCA}$ is exact with respect to the exact structure of $\mathsf{LCA}$
if and only if it is exact in the classical sense in the abelian category
$\mathcal{LH}(\mathsf{LCA})$, see \cite[Cor. 1.2.28]{MR1779315}. Consider the
following $(3\times3)$-diagram in $\mathcal{LH}(\mathsf{LCA})$.
\begin{equation}%
\begin{array}
[c]{ccccc}%
\left[
{
\begin{tikzcd}
	{0} & {X'_d}
	\arrow[from=1-1, to=1-2]
\end{tikzcd}
}%
\right]  & \longrightarrow & \left[
{
\begin{tikzcd}
	{0} & {X_d}
	\arrow[from=1-1, to=1-2]
\end{tikzcd}
}%
\right]  & \longrightarrow & \left[
{
\begin{tikzcd}
	{0} & {X''_d}
	\arrow[from=1-1, to=1-2]
\end{tikzcd}
}%
\right] \\
\downarrow &  & \downarrow &  & \downarrow\\
\left[
{
\begin{tikzcd}
	{0} & {cX'}
	\arrow[from=1-1, to=1-2]
\end{tikzcd}
}%
\right]  & \longrightarrow & \left[
{
\begin{tikzcd}
	{0} & {cX}
	\arrow[from=1-1, to=1-2]
\end{tikzcd}
}%
\right]  & \longrightarrow & \left[
{
\begin{tikzcd}
	{0} & {cX''}
	\arrow[from=1-1, to=1-2]
\end{tikzcd}
}%
\right] \\
\downarrow &  & \downarrow &  & \downarrow\\
\left[
{
\begin{tikzcd}
	{X'_d} & {cX'}
	\arrow[from=1-1, to=1-2]
\end{tikzcd}
}%
\right]  & \longrightarrow & \left[
{
\begin{tikzcd}
	{X_d} & {cX}
	\arrow[from=1-1, to=1-2]
\end{tikzcd}
}%
\right]  & \longrightarrow & \left[
{
\begin{tikzcd}
	{X''_d} & {cX''}
	\arrow[from=1-1, to=1-2]
\end{tikzcd}
}%
\right]
\end{array}
\label{d_6}%
\end{equation}
The first row is exact in $\mathcal{LH}(\mathsf{LCA})$ since $X_{d}^{\prime
}\hookrightarrow X_{d}\twoheadrightarrow X_{d}^{\prime\prime}$ is exact in
$\mathsf{LCA}$ by Lemma \ref{lemma_functor_d_exact}. The second row is exact
in $\mathcal{LH}(\mathsf{LCA})$ because $cX^{\prime}\hookrightarrow
cX\twoheadrightarrow cX^{\prime\prime}$ is exact in $\mathsf{LCA}$ by the
exactness of Weil completion in the present setting, Lemma
\ref{prop_functor_completion_exact}. The columns are all exact since they are
cone sequences\footnote{The third row terms are the cones of the top morphism
between the $2$-term complexes above in the derived category
$\operatorname*{D}^{b}(\mathsf{LCA})$}, and a sequence in the heart
$\mathcal{LH}(\mathsf{LCA})$ is a short exact sequence if and only if it can
be extended to a distinguished triangle in the derived category
$\operatorname*{D}\nolimits^{b}(\mathsf{LCA})$ \cite[Cor. A.7.9]{MR4337423}.
This means that all conditions for the $(3\times3)$-lemma in abelian
categories are satisfied, so we deduce that the bottom horizontal row is also
exact in $\mathcal{LH}(\mathsf{LCA})$, but it agrees with $\Theta(X^{\prime
})\longrightarrow\Theta(X)\longrightarrow\Theta(X^{\prime\prime})$. Hence,
$\Theta$ is an exact functor. The proof for $\Theta$ on $\mathsf{PCA}$ is
exactly the same. In the latter setting, the completions $c(-)$ are compact,
and so they do not contain a real line summand, so we can work in
$\mathcal{LH}(\mathsf{LCA}_{\operatorname*{vf}})$.
\end{proof}

\section{Rewriting Rules for Roofs}

In this section, we will develop a mechanism to represent morphisms in the
left heart in a specific format.

\begin{lemma}
\label{lemmaB2}Suppose%
\begin{equation}%
{
\begin{tikzcd}
	{D'} && {C'} \\
	\\
	D && C
	\arrow["f"{inner sep=.8ex}, "\bullet"{marking}, from=1-1, to=1-3]
	\arrow["u"', from=1-1, to=3-1]
	\arrow["v", from=1-3, to=3-3]
	\arrow["g"'{inner sep=.8ex}, "\bullet"{marking}, from=3-1, to=3-3]
\end{tikzcd}
}
\label{diag_9}%
\end{equation}
is a bicartesian square in $\mathsf{LCA}$ or $\mathsf{LCA}_{\operatorname*{vf}%
}$. Then the columns extend to exact sequences%
\[
\ker(u)\hookrightarrow D^{\prime}\overset{u}{\longrightarrow}%
D\twoheadrightarrow\operatorname*{coker}(u)\qquad\text{and}\qquad
\ker(v)\hookrightarrow C^{\prime}\overset{v}{\longrightarrow}%
C\twoheadrightarrow\operatorname*{coker}(v)
\]
such that $\ker(u)\cong\ker(v)$ as well as $\operatorname*{coker}%
(u)\cong\operatorname*{coker}(v)$.

\begin{enumerate}
\item If both rows in Diag. \ref{diag_9} are d-c, these kernels and cokernels
are finite groups.

\item If both rows in Diag. \ref{diag_9} are d-cg, these kernels and cokernels
are finitely generated discrete groups.
\end{enumerate}
\end{lemma}

\begin{proof}
Since the square in Diag. \ref{diag_9} is bicartesian, we have $D^{\prime
}\cong\ker\left(  C^{\prime}\oplus D\overset{v-g}{\longrightarrow}C\right)  $
and $C\cong\operatorname*{coker}\left(  D^{\prime}%
\overset{(f,u)}{\longrightarrow}C^{\prime}\oplus D\right)  $. Hence, we obtain
the exact sequence in the middle row of the following diagram, and the
existence of all solid arrows follows from naturality:%
\[%
{
\begin{tikzcd}
	{D'\cap C'} && {C'} && {C'/(D' \cap C')} \\
	\\
	{D'} && {C' \oplus D} && C \\
	\\
	{D'/(D' \cap C')} && D && Q
	\arrow[hook, from=1-1, to=1-3]
	\arrow[hook, from=1-1, to=3-1]
	\arrow[dashed, two heads, from=1-3, to=1-5]
	\arrow[hook, from=1-3, to=3-3]
	\arrow["v", dashed, hook, from=1-5, to=3-5]
	\arrow["{(f,u)}", hook, from=3-1, to=3-3]
	\arrow[two heads, from=3-1, to=5-1]
	\arrow["{v-g}", two heads, from=3-3, to=3-5]
	\arrow[two heads, from=3-3, to=5-3]
	\arrow[dashed, two heads, from=3-5, to=5-5]
	\arrow["u"', hook, from=5-1, to=5-3]
	\arrow[dashed, two heads, from=5-3, to=5-5]
\end{tikzcd}
}%
\]
The injectivity of the lower left arrow is easy to check. Since the groups in
question are discrete, it is trivial that the injection is an admissible
monic. The existence of all dashed arrows as well as the exactness of the
right column follow from \cite[Exercise 3.7]{MR2606234}. Unravelling the
bottom row and rightmost column, we see that $\ker(u)\cong D^{\prime}\cap
C^{\prime}\cong\ker(v)$ and $\operatorname*{coker}(u)\cong Q\cong%
\operatorname*{coker}(v)$. It remains to analyze their structure:\newline(1)
In the d-c setting, $D^{\prime}\cap C^{\prime}$ is both discrete and compact,
hence finite. Since $Q$ is a quotient of something compact, it must be
compact, and since it is a quotient of something discrete, it is also
discrete. Thus, $Q$ is also finite.\newline(2) In the d-cg setting,
$D^{\prime}\cap C^{\prime}$ is both discrete and a closed subgroup of a
group-theoretically compactly generated group. Thus, by Lemma \ref{lemma11}
and Lemma \ref{lemma12}, $D^{\prime}\cap C^{\prime}$ must be a finitely
generated discrete group. Moreover, $Q$ is a quotient of something discrete
and something group-theoretically compactly generated, so by the same lemmas,
it must also be a finitely generated discrete group.
\end{proof}

In preparation for proving essential surjectivity, we can already describe the
reverse of Lemma \ref{lemma15}.

\begin{lemma}
\label{lemma10}Suppose we are given $\left[
{
\begin{tikzcd}
	{Z'} & Z
	\arrow["z"{inner sep=.8ex}, "\bullet"{marking}, from=1-1, to=1-2]
\end{tikzcd}
}%
\right]  $ in $\mathcal{LH}(\mathsf{LCA})$ with $Z^{\prime}$ discrete and $Z$

\begin{enumerate}
\item of type d-c or

\item of type d-cg.
\end{enumerate}

Then there is a natural isomorphism%
\[
\left[
{
\begin{tikzcd}
	{Z'} & Z
	\arrow["z"{inner sep=.8ex}, "\bullet"{marking}, from=1-1, to=1-2]
\end{tikzcd}
}%
\right]  \cong\Theta(W)
\]
for a uniquely defined object (in case 1) $W\in\mathsf{PCA}$ or (in case 2)
$W\in\mathsf{PGA}$. Concretely: $W$ has $Z^{\prime}$ as the underlying abelian
group, but equipped with the subspace topology as a subset in $Z$. Morphisms
between such objects of type d-c or d-cg in $\mathcal{LH}(\mathsf{LCA})$ lift
uniquely to morphisms in $\mathsf{PCA}$ resp. $\mathsf{PGA}$.
\end{lemma}

\begin{proof}
Define $W$ to be $Z^{\prime}$, but equipped with the subspace topology coming
from the inclusion $z$. Then the top horizontal arrow in the following diagram
is continuous by construction:%
\begin{equation}%
{
\begin{tikzcd}
	W && Z \\
	& cW
	\arrow["z"{inner sep=.8ex}, "\bullet"{marking}, from=1-1, to=1-3]
	\arrow["{w}"'{inner sep=.8ex}, "\bullet"{marking}, from=1-1, to=2-2]
	\arrow["h"', dashed, from=2-2, to=1-3]
\end{tikzcd}
}
\label{l_bz1}%
\end{equation}
In setting (1), the inclusion exhibits $W$ as a dense subgroup inside a
compact group, so Example \ref{example_2} shows that $W\in\mathsf{PCA}$. In
setting (2) $W$ is dense inside a group-theoretically compactly generated
group, so $W\in\mathsf{PGA}$ by Example \ref{example_3}. The left downward
diagonal arrow $w$ is the canonical map to the completion $cW$ of $W$. It is
injective with dense image. Since $Z$ is locally compact, it is also a
complete group \cite[Theorem 3.6.24]{MR2433295}. By the universal property of
complete groups, the arrow $z\colon W\rightarrow Z$ to a complete group must
uniquely factor over its completion, giving us the arrow $h$ in Diag.
\ref{l_bz1} \cite[Theorem 7.1.18]{MR4510389}. We note that%
\[
h\mid_{W}\colon W\longrightarrow z(W)
\]
(with $W$ regarded as a subset of $cW$, as $w$ is injective) is a topological
isomorphism by construction: It is injective as it agrees with the injective
map $z$. It is clearly surjective, hence bijective. It is continuous since $h$
is continuous. It is open since $W$ was topologized so that it carries
precisely the subspace topology as a subset of $Z$. Now, by \cite[Prop.
3.6.13]{MR2433295} and since $W$ is dense in $cW$, and $z(W)$ is dense in $Z$,
it follows that $h\mid_{W}$ induces an isomorphism on its completions, i.e.,
$h\colon cW\longrightarrow\overline{z(W)}=Z$ is an isomorphism.\footnote{We
have used that for a subgroup inside a complete group (like $Z$), completion
agrees with taking the closure \cite[Exercise 7.3.5]{MR4510389}. Since $z$ was
monic and epic, we know that $\overline{z(W)}=Z$.} Hence,%
\[
\Theta(W)=\left[
{
\begin{tikzcd}
	{W_d} & {cW}
	\arrow["w"{inner sep=.8ex}, "\bullet"{marking}, from=1-1, to=1-2]
\end{tikzcd}
}%
\right]  \overset{\sim}{\longrightarrow}\left[
{
\begin{tikzcd}
	{Z'} & Z
	\arrow["z"{inner sep=.8ex}, "\bullet"{marking}, from=1-1, to=1-2]
\end{tikzcd}
}%
\right]  \text{,}%
\]
where the isomorphism in the middle is induced from the isomorphism $(1,h)$.
We see that the category of objects of the shape $\Theta(W)$ for
$W\in\mathsf{PCA}$ (resp. $\mathsf{PGA}$) is equivalent to the full
subcategory of d-c (resp. d-cg) type objects, so for the rest of the proof we
may work in the strict image category\footnote{If the reader does not like
this, note that the previous construction was natural, so if you prefer to
work in $\mathcal{LH}(\mathsf{LCA})$ on the nose, pre- and post-compose each
morphism with the natural isomorphisms to present source and target as
$\Theta(-)$ for appropriate input objects.} $\Theta(-)$.

It remains to show that morphisms lift uniquely:\ Suppose we are given a
morphism $\left[
{
\begin{tikzcd}
	{Z'} & Z
	\arrow["z"{inner sep=.8ex}, "\bullet"{marking}, from=1-1, to=1-2]
\end{tikzcd}
}%
\right]  \longrightarrow\left[
{
\begin{tikzcd}
	{Y'} & Y
	\arrow["y"{inner sep=.8ex}, "\bullet"{marking}, from=1-1, to=1-2]
\end{tikzcd}
}%
\right]  $. Then by the construction in the first half of the proof, it is
given by a commutative square%
\[%
{
\begin{tikzcd}
	{W_d} && cW \\
	\\
	{W'_d} && {cW'}
	\arrow["z"{inner sep=.8ex}, "\bullet"{marking}, from=1-1, to=1-3]
	\arrow["{f'}"', from=1-1, to=3-1]
	\arrow["f", from=1-3, to=3-3]
	\arrow["y"'{inner sep=.8ex}, "\bullet"{marking}, from=3-1, to=3-3]
\end{tikzcd}
}%
\]
and we need to show that $f^{\prime}$ is continuous as a map $f^{\prime}\colon
W\rightarrow W^{\prime}$ with the topologizations as subspaces of $cW$ resp.
$cW^{\prime}$. However, this follows immediately from the map $f\colon
cW\rightarrow cW^{\prime}$ being continuous (i.e., the functoriality of
completion) and $f^{\prime}(W)\subseteq W^{\prime}$, which stems from the
commutativity of the square.
\end{proof}

Lemma \ref{lemma10} is not yet enough to prove essential surjectivity of
$\Theta$, because the left heart will contain many objects which are not of
type d-c or d-cg. We will refine the above argument and solve this problem in
\S \ref{sect_RewritingRules}.

\begin{theorem}
\label{thm_PrepareRoof_PCA}Suppose $X,Y\in\mathsf{PCA}$ [resp. $X,Y\in
\mathsf{PGA}$]. Suppose we are given a left roof%
\[
\left[
{
\begin{tikzcd}
	{X_d} & {cX}
	\arrow["\bullet"{marking}, from=1-1, to=1-2]
\end{tikzcd}
}%
\right]  \overset{(u^{\prime},u)}{\longleftarrow}\left[
{
\begin{tikzcd}
	{Z'} & Z
	\arrow[from=1-1, to=1-2]
\end{tikzcd}
}%
\right]  \overset{(v^{\prime},v)}{\longrightarrow}\left[
{
\begin{tikzcd}
	{Y_d} & {cY}
	\arrow["\bullet"{marking}, from=1-1, to=1-2]
\end{tikzcd}
}%
\right]
\]
in $\mathcal{LH}(\mathsf{LCA}_{\operatorname*{vf}})$ [resp. $\mathcal{LH}%
(\mathsf{LCA})$] with $(u^{\prime},u)$ given by a bicartesian square in
$\mathsf{LCA}_{\operatorname*{vf}}$ [resp.$\mathsf{LCA}$] and the leftmost and
rightmost object both of type d-c [resp. d-cg] (the middle object can be
anything in the left heart). Then the roof is equivalent to a roof%
\[
\left[
{
\begin{tikzcd}
	{X_d} & {cX}
	\arrow["\bullet"{marking}, from=1-1, to=1-2]
\end{tikzcd}
}%
\right]  \overset{(u_{new}^{\prime},u_{new})}{\longleftarrow}\left[
{
\begin{tikzcd}
	{W'} & W
	\arrow[from=1-1, to=1-2]
\end{tikzcd}
}%
\right]  \overset{(v_{new}^{\prime},v_{new})}{\longrightarrow}\left[
{
\begin{tikzcd}
	{Y_d} & {cY}
	\arrow["\bullet"{marking}, from=1-1, to=1-2]
\end{tikzcd}
}%
\right]
\]
\newline with the middle term also of type d-c [resp. d-cg]. If $(v^{\prime
},v)$ is the zero map, so is $(v_{new}^{\prime},v_{new})$.
\end{theorem}

\begin{proof}
Suppose we are given a left roof%
\[
\left[
{
\begin{tikzcd}
	{X_d} & {cX}
	\arrow["\bullet"{marking}, from=1-1, to=1-2]
\end{tikzcd}
}%
\right]  \overset{(u^{\prime},u)}{\longleftarrow}\left[
{
\begin{tikzcd}
	{Z'} & Z
	\arrow[from=1-1, to=1-2]
\end{tikzcd}
}%
\right]  \longrightarrow\left[
{
\begin{tikzcd}
	{Y_d} & {cY}
	\arrow["\bullet"{marking}, from=1-1, to=1-2]
\end{tikzcd}
}%
\right]
\]
in $\mathcal{LH}(\mathsf{LCA})$ with $(u^{\prime},u)$ inducing a bicartesian
square, as explained in Remark \ref{rmk_SchneidersDescriptionOfLH}. Firstly,
since $(u^{\prime},u)$ comes from a bicartesian square, Lemma \ref{lemma9}
guarantees that the arrow $z\colon Z^{\prime}\longrightarrow Z$ is both monic
and epic. We now begin with a general observation:

Suppose the roof has the shape%
\begin{equation}%
{
\begin{tikzcd}
	&&& U & U \\
	&&& {Z'} & Z \\
	{X_d} & cX &&&&& {Y_d} & cY
	\arrow["\sim", from=1-4, to=1-5]
	\arrow["{{i'}}"', hook, from=1-4, to=2-4]
	\arrow["i", hook, from=1-5, to=2-5]
	\arrow["z"{inner sep=.8ex}, "\bullet"{marking}, from=2-4, to=2-5]
	\arrow["{{u'}}"', from=2-4, to=3-1]
	\arrow["{{v'}}"', from=2-4, to=3-7]
	\arrow["u", from=2-5, to=3-2]
	\arrow["v", from=2-5, to=3-8]
	\arrow["\bullet"{marking}, from=3-1, to=3-2]
	\arrow["\bullet"{marking}, from=3-7, to=3-8]
\end{tikzcd}
}
\label{diag_4}%
\end{equation}
such that

\begin{enumerate}
\item $z$, restricted to $U\in\mathsf{LCA}$, is an isomorphism,

\item $i$ and $i^{\prime}$ are admissible monics in $\mathsf{LCA}$, and

\item the left and right legs of the roof map $U$ to zero on either side.
\end{enumerate}

Then the diagram%
\[%
{
\begin{tikzcd}
	& {Z^{\bullet}} \\
	{X^{\bullet}} & {Z^{\bullet}} & {Y^{\bullet}} \\
	& {Z^{\bullet}/U}
	\arrow[from=1-2, to=2-1]
	\arrow[from=1-2, to=2-3]
	\arrow["1"', from=2-2, to=1-2]
	\arrow[from=2-2, to=3-2]
	\arrow[from=3-2, to=2-1]
	\arrow[from=3-2, to=2-3]
\end{tikzcd}
}%
\]
(where we write $Z^{\bullet}$ as a contraction for
{
\begin{tikzcd}
	{Z'} & Z
	\arrow["\bullet"{marking}, from=1-1, to=1-2]
\end{tikzcd}
}
and analogously for $X,Y$) shows that the roof in Diag. \ref{diag_4} is
equivalent to the roof%
\[%
{
\begin{tikzcd}
	&&& {Z'/U} & {Z/U} \\
	{X_d} & cX &&&&& {Y_d} & {cY.}
	\arrow["{{z/U}}"{inner sep=.8ex}, "\bullet"{marking}, from=1-4, to=1-5]
	\arrow["{{{u'}}}"', from=1-4, to=2-1]
	\arrow["{{{v'}}}"', from=1-4, to=2-7]
	\arrow["u", from=1-5, to=2-2]
	\arrow["v", from=1-5, to=2-8]
	\arrow["\bullet"{marking}, from=2-1, to=2-2]
	\arrow["\bullet"{marking}, from=2-7, to=2-8]
\end{tikzcd}
}%
\]
The arrow
{
\begin{tikzcd}
	{Z'/U} & {Z/U}
	\arrow["\bullet"{marking}, from=1-1, to=1-2]
\end{tikzcd}
}
is again monic and epic by Lemma \ref{lemma9}. [In more detail: We have the
diagram%
\begin{equation}%
{
\begin{tikzcd}
	U && {Z'} && {Z'/U} \\
	&&& {{(\bigstar)}} \\
	U && Z && {Z/U}
	\arrow["{i'}", hook, from=1-1, to=1-3]
	\arrow["{z\mid_U}"', from=1-1, to=3-1]
	\arrow[two heads, from=1-3, to=1-5]
	\arrow[from=1-3, to=3-3]
	\arrow[from=1-5, to=3-5]
	\arrow["i"', hook, from=3-1, to=3-3]
	\arrow[two heads, from=3-3, to=3-5]
\end{tikzcd}
}
\label{l_3h}%
\end{equation}
with exact rows and by \cite[Prop. 2.12 (dual)]{MR2606234} the square
$(\bigstar)$ is bicartesian iff $z\mid_{U}$ is an isomorphism.]\bigskip

Following this mechanism, we now simplify the input roof.\bigskip

(Step 1)\footnote{This step does nothing in the d-c setting and can be
omitted.} Suppose $Z^{\prime}$ has a vector summand $\mathbb{R}^{n}$, and
assume $n$ is chosen maximally, say $Z^{\prime}\cong Z_{0}^{\prime}%
\oplus\mathbb{R}^{n}$. As the left leg of the roof comes from a bicartesian
square, we have the exact sequence%
\[
\mathbb{R}^{n}\oplus Z_{0}^{\prime}\hookrightarrow Z\oplus X_{d}%
\twoheadrightarrow cX\text{.}%
\]
As $\mathbb{R}^{n}$ is an injective object in $\mathsf{LCA}$ \cite[Theorem
3.2]{MR0215016}, it must also appear as a direct summand of the middle term
(and clearly as a direct summand of $Z$, since $X_{d}$ is discrete). We get
compatible admissible monics\newline%
\begin{equation}%
{
\begin{tikzcd}
	{\mathbb{R}^n} && {\mathbb{R}^n} \\
	\\
	{Z'} && Z
	\arrow["1", from=1-1, to=1-3]
	\arrow[hook, from=1-1, to=3-1]
	\arrow[hook, from=1-3, to=3-3]
	\arrow["z"'{inner sep=.8ex}, "\bullet"{marking}, from=3-1, to=3-3]
\end{tikzcd}
}
\label{diag_8}%
\end{equation}
as they are required for the reduction step explained in Diag. \ref{diag_4},
because the summands $\mathbb{R}^{n}$ are connected and continuous maps to the
discrete groups $X_{d}$ (resp. $Y_{d}$) must necessarily map them entirely to
zero. Since $X_{d}$ (resp. $Y_{d}$) are dense in $cX$ (resp. $cY$), it follows
that the maps to the Weil completions must also send them to zero, as the
maps, restricted to these summands, are zero on a dense subset. Hence, we may
henceforth assume that $Z^{\prime}$ is vector-free.\bigskip

(Step 2) Next, by the structure theorem for LCA groups \cite[Ch. 6, Theorem
25]{MR0442141} $Z^{\prime}$ has a compact (cl)open subgroup $C$. Analogously
to Diagram \ref{diag_8}, we obtain a commutative square%
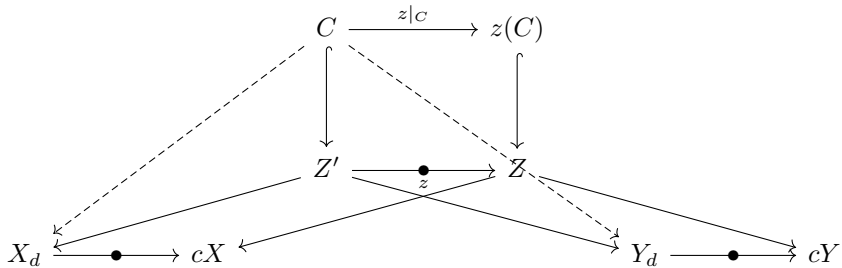
\[%
{
\begin{tikzcd}
	&&& C && {z(C)} \\
	\\
	&&& {Z'} && Z \\
	{X_d} && cX &&&& {Y_d} && cY
	\arrow["{z\mid_C}", from=1-4, to=1-6]
	\arrow[hook, from=1-4, to=3-4]
	\arrow[dashed, from=1-4, to=4-1]
	\arrow[dashed, from=1-4, to=4-7]
	\arrow[hook, from=1-6, to=3-6]
	\arrow["z"'{inner sep=.8ex}, "\bullet"{marking}, from=3-4, to=3-6]
	\arrow[from=3-4, to=4-1]
	\arrow[from=3-4, to=4-7]
	\arrow[from=3-6, to=4-3]
	\arrow[from=3-6, to=4-9]
	\arrow["\bullet"{marking}, from=4-1, to=4-3]
	\arrow["\bullet"{marking}, from=4-7, to=4-9]
\end{tikzcd}
}%
\]
However, $z\mid_{C}$ is injective (as the downward arrow and $z$ are
injective), surjective by construction, and since $C$ is compact, $z(C)$ is
compact and therefore closed in $Z$. We observe that $z\mid_{C}$ is a
continuous bijective map from a compact space to a Hausdorff space and
therefore an isomorphism of topological abelian groups. Now, the group $C$
need not map to zero under the left and right leg of the roof, but almost:
Under the dashed arrows $C\longrightarrow X_{d}$ (resp. $C\longrightarrow
Y_{d}$) the image of the compact group $C$ must be a \textit{finite} group.
Thus, $C$ has a finite-index subgroup $C^{\prime}\subset C$, restricted to
which the dashed arrows are both zero. As it is finite index, it will also be
compact and clopen in $Z^{\prime}$, so without loss of generality we may
assume to have picked $C:=C^{\prime}$ right from the start. But then we are
again in the situation of Eq. \ref{diag_4} for $U:=C$. In particular, the
construction loc. cit. replaces $\left[
{
\begin{tikzcd}
	{Z'} & {Z}
	\arrow["\bullet"{marking}, from=1-1, to=1-2]
\end{tikzcd}
}%
\right]  $ by $\left[
{
\begin{tikzcd}
	{Z'/C} & {Z/C}
	\arrow["\bullet"{marking}, from=1-1, to=1-2]
\end{tikzcd}
}%
\right]  $ and since $C$ was \textit{open} in $Z^{\prime}$, the quotient
$Z^{\prime}/C$ is now discrete. Hence, from now on we may assume that
$Z^{\prime}$ is discrete.\bigskip

(Step 3) Again by \cite[Ch. 6, Theorem 25]{MR0442141} the group $Z$ has a
(cl)open subgroup $C$ [which can be assumed compact in the d-c setting since
in $\mathsf{LCA}_{\operatorname*{vf}}$ all groups are vector-free, and which
will be compactly generated in the d-cg setting]. We obtain the commutative
diagram below on the left:
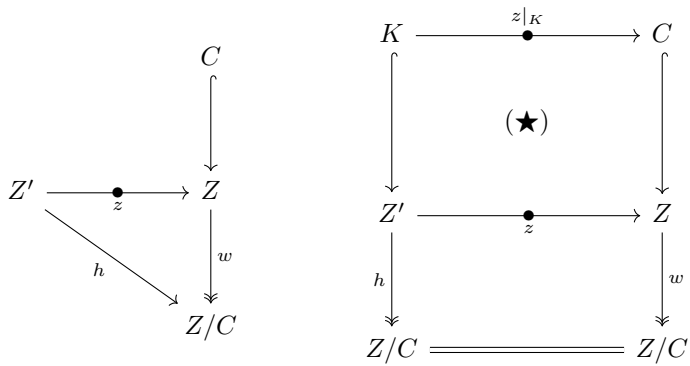
\[%
{
\begin{tikzcd}
	&& C \\
	\\
	{Z'} && Z \\
	\\
	&& {Z/C}
	\arrow[hook, from=1-3, to=3-3]
	\arrow["z"'{inner sep=.8ex}, "\bullet"{marking}, from=3-1, to=3-3]
	\arrow["h"', from=3-1, to=5-3]
	\arrow["w", two heads, from=3-3, to=5-3]
\end{tikzcd}
}%
\qquad\qquad%
{
\begin{tikzcd}
	K && C \\
	& {(\bigstar)} \\
	{Z'} && Z \\
	\\
	{Z/C} && {Z/C}
	\arrow["{z\mid_K}"{inner sep=.8ex}, "\bullet"{marking}, from=1-1, to=1-3]
	\arrow[hook, from=1-1, to=3-1]
	\arrow[hook, from=1-3, to=3-3]
	\arrow["z"'{inner sep=.8ex}, "\bullet"{marking}, from=3-1, to=3-3]
	\arrow["h"', two heads, from=3-1, to=5-1]
	\arrow["w", two heads, from=3-3, to=5-3]
	\arrow[equals, from=5-1, to=5-3]
\end{tikzcd}
}%
\]
Since $Z^{\prime}$ is dense in $Z$ (embedded along $z$) and $w$ continuous,
Lemma \ref{lemma1} shows that $h$ maps $Z^{\prime}$ to a dense subset of the
\textit{discrete} space $Z/C$. It follows that $h$ is surjective. And since
$Z^{\prime}$ is discrete, $h$ is tautologically open and thus an admissible
epic in $\mathsf{LCA}$. Define $K:=\ker(h)$. We obtain the diagram above on
the right. Again by \cite[Prop. 2.12 (dual)]{MR2606234} the square
$(\bigstar)$ is bicartesian (this is another instance of the same argument as
for Diagram \ref{l_3h}), so by Lemma \ref{lemma9} $z\mid_{K}$ is monic and
epic. The diagram%
\[%
{
\begin{tikzcd}
	&& {\left[K \overset{\bullet}{\rightarrow} C \right]} \\
	\\
	{X^{\bullet}} && {\left[K \overset{\bullet}{\rightarrow} C \right]}
&& {Y^{\bullet}} \\
	\\
	&& {\left[Z' \overset{\bullet}{\rightarrow} Z \right]}
	\arrow["{{(u'\mid_K,u\mid_C)}}"', from=1-3, to=3-1]
	\arrow["{{(v'\mid_K,v\mid_C)}}", from=1-3, to=3-5]
	\arrow[equals, from=3-3, to=1-3]
	\arrow["{{(\bigstar)}}", from=3-3, to=5-3]
	\arrow["{{(u',u)}}", from=5-3, to=3-1]
	\arrow["{{(v',v)}}"', from=5-3, to=3-5]
\end{tikzcd}
}%
\]
shows that our roof is equivalent to the roof with middle term $\left[
{
\begin{tikzcd}
	{K} & {C}
	\arrow["\bullet"{marking}, from=1-1, to=1-2]
\end{tikzcd}
}%
\right]  $, where $K$ is discrete (it is a subgroup of the discrete group
$Z^{\prime}$), and $C$ is compact [in the d-c setting] or compactly generated
[in the d-cg setting]. The arrow $(u^{\prime}\mid_{K},u\mid_{C})$ comes from a
bicartesian square by the pasting lemma for bicartesian squares: It arises
from concatenating the bicartesian square $(\bigstar)$ with the bicartesian
square coming from $(u^{\prime},u)$.\bigskip

Finally, note that if $(v^{\prime},v)$ was the zero map to start with, the
right leg of the resulting roof will still be the zero map. This completes the proof.
\end{proof}

\begin{lemma}
\label{lemma_f_full}The functor $\Theta$, in either version%
\[
\mathsf{PCA}/\mathsf{Ab}_{\operatorname*{fin}}\longrightarrow\mathcal{LH}%
(\mathsf{LCA}_{\operatorname*{vf}})\qquad\text{resp.}\qquad\mathsf{PGA}%
/\mathsf{Ab}_{\operatorname*{fg}}\longrightarrow\mathcal{LH}(\mathsf{LCA}%
)\text{,}%
\]
is full.
\end{lemma}

\begin{proof}
Suppose $X,Y$ are objects in $\mathsf{PCA}$ [resp. $\mathsf{PGA}$]. Suppose we
are given a left roof $X\overset{s}{\longleftarrow}%
Z\overset{f}{\longrightarrow}Y$ in the left heart $\mathcal{LH}(\mathsf{LCA}%
_{\operatorname*{vf}})$ [resp. $\mathcal{LH}(\mathsf{LCA})$]. We need to
produce a pre-image under $\Theta$ in the source category. By Theorem
\ref{thm_PrepareRoof_PCA}, this roof has an equivalent representative such
that $Z$ is of type d-c [resp. d-cg] and without loss of generality, we may
assume this is the roof we started with. By Lemma \ref{lemma10} our roof can
be lifted to a left roof%
\begin{equation}%
{
\begin{tikzcd}
	& {{\Theta}(Z)} \\
	{{\Theta}(X)} && {{\Theta}(Y).}
	\arrow["{s'}"', from=1-2, to=2-1]
	\arrow["{f'}", from=1-2, to=2-3]
\end{tikzcd}
}
\label{l_bz4}%
\end{equation}
The morphism $s^{\prime}$ is, by the construction of the cited lemma, induced
from the left downward arrow in%
\begin{equation}%
{
\begin{tikzcd}
	{Z_d} && cZ \\
	\\
	{X_d} && cX,
	\arrow["\bullet"{marking}, from=1-1, to=1-3]
	\arrow["{s'}"', from=1-1, to=3-1]
	\arrow[from=1-3, to=3-3]
	\arrow["\bullet"{marking}, from=3-1, to=3-3]
\end{tikzcd}
}
\label{l_bz5}%
\end{equation}
but re-topologized with the induced subspace topologies of $cZ$ and $cX$.
Since the left heart arises as the localization inverting bicartesian squares
(Remark \ref{rmk_SchneidersDescriptionOfLH}), the square in Diag. \ref{l_bz5}
is bicartesian and therefore Lemma \ref{lemmaB2} shows that $s^{\prime}$ has
finite kernel and cokernel, i.e., $s^{\prime}$ is inverted in $\mathsf{PCA}%
/\mathsf{Ab}_{\operatorname*{fin}}$ [resp. $s^{\prime}$ has finitely generated
discrete kernel and cokernel, i.e., $s^{\prime}$ is inverted in $\mathsf{PGA}%
/\mathsf{Ab}_{\operatorname*{fg}}$], so that the roof in Diag. \ref{l_bz4} is
a well-defined left roof in the source category. Thus, we have found a
pre-image of our morphism.
\end{proof}

\begin{lemma}
\label{lemma_f_faithful}The functor $\Theta$, in either version%
\[
\mathsf{PCA}/\mathsf{Ab}_{\operatorname*{fin}}\longrightarrow\mathcal{LH}%
(\mathsf{LCA}_{\operatorname*{vf}})\qquad\text{resp.}\qquad\mathsf{PGA}%
/\mathsf{Ab}_{\operatorname*{fg}}\longrightarrow\mathcal{LH}(\mathsf{LCA}%
)\text{,}%
\]
is faithful.
\end{lemma}

\begin{proof}
Suppose $X,Y$ are objects in $\mathsf{PCA}$ [resp. $\mathsf{PGA}$]. Suppose we
are given a left roof $X\overset{s}{\longleftarrow}Z\longrightarrow Y$ in
$\mathsf{PCA}/\mathsf{Ab}_{\operatorname*{fin}}$ [resp. $\mathsf{PGA}%
/\mathsf{Ab}_{\operatorname*{fg}}$]. Suppose that $\Theta$ sends this roof to
zero. This means that there is an equivalence of left roofs%
\[%
{
\begin{tikzcd}
	&& {{\Theta}(Z)} \\
	\\
	{{\Theta}(X)} && S && {{\Theta}(Y)} \\
	\\
	&& N
	\arrow["{{{s}}}"', from=1-3, to=3-1]
	\arrow[from=1-3, to=3-5]
	\arrow[from=3-3, to=1-3]
	\arrow[from=3-3, to=5-3]
	\arrow[from=5-3, to=3-1]
	\arrow["0"', from=5-3, to=3-5]
\end{tikzcd}
}%
\]
in the left heart, for some object $N$ and suitable arrows. We can simplify
this to the commutative diagram%
\[%
{
\begin{tikzcd}
	&& S \\
	\\
	&& {{\Theta}(Z)} \\
	{{\Theta}(X)} &&&& {{\Theta}(Y),}
	\arrow[from=1-3, to=3-3]
	\arrow["{{s'}}"', from=1-3, to=4-1]
	\arrow["0", from=1-3, to=4-5]
	\arrow["{{{s}}}"', from=3-3, to=4-1]
	\arrow[from=3-3, to=4-5]
\end{tikzcd}
}%
\]
where both $s$ and $s^{\prime}$ are invertible in the localization, i.e., they
are induced from bicartesian squares (Remark
\ref{rmk_SchneidersDescriptionOfLH}). By Theorem \ref{thm_PrepareRoof_PCA}, we
may assume that $S$ is of the shape $\left[
{
\begin{tikzcd}
	{S_d} & {cS}
	\arrow["\bullet"{marking}, from=1-1, to=1-2]
\end{tikzcd}
}%
\right]  $ for $S$ in $\mathsf{PCA}$ [resp. $\mathsf{PGA}$] and this does not
destroy the property that the morphism along the right leg is zero. Since the
right leg of the upper roof (with apex $S$) comes from the zero morphism, it
is induced from the solid arrows in the square%
\begin{equation}%
{
\begin{tikzcd}
	{S_d} && cS \\
	\\
	{Y_d} && {cY,}
	\arrow["s"{inner sep=.8ex}, "\bullet"{marking}, from=1-1, to=1-3]
	\arrow["{{f'}}"', from=1-1, to=3-1]
	\arrow["\delta"', dashed, from=1-3, to=3-1]
	\arrow["f", from=1-3, to=3-3]
	\arrow["y"'{inner sep=.8ex}, "\bullet"{marking}, from=3-1, to=3-3]
\end{tikzcd}
}
\label{l_bz6}%
\end{equation}
but beware that the left heart is defined as the localization of the
\textit{homotopy category} $\mathcal{K}(\mathsf{LCA}_{\operatorname*{vf}})$
[resp. $\mathcal{K}(\mathsf{LCA})$] of complexes modulo bicartesian squares
(Remark \ref{rmk_SchneidersDescriptionOfLH}), i.e., the morphism being zero in
$\mathcal{K}(-)$ does not mean that both morphisms $f,f^{\prime}$ are
literally zero in $\mathsf{LCA}_{\operatorname*{vf}}$ [resp. $\mathsf{LCA}%
$].\ Instead, it means that there exists a contracting homotopy $\delta$
(displayed using the dashed arrow in Diag. \ref{l_bz6}), making the diagram
commute.\ However, since $cS$ is compact, say $cS\simeq C$ for $C$ compact
[resp. group-theoretically compactly generated so that (by Lemma
\ref{lemma11}) $cS\simeq\mathbb{R}^{n}\oplus\mathbb{Z}^{m}\oplus C$ for $C$
compact], $\delta(\mathbb{R}^{n})$ must be zero since $\mathbb{R}^{n}$ is
connected and $Y_{d}$ discrete, $\delta(\mathbb{Z}^{m})$ is finitely
generated, and $\delta(C)$ finite, again since $Y_{d}$ is discrete. Hence,
$F:=\delta(S)$ is finite in the d-c setting [resp. finitely generated in the
d-cg setting] subgroup of $Y_{d}$. As the upper left solid triangle in Diag.
\ref{l_bz6} commutes, it follows that $f^{\prime}$ factors over $F$. As the
lower right solid triangle commutes and the lower horizontal arrow is
injective, it also follows that $f$ factors over $F$. We arrive at%
\[%
{
\begin{tikzcd}
	& {S_d} &&& cS \\
	F &&&&& F \\
	& {Y_d} &&& {cY.}
	\arrow["s"{inner sep=.8ex}, "\bullet"{marking}, from=1-2, to=1-5]
	\arrow[dashed, from=1-2, to=2-1]
	\arrow["{{{{f'}}}}"'{pos=0.3}, from=1-2, to=3-2]
	\arrow[dashed, from=1-5, to=2-6]
	\arrow["\delta"{pos=0.7}, from=1-5, to=3-2]
	\arrow["f"{pos=0.3}, from=1-5, to=3-5]
	\arrow["1"'{pos=0.7}, dashed, from=2-1, to=2-6]
	\arrow[dashed, from=2-1, to=3-2]
	\arrow[dashed, from=2-6, to=3-5]
	\arrow["y"'{inner sep=.8ex}, "\bullet"{marking}, from=3-2, to=3-5]
\end{tikzcd}
}%
\]
We have obtained a factorization%
\begin{equation}
\Theta(S)\longrightarrow\Theta(F)\longrightarrow\Theta(Y)
\end{equation}
for the arrow $g$ since $\Theta(F)=\left[
{
\begin{tikzcd}
	{F_d} & {cF}
	\arrow["\bullet"{marking}, from=1-1, to=1-2]
\end{tikzcd}
}%
\right]  $, and $F_{d}=cF$, as $F$ is both discrete and complete, and by
lifting the morphisms using Lemma \ref{lemma10}, this can be promoted to a
factorization%
\begin{equation}
S\longrightarrow F\longrightarrow Y\text{.} \label{l_6}%
\end{equation}
Since the zero map $0\overset{0}{\rightarrow}F$ is an isomorphism in the
localization $\mathsf{PCA}/\mathsf{Ab}_{\operatorname*{fin}}$ (its kernel and
cokernel lie in $\mathsf{Ab}_{\operatorname*{fin}}$) in the d-c setting [resp.
$0\overset{0}{\rightarrow}F$ is an isomorphism $\mathsf{PGA}/\mathsf{Ab}%
_{\operatorname*{fg}}$ in the d-cg setting], it follows that the arrow in Eq.
\ref{l_6} factors over a zero object, hence must be zero. Thus, the left roof
that our proof had started with, had the shape $X\overset{s}{\longleftarrow
}Z\overset{0}{\longrightarrow}Y$. Specifically, the diagram%
\[%
{
\begin{tikzcd}
	& Z \\
	X & Z & Y \\
	& X
	\arrow["s"', from=1-2, to=2-1]
	\arrow["{g=0}", from=1-2, to=2-3]
	\arrow["1"', from=2-2, to=1-2]
	\arrow["s", from=2-2, to=3-2]
	\arrow["1", from=3-2, to=2-1]
	\arrow["0"', from=3-2, to=2-3]
\end{tikzcd}
}%
\]
is the equivalence relation identifying the input left roof with the roof of
the zero map.
\end{proof}

\section{Rewriting Rules for Objects\label{sect_RewritingRules}}

\begin{proposition}
\label{prop_ess_surjectivity}The functor $\Theta$ is essentially surjective
onto the torsion class

\begin{enumerate}
\item $\mathsf{T}$ of $(\mathsf{T},\mathsf{LCA}_{\operatorname*{vf}})$ inside
$\mathcal{LH}(\mathsf{LCA}_{\operatorname*{vf}})$ in the $\mathsf{PCA}$
variant, resp.

\item $\mathsf{T}$ of $(\mathsf{T},\mathsf{LCA})$ inside $\mathcal{LH}%
(\mathsf{LCA})$ in the $\mathsf{PGA}$ variant.
\end{enumerate}
\end{proposition}

We split this proof into many little steps, but keep the notation uniform
without redefining each term as we jump from lemma to lemma. We treat the
cases of $\mathcal{LH}(\mathsf{LCA}_{\operatorname*{vf}})$ vs. $\mathcal{LH}%
(\mathsf{LCA})$ side by side.

\begin{lemma}
[Step 1]\label{lemma_S1}Every object in $\mathsf{T}$ is isomorphic to an
object of the shape%
\[
\left[
{
\begin{tikzcd}
	{X'} && X
	\arrow["x"{inner sep=.8ex}, "\bullet"{marking}, from=1-1, to=1-3]
\end{tikzcd}
}%
\right]
\]
with $X^{\prime}\simeq\mathbb{R}^{n}\oplus D^{\prime}$ and $D^{\prime}$ discrete.
\end{lemma}

\begin{proof}
By Prop. \ref{prop_QAbTorsionPairInLeftHeart} every object in $\mathsf{T}$ has
the shape%
\[
\left[
{
\begin{tikzcd}
	{X'} && X
	\arrow["x"{inner sep=.8ex}, "\bullet"{marking}, from=1-1, to=1-3]
\end{tikzcd}
}%
\right]
\]
for $X^{\prime},X\in\mathsf{LCA}$ (resp. $\mathsf{LCA}_{\operatorname*{vf}}$).
(Step 1) By the structure theorem for LCA groups there is a (cl)open subgroup
$\mathbb{R}^{n}\oplus C^{\prime}$ with $C^{\prime}$ compact in $X^{\prime}$.
In particular, $D^{\prime}:=X^{\prime}/(\mathbb{R}^{n}\oplus C^{\prime})$ is
discrete. Then the image $x(C^{\prime})$ in $X$ is also compact, hence closed
as $X$ is Hausdorff. Moreover, since $x$ is injective, the map $C^{\prime
}\longrightarrow x(C^{\prime})$ is injective and clearly surjective. As a
continuous bijective map from a compact space to a Hausdorff space,
$x\mid_{C^{\prime}}$ is an isomorphism. Hence, the square%
\[%
{
\begin{tikzcd}
	{C'} && {x(C')} \\
	\\
	{X'} && X
	\arrow["{{x\mid_{C' }}}", from=1-1, to=1-3]
	\arrow[hook, from=1-1, to=3-1]
	\arrow[hook, from=1-3, to=3-3]
	\arrow["x"{inner sep=.8ex}, "\bullet"{marking}, from=3-1, to=3-3]
\end{tikzcd}
}%
\]
exhibits a quasi-isomorphism $\left[
{
\begin{tikzcd}
	{X'} && X
	\arrow["x"{inner sep=.8ex}, "\bullet"{marking}, from=1-1, to=1-3]
\end{tikzcd}
}%
\right]  \overset{\sim}{\longrightarrow}\left[
{
\begin{tikzcd}
	{X'/C'} && {X/C'}
	\arrow["x/C'"{inner sep=.8ex}, "\bullet"{marking}, from=1-1, to=1-3]
\end{tikzcd}
}%
\right]  $ and from now on we may assume that $X^{\prime}\simeq\mathbb{R}%
^{n}\oplus D^{\prime}$ with $D^{\prime}$ (as defined above) discrete.
\end{proof}

\begin{lemma}
[Step 2]\label{lemma_S2}Every object in $\mathsf{T}$ is isomorphic to an
object of the shape%
\[
\left[
{
\begin{tikzcd}
	{X'} && X
	\arrow["x"{inner sep=.8ex}, "\bullet"{marking}, from=1-1, to=1-3]
\end{tikzcd}
}%
\right]
\]
with $X^{\prime}\simeq\mathbb{R}^{n}\oplus D^{\prime}$ and $D^{\prime}$
discrete, and simultaneously $X\simeq\mathbb{R}^{m}\oplus C$ with $C$ compact.
\end{lemma}

\begin{proof}
Analogously, pick a (cl)open subgroup $\mathbb{R}^{m}\oplus C$ with $C$
compact in $X$. Then the quotient $D:=X/(\mathbb{R}^{m}\oplus C)$ is discrete.
The arrow $h$ in%
\begin{equation}%
{
\begin{tikzcd}
	{\overbrace{\mathbb{R}^n \oplus D'}^{X'}} && X \\
	\\
	&& D
	\arrow["x"{inner sep=.8ex}, "\bullet"{marking}, from=1-1, to=1-3]
	\arrow["h"', from=1-1, to=3-3]
	\arrow["w", two heads, from=1-3, to=3-3]
\end{tikzcd}
}
\label{diag_10}%
\end{equation}
must send the real line factor $\mathbb{R}^{n}$ in $X^{\prime}\simeq
\mathbb{R}^{n}\oplus D^{\prime}$ to zero since it is connected and $D$ is
discrete and therefore $h$ factors over the quotient%
\begin{equation}
h\colon X^{\prime}\twoheadrightarrow D^{\prime}\overset{q}{\rightarrow
}D\text{.} \label{l_9}%
\end{equation}
Moreover, $X^{\prime}$ is dense in $X$, so $h(X^{\prime})$ must be dense in
$D$ by Lemma \ref{lemma1}. As $D$ is discrete, this means that $h$ is
surjective. Hence, the map $q$ in Eq. \ref{l_9} is a surjective map between
discrete groups and therefore tautologically an admissible epic. Thus, $h$ is
an admissible epic in Diag. \ref{diag_10}. Hence, we obtain the lower square
in the diagram below%
\[%
{
\begin{tikzcd}
	{\mathbb{R}^n \oplus\operatorname{ker}(q)} && {\mathbb{R}^m \oplus C} \\
	\\
	{\overbrace{\mathbb{R}^n \oplus D'}^{X'}} && X \\
	\\
	D && D
	\arrow["{\tilde{x}}"{inner sep=.8ex}, "\bullet"{marking}, from=1-1, to=1-3]
	\arrow["{1\oplus\iota}"', from=1-1, to=3-1]
	\arrow[from=1-3, to=3-3]
	\arrow["x"{inner sep=.8ex}, "\bullet"{marking}, from=3-1, to=3-3]
	\arrow["h"', two heads, from=3-1, to=5-1]
	\arrow["w", two heads, from=3-3, to=5-3]
	\arrow["1"', from=5-1, to=5-3]
\end{tikzcd}
}%
\]
and conclude that the top square is bicartesian. By Lemma \ref{lemma9} the
arrow $\tilde{x}$ is a monic and epic, because $x$ is. The upper square
exhibits a quasi-isomorphism%
\[
\left[
{
\begin{tikzcd}
	{X'} && X
	\arrow["x"{inner sep=.8ex}, "\bullet"{marking}, from=1-1, to=1-3]
\end{tikzcd}
}%
\right]  \overset{\sim}{\longrightarrow}\left[
{
\begin{tikzcd}
	{\mathbb{R}^n \oplus\operatorname{ker}(q)} && {\mathbb{R}^m \oplus C}
	\arrow["\tilde{x}'"{inner sep=.8ex}, "\bullet"{marking}, from=1-1, to=1-3]
\end{tikzcd}
}%
\right]  \text{,}%
\]
so from now on we may assume that $X^{\prime}$ is a direct sum of
$\mathbb{R}^{n}$ with some discrete group, and $X$ is a direct sum of
$\mathbb{R}^{m}$ with some compact group.
\end{proof}

\begin{remark}
There is an alternative way to prove Lemma \ref{lemma_S2}: Start with Lemma
\ref{lemma_S1}. Then take the Pontryagin dual of $\mathbb{R}^{n}\oplus
D^{\prime}\longrightarrow X$. Then use the same strategy as in the proof of
Lemma \ref{lemma_S1} for this dual $X^{\vee}\longrightarrow(\mathbb{R}%
^{n})^{\vee}\oplus D^{\prime\vee}$ to quasi-isomorphically replace $X^{\vee}$
by some $\mathbb{R}^{\ell}\oplus D^{\prime\prime}$ with $D^{\prime\prime}$
discrete. A further dualization yields $\mathbb{R}^{n}\oplus D^{\prime
}\longrightarrow X^{\prime\prime}$, where $X^{\prime\prime}\simeq
\mathbb{R}^{\ell}\oplus D^{\prime\prime\vee}$ and $D^{\prime\prime\vee}$ is
compact. Keeping track of how this changes the other end of the complex, gives
an alternative proof.
\end{remark}

The proof of Prop. \ref{prop_ess_surjectivity} for the variant%
\[
\Theta\colon\mathsf{PCA}\longrightarrow\mathcal{LH}(\mathsf{LCA}%
_{\operatorname*{vf}})
\]
following immediately from Lemma \ref{lemma_S2}: We must have $n=m=0$ since no
real line summands can occur in $\mathsf{LCA}_{\operatorname*{vf}}$, so the
object in Eq. \ref{l_10} is of type d-c and by Lemma \ref{lemma10} we have
$\left[
{
\begin{tikzcd}
	{\mathbb{R}^n \oplus D'} && {\mathbb{R}^m \oplus C}
	\arrow["x"{inner sep=.8ex}, "\bullet"{marking}, from=1-1, to=1-3]
\end{tikzcd}
}%
\right]  =\Theta(D_{\subseteq C}^{\prime})$, where $D_{\subseteq C}^{\prime}$
is $D^{\prime}$, equipped with the subspace topology of $C$.

For the rest of the section, we consider the remaining case $\Theta
\colon\mathsf{PGA}\longrightarrow\mathcal{LH}(\mathsf{LCA})$.

\begin{lemma}
[Step 3]\label{lemma_S3}Any object%
\begin{equation}
\left[
{
\begin{tikzcd}
	{\mathbb{R}^n \oplus D'} && {\mathbb{R}^m \oplus C}
	\arrow["x"{inner sep=.8ex}, "\bullet"{marking}, from=1-1, to=1-3]
\end{tikzcd}
}%
\right]  \label{l_10}%
\end{equation}
with $D^{\prime}$ discrete and $C$ compact is quasi-isomorphic to one of the
same shape, but additionally $x(\mathbb{R}^{n})\subseteq C$.
\end{lemma}

\begin{proof}
The set-theoretic image $x(\mathbb{R}^{n})$ inside the right side must be
connected since $\mathbb{R}^{n}$ is connected. By Lemma \ref{lemma1} its
closure $\overline{x(\mathbb{R}^{n})}$ inside $\mathbb{R}^{m}\oplus C$ is also
connected. We get the commutative diagram%
\[%
{
\begin{tikzcd}
	{\mathbb{R}^n} && {\overline{x(\mathbb{R}^n)}} \\
	\\
	{\mathbb{R}^n \oplus D'} && {\mathbb{R}^m \oplus C.}
	\arrow["{x\mid_{{\mathbb{R}}^n }}", from=1-1, to=1-3]
	\arrow[hook, from=1-1, to=3-1]
	\arrow[hook, from=1-3, to=3-3]
	\arrow["x"{inner sep=.8ex}, "\bullet"{marking}, from=3-1, to=3-3]
\end{tikzcd}
}%
\]
By \cite[Cor. 14.2.11]{MR4510389} we must have $\overline{x(\mathbb{R}^{n}%
)}\simeq\mathbb{R}^{\ell}\oplus\tilde{C}$ for $\tilde{C}$ compact connected.
We obtain the diagram%
\[%
{
\begin{tikzcd}
	&& {\tilde{C}} \\
	\\
	{\mathbb{R}^n} && {\overline{x(\mathbb{R}^n)}} \\
	\\
	&& {\mathbb{R}^{\ell},}
	\arrow[hook, from=1-3, to=3-3]
	\arrow[from=3-1, to=3-3]
	\arrow["t"', from=3-1, to=5-3]
	\arrow["p", shift left=2, two heads, from=3-3, to=5-3]
	\arrow["{{i_0}}", shift left=2, hook', from=5-3, to=3-3]
\end{tikzcd}
}%
\]
where $i_{0}$ is a splitting of the projector onto the summand $\mathbb{R}%
^{\ell}$. Since $x(\mathbb{R}^{n})$ is dense in $\overline{x(\mathbb{R}^{n})}%
$, the set-theoretic image of $t$ is dense in $\mathbb{R}^{\ell}$ by Lemma
\ref{lemma1}. We know that $t$ is a continuous abelian group homomorphism
$\mathbb{R}^{n}\rightarrow\mathbb{R}^{\ell}$. It is easy to see that it must
be an $\mathbb{R}$-linear map then (first show that it is a $\mathbb{Q}%
$-linear map and then use continuity to show $\mathbb{R}$-linearity). But then
its image must be a linear subspace of $\mathbb{R}^{\ell}$, so it can only
have dense image if it is surjective. In particular, $n\geq\ell$. We obtain a
closed immersion as the composition%
\[
\mathbb{R}^{\ell}\overset{i_{0}}{\hookrightarrow}\overline{f(\mathbb{R}^{n}%
)}\hookrightarrow X\text{,}%
\]
where $i_{0}$ is the split inverse of $q$ (and thus a closed immersion) and
the second arrow is the (evidently closed immersion) of the closure
$\overline{f(\mathbb{R}^{n})}$ inside $\mathbb{R}^{m}\oplus C$. Call this
composition $i$. Then we obtain the diagram%
\begin{equation}%
{
\begin{tikzcd}
	{\mathbb{R}^{\ell}} && {\mathbb{R}^{\ell}} \\
	&&& {\overline{f(\mathbb{R}^n)}} \\
	{\mathbb{R}^n\oplus D'} && X \\
	\\
	{\mathbb{R}^{n-\ell}\oplus D'} && {X/{\mathbb{R}^{\ell}}}
	\arrow["1", from=1-1, to=1-3]
	\arrow["{{i'}}"', hook, from=1-1, to=3-1]
	\arrow[dashed, from=1-3, to=2-4]
	\arrow["i", hook, from=1-3, to=3-3]
	\arrow[dashed, hook, from=2-4, to=3-3]
	\arrow["x"{inner sep=.8ex}, "\bullet"{marking}, from=3-1, to=3-3]
	\arrow["p"', two heads, from=3-1, to=5-1]
	\arrow[two heads, from=3-3, to=5-3]
	\arrow["{{x/{\mathbb{R}^{\ell}}}}"'{inner sep=.8ex}, "\bullet"{marking}%
, from=5-1, to=5-3]
\end{tikzcd}
}
\label{diag_11}%
\end{equation}
with exact columns, where $i^{\prime}$ is the closed immersion of the subspace
$\mathbb{R}^{m}$ we had found into the direct summand $\mathbb{R}^{n}$ of
$X^{\prime}$. Since the top horizontal arrow in Diag. \ref{diag_11} is an
isomorphism, the lower square is bicartesian and by Lemma \ref{lemma9} the
arrow $x/\mathbb{R}^{\ell}$ induced on the quotients is also a monic and epic.
\end{proof}

\begin{lemma}
[Step 4]\label{lemma_S4}Every object of the shape%
\[
\left[
{
\begin{tikzcd}
	{{\mathbb{R}^n} \oplus D'} && {{\mathbb{R}^m} \oplus C}
	\arrow["x"{inner sep=.8ex}, "\bullet"{marking}, from=1-1, to=1-3]
\end{tikzcd}
}%
\right]
\]
with $D^{\prime}$ discrete, $C$ compact, and $x(\mathbb{R}^{n})\subseteq C$,
is quasi-isomorphic to%
\[
\left[
{
\begin{tikzcd}
	{{\mathbb{R}^n} \oplus D'_{new}} &&  { C_{new}}
	\arrow["x"{inner sep=.8ex}, "\bullet"{marking}, from=1-1, to=1-3]
\end{tikzcd}
}%
\right]
\]
with $D_{new}^{\prime}$ discrete and $C_{new}$ compact. The value $n$ does not
change under this quasi-isomorphism.
\end{lemma}

\begin{proof}
The assumption $x(\mathbb{R}^{n})\subseteq C$ means that $x$ factors as in
\[%
{
\begin{tikzcd}
	{\mathbb{R}^n} && C \\
	\\
	{\mathbb{R}^n \oplus D'} && {\mathbb{R}^m \oplus C,}
	\arrow["{x\mid_{\mathbb{R}^n}}", from=1-1, to=1-3]
	\arrow[hook, from=1-1, to=3-1]
	\arrow[hook, from=1-3, to=3-3]
	\arrow["x"'{inner sep=.8ex}, "\bullet"{marking}, from=3-1, to=3-3]
\end{tikzcd}
}%
\]
but a priori we know very little about the nature of the restriction
$x\mid_{\mathbb{R}^{n}}$. Consider the composition $j$ in%
\[%
{
\begin{tikzcd}
	{\mathbb{R}^n \oplus D'} && {\mathbb{R}^m \oplus C} \\
	\\
	&& {\mathbb{R}^m}
	\arrow["x"'{inner sep=.8ex}, "\bullet"{marking}, from=1-1, to=1-3]
	\arrow["j"', from=1-1, to=3-3]
	\arrow["q", two heads, from=1-3, to=3-3]
\end{tikzcd}
}%
\]
Since $x(\mathbb{R}^{n})\subseteq C$ and $q$ is the projection of the opposite
direct summand, we see that $j(\mathbb{R}^{n})=0$. Since $x$ is epic and thus
has dense image, we deduce from Lemma \ref{lemma1} that $\mathbb{R}^{n}\oplus
D^{\prime}$ is dense in $\mathbb{R}^{m}$. But since $j(\mathbb{R}^{n})=0$,
already $j(D^{\prime})$ must be dense in $\mathbb{R}^{m}$. Now successively
pick elements $x_{1},\ldots,x_{m}\in D^{\prime}$ such that $j(x_{i})$ does not
lie in the $\mathbb{R}$-span of $j(x_{1}),\ldots,j(x_{i-1})$ inside
$\mathbb{R}^{m}$. If this were impossible, it would show that the image of
$D^{\prime}$ inside $\mathbb{R}^{m}$ lies inside a vector subspace of
dimension $<m$, but this contradicts $j$ having dense image in $\mathbb{R}%
^{m}$. We observe that%
\[
\dim\mathbb{R}\left\langle j(x_{1}),\ldots,j(x_{i})\right\rangle =i
\]
for all $i\leq m$. Let $E\subseteq D^{\prime}$ be the subgroup generated by
$x_{1},\ldots,x_{m}$. This group must be torsion-free because if some non-zero
vector%
\[
v:=\sum a_{i}x_{i}%
\]
satisfies $k\cdot v=0$ in $D^{\prime}$ for some positive integer $k$, then
$kj(v)=0$ in $\mathbb{R}^{m}$, but the vectors $j(x_{i})$ were chosen linearly
independent. The structure theorem for finitely generated abelian groups shows
that $E\simeq\mathbb{Z}^{m}$. It follows that $j(E)$ is a full rank free
$\mathbb{Z}$-lattice in $\mathbb{R}^{m}$. We obtain the following commutative
diagram%
\begin{equation}%
{
\begin{tikzcd}
	E &&& {x(E)} && C \\
	\\
	&& {D'} &&& {\mathbb{R}^m \oplus C} \\
	\\
	&&&&& {\mathbb{R}^m}
	\arrow["{x\mid_E}", from=1-1, to=1-4]
	\arrow[hook, from=1-1, to=3-3]
	\arrow["i", from=1-4, to=3-6]
	\arrow["{\overline{i}}"'{pos=0.2}, dashed, hook', from=1-4, to=5-6]
	\arrow[hook, from=1-6, to=3-6]
	\arrow["{x\mid_{D'}}"'{inner sep=.8ex}, "\bullet"{marking}, from=3-3, to=3-6]
	\arrow["j"', from=3-3, to=5-6]
	\arrow["q", two heads, from=3-6, to=5-6]
\end{tikzcd}
}
\label{diag_12}%
\end{equation}
Here $x(E)$ is the set-theoretic image of $E$ in $\mathbb{R}^{m}\oplus C$ (so
far we do not know whether this is a closed subgroup). We write $i$ for the
inclusion of the subgroup. By our construction the composition $\overline
{i}:=q\circ i$ is (1) injective and (2) has a full rank $\mathbb{Z}$-lattice
in $\mathbb{R}^{m}$ as its image. As the image of $\overline{i}$ is closed and
discrete, it follows that image of $i$ must also be closed and discrete. We
provide the details: let $y\in\mathbb{R}^{m}\oplus C$ be a point in the image
of $i$. Then $q(y)$ is a point in $\mathbb{R}^{m}$ and since it is isolated,
there exists an open neighbourhood $U$ of $q(y)$ in $\mathbb{R}^{m}$ which
contains no other point coming from $x(E)$. The pre-image $q^{-1}(U)$ is an
open neighbourhood of $y$ in $\mathbb{R}^{m}\oplus C$. Suppose $q^{-1}(U)$
contains a different point $y^{\prime}\neq y$ from $x(E)$. Then $q(y^{\prime
})$ lies in $U$, but since $q(y)$ was the only point from $x(E)$ in this open,
we must have $q(y^{\prime})=q(y)$. This contradicts the injectivity of
$\overline{i}$. Hence, $q^{-1}(U)$ is a neighbourhood of $y$ also not
containing any other point from $x(E)$.\newline It follows that $i$ is a
closed injective map, hence an admissible monic. Moreover, $x\mid_{E}$ is an
injective (because $i$ and $x\mid_{D^{\prime}}$ are) and evidently surjective
map between discrete groups $\mathbb{Z}^{m}$, so $x\mid_{E}$ is an
isomorphism. We obtain%
\[%
{
\begin{adjustbox}{width=\linewidth,center}
\begin{tikzcd}
	E &&& {x(E)} \\
	\\
	&& {\mathbb{R}^n \oplus D'} &&& {\mathbb{R}^m \oplus C} \\
	\\
	&&&& {\mathbb{R}^n \oplus D'/E} &&& {(\mathbb{R}^m \oplus C)/x(E)}
	\arrow["1", from=1-1, to=1-4]
	\arrow["\iota"', hook, from=1-1, to=3-3]
	\arrow["i", hook, from=1-4, to=3-6]
	\arrow["x"'{inner sep=.8ex}, "\bullet"{marking}, from=3-3, to=3-6]
	\arrow["{1 \oplus p}"', two heads, from=3-3, to=5-5]
	\arrow["r", two heads, from=3-6, to=5-8]
	\arrow["{\tilde{x}}"'{inner sep=.8ex}, "\bullet"{marking}, from=5-5, to=5-8]
\end{tikzcd}
\end{adjustbox}
}%
\]
where $\iota$ factors over $D^{\prime}\hookrightarrow\mathbb{R}^{n}\oplus
D^{\prime}$, $p$ is the quotient map by $\iota$, $r$ is the quotient map by
$i$, and both $\iota$ and $i$ are admissible monics. The lower square is
bicartesian since $E\overset{1}{\rightarrow}x(E)$ is an isomorphism. By Lemma
\ref{lemma9} the arrow $\tilde{x}$ induced on the quotients is both a monic
and epic. Now take $D_{new}^{\prime}:=\mathbb{R}^{n}\oplus D^{\prime}/E$ and
$C_{new}^{\prime}:=(\mathbb{R}^{m}\oplus C)/x(E)$. We obtain the diagram%
\[%
{
\begin{tikzcd}
	&& C && C \\
	\\
	{x(E)} && { \mathbb{R}^m \oplus C} && {C'_{new}} \\
	\\
	{ \mathbb{Z}^m} && { \mathbb{R}^m} && { \mathbb{T}^m}.
	\arrow["{r\mid_C}", from=1-3, to=1-5]
	\arrow[hook, from=1-3, to=3-3]
	\arrow[hook, from=1-5, to=3-5]
	\arrow["i", hook, from=3-1, to=3-3]
	\arrow["1"', from=3-1, to=5-1]
	\arrow["r", two heads, from=3-3, to=3-5]
	\arrow["q", two heads, from=3-3, to=5-3]
	\arrow[two heads, from=3-5, to=5-5]
	\arrow[hook, from=5-1, to=5-3]
	\arrow[two heads, from=5-3, to=5-5]
\end{tikzcd}
}%
\]
We have $C\cap x(E)=0$, as it is an intersection of two closed subgroups of
$\mathbb{R}^{m}\oplus C$, and therefore compact (since $C$ is compact) and
discrete (since $x(E)$ is discrete), hence finite, but since $x(E)$ is
torsion-free, so it must be zero. Hence, $r\mid_{C}$ is an isomorphism. We see
that $C_{new}^{\prime}$ is an extension of a (compact) torus $\mathbb{T}^{m}$
by another compact group, hence itself compact. This proves the lemma.
\end{proof}

We are now ready to prove Prop. \ref{prop_ess_surjectivity} for the remaining
variant%
\[
\Theta\colon\mathsf{PGA}\longrightarrow\mathcal{LH}(\mathsf{LCA})\text{.}%
\]
By Lemma \ref{lemma_S2} every object $X^{\bullet}\in\mathsf{T}$ of
$\mathcal{LH}(\mathsf{LCA})$ has the shape%
\[
X^{\bullet}=\left[
{
\begin{tikzcd}
	{\mathbb{R}^n \oplus D'} && {\mathbb{R}^m \oplus C}
	\arrow["x"{inner sep=.8ex}, "\bullet"{marking}, from=1-1, to=1-3]
\end{tikzcd}
}%
\right]
\]
with $D^{\prime}$ discrete and $C$ compact. Then the Pontryagin dual of the
entire complex has the same shape,%
\[
X^{\bullet\vee}=\left[
{
\begin{tikzcd}
	{\mathbb{R}^{n \vee} \oplus D'^{\vee}} && {\mathbb{R}^{m \vee} \oplus
C^{\vee}}
	\arrow["{x^{\vee}}"'{inner sep=.8ex}, "\bullet"{marking}, from=1-3, to=1-1]
\end{tikzcd}
}%
\right]
\]
with $C^{\vee}$ discrete and $D^{\prime\vee}$ compact. Note that the
Pontryagin dual of a morphism which is both monic and epic is again both monic
and epic. Now apply Lemma \ref{lemma_S3} and then Lemma \ref{lemma_S4} to
$X^{\bullet\vee}$, so that we get a quasi-isomorphism%
\[
X^{\bullet\vee}\overset{\sim}{\longrightarrow}\left[
{
\begin{tikzcd}
	{C_{new}} && \mathbb{R}^{n \vee} \oplus{D'_{new}}
	\arrow["{{x}_{new}}"'{inner sep=.8ex}, "\bullet"{marking}, from=1-3, to=1-1]
\end{tikzcd}
}%
\right]
\]
with $D_{new}^{\prime}$ discrete and $C_{new}$ compact, and possibly a
different value for $n$ (to be sure about the place of the vector group
$\mathbb{R}^{n}$, note that we have chosen to write the arrow of the complex
in the reverse direction temporarily). Dualizing back, we obtain%
\begin{equation}
X^{\bullet}\overset{\sim}{\longrightarrow}\left[
{
\begin{tikzcd}
	{C_{new}^{\vee}} && {\mathbb{R}^n \oplus{D'^{\vee}_{new}}}
	\arrow["{{{x}_{new}^{\vee}}}"{inner sep=.8ex}, "\bullet"{marking}%
, from=1-1, to=1-3]
\end{tikzcd}
}%
\right]  \label{l_h3}%
\end{equation}
and since $C_{new}^{\vee}$ is discrete and $D_{new}^{\prime\vee}$ compact,
this shows that $X^{\bullet}$ is isomorphic to a complex of type d-cg. By
Lemma \ref{lemma10} we have $\left[
{
\begin{tikzcd}
	{C_{new}^{\vee}} && {\mathbb{R}^n \oplus{D'^{\vee}_{new}}}
	\arrow["{{{x}_{new}^{\vee}}}"{inner sep=.8ex}, "\bullet"{marking}%
, from=1-1, to=1-3]
\end{tikzcd}
}%
\right]  =\Theta((C_{new}^{\vee})_{\subseteq D_{new}^{\prime\vee}})$, where
$(C_{new}^{\vee})_{\subseteq D_{new}^{\prime\vee}}$ is $C_{new}^{\vee}$,
equipped with the subspace topology of $D_{new}^{\prime\vee}$. This finishes
the proof of Prop. \ref{prop_ess_surjectivity}.

We could even go a bit further. Applying Lemma \ref{lemma_S4} to Eq.
\ref{l_h3} in the special case $n=0$ yields the following:

\begin{corollary}
\label{cor_EveryObjectIsOfTypeDC}Every object in $\mathsf{T}$ in
$\mathcal{LH}(\mathsf{LCA})$ is isomorphic to an object of type d-c, i.e.,%
\[
\left[
{
\begin{tikzcd}
	{D'} && {C},
	\arrow["x"{inner sep=.8ex}, "\bullet"{marking}, from=1-1, to=1-3]
\end{tikzcd}
}%
\right]
\]
with $D^{\prime}$ discrete and $C$ compact.
\end{corollary}

In particular, the composition%
\[
\mathsf{PCA}\longrightarrow\mathsf{PGA}/\mathsf{Ab}_{\operatorname*{fg}%
}\overset{\Theta}{\longrightarrow}\mathsf{T}%
\]
is also essentially surjective.

\begin{example}
In this example we demonstrate that under%
\[
\mathsf{PCA}\longrightarrow\mathsf{PGA}/\mathsf{Ab}_{\operatorname*{fg}%
}\overset{\Theta}{\longrightarrow}\mathsf{T}%
\]
non-isomorphic objects become isomorphic. Suppose $\alpha\in\mathbb{R}%
\setminus\mathbb{Q}$ and $X_{\alpha}$ is the object defined in Example
\ref{example_4}. The preimage of $X_{\alpha}$ in $\mathsf{PCA}$ is
$\alpha\mathbb{Z}_{\subseteq\mathbb{T}}$, meaning $\mathbb{Z}$, but equipped
with the subspace topology of the torus under multiplication by $\alpha$.
(Step 1) We claim that if $\beta$ is a further irrational number and there is
an isomorphism $\alpha\mathbb{Z}_{\subseteq\mathbb{T}}\simeq\beta
\mathbb{Z}_{\subseteq\mathbb{T}}$, then we must have $\frac{\alpha}{\beta}%
\in\{\pm1\}$. To prove this, note that since completion is a functor, we get
an induced isomorphism on the completions, i.e.,%
\[%
{
\begin{tikzcd}
	{\mathbb{Z}} && {\mathbb{R}/\mathbb{Z}} \\
	\\
	{\mathbb{Z}} && {\mathbb{R}/\mathbb{Z}}
	\arrow["\alpha"{inner sep=.8ex}, "\bullet"{marking}, from=1-1, to=1-3]
	\arrow["{f'}"', from=1-1, to=3-1]
	\arrow["f", from=1-3, to=3-3]
	\arrow["\beta"'{inner sep=.8ex}, "\bullet"{marking}, from=3-1, to=3-3]
\end{tikzcd}
}%
\]
where $f^{\prime},f$ are both isomorphisms. Even when we completely disregard
topology, the map $f^{\prime}$ can only be $\pm1$ since it must be a
bijection, and possibly swapping our isomorphism by its negative, we may
assume $f^{\prime}=1$. We have $\operatorname*{Hom}(\mathbb{T},\mathbb{T}%
)\cong\mathbb{Z}$, and again only $\pm1$ is possible to get an isomorphism.
The claim follows. We also deduce that $\operatorname*{Aut}%
\nolimits_{\mathsf{PCA}}(\alpha\mathbb{Z}_{\subseteq\mathbb{T}})=\{\pm1\}$.
(Step 2) Now consider the left roof%
\[%
{
\begin{adjustbox}{width=\linewidth,center}
\begin{tikzcd}
	&&& {\alpha\mathbb{Z} \oplus\mathbb{Z}} && {\mathbb{R}} \\
	\\
	{\alpha\mathbb{Z}} && {\mathbb{R}/\mathbb{Z}} &&&& {\mathbb{Z}}
&& {\mathbb{R}/\alpha\mathbb{Z}} \\
	\\
	{\mathbb{Z}} && {\mathbb{R}/\mathbb{Z}} &&&& {\mathbb{Z}} && {\mathbb
{R}/\mathbb{Z}}
	\arrow["{{\operatorname{incl }}}"{inner sep=.8ex}, "\bullet"{marking}%
, from=1-4, to=1-6]
	\arrow[two heads, from=1-4, to=3-1]
	\arrow[two heads, from=1-4, to=3-7]
	\arrow[two heads, from=1-6, to=3-3]
	\arrow[two heads, from=1-6, to=3-9]
	\arrow["1"'{inner sep=.8ex}, "\bullet"{marking}, from=3-1, to=3-3]
	\arrow["{{\cdot\alpha^{-1}}}"', from=3-1, to=5-1]
	\arrow["1", from=3-3, to=5-3]
	\arrow["1"'{inner sep=.8ex}, "\bullet"{marking}, from=3-7, to=3-9]
	\arrow["1"', from=3-7, to=5-7]
	\arrow["{{\cdot\alpha^{-1}}}", from=3-9, to=5-9]
	\arrow["{{\cdot\alpha}}"', from=5-1, to=5-3]
	\arrow["{{\cdot\alpha^{-1}}}"', from=5-7, to=5-9]
\end{tikzcd}
\end{adjustbox}
}%
\]
in $\mathcal{LH}(\mathsf{LCA})$. Both legs are bicartesian squares, they arise
from quotienting out $\mathbb{Z}\overset{1}{\longrightarrow}\mathbb{Z}$ for
the left leg, resp. $\alpha\mathbb{Z}\overset{1}{\longrightarrow}%
\alpha\mathbb{Z}$ for the right leg. Since both legs stem from bicartesian
squares, this roof determines an isomorphism%
\[
X_{\alpha}\simeq X_{\alpha^{-1}}\text{.}%
\]
Under the correspondence of Thm. \ref{thm_main_correspondence}, we see that
both legs correspond to an epimorphism with kernel in $\mathsf{Ab}%
_{\operatorname*{fg}}$. This cannot correspond to an isomorphism in
$\mathsf{PCA}$.
\end{example}

\section{Proof of the Main Correspondence\label{sect_Proofs}}

We now prove the main claims of the introduction (Theorem \ref{thma_main} and
Theorem \ref{thma_vfversion}). The starting point is the decomposition of the
left heart given by its natural cotilting torsion pair of Prop.
\ref{prop_QAbTorsionPairInLeftHeart}.\footnote{Artusa also uses such a
decomposition in \cite{artusa2025}.}

\begin{theorem}
[Main Correspondence]\label{thm_main_correspondence}The functor $\Theta$
induces exact equivalences of quasi-abelian categories%
\begin{align*}
\Theta\colon\mathsf{PCA}/\mathsf{Ab}_{\operatorname*{fin}}\overset{\sim
}{\longrightarrow}\mathsf{T}  &  \subset\mathcal{LH}(\mathsf{LCA}%
_{\operatorname*{vf}})\\
\Theta\colon\mathsf{PGA}/\mathsf{Ab}_{\operatorname*{fg}}\overset{\sim
}{\longrightarrow}\mathsf{T}  &  \subset\mathcal{LH}(\mathsf{LCA})
\end{align*}
Given $X$ on the left side, $\Theta(X)$ is the complex%
\[
\left[
{
\begin{tikzcd}
	{X_d} && {cX},
	\arrow["x"{inner sep=.8ex}, "\bullet"{marking}, from=1-1, to=1-3]
\end{tikzcd}
}%
\right]  \text{,}%
\]
where $X_{d}$ is $X$, equipped with the discrete topology, and $cX$ is the
Weil completion of $X$. Conversely, in the special case when a complex is
given of the shape%
\begin{equation}
\left[
{
\begin{tikzcd}
	{X'} && {X},
	\arrow["x"{inner sep=.8ex}, "\bullet"{marking}, from=1-1, to=1-3]
\end{tikzcd}
}%
\right]  \label{l_w2}%
\end{equation}
with $X^{\prime}$ discrete and $X$ group-theoretically compactly generated,
its preimage is $X_{\subseteq X}^{\prime}$, i.e., the group $X^{\prime}$, but
equipped with the subspace topology when regarding it as a subset inside $X$.
\end{theorem}

When an object in $\mathsf{T}$ is not of the shape in Eq. \ref{l_w2}, the
tools in \S \ref{sect_RewritingRules} always permit to replace it by an
isomorphic representative such that $X^{\prime}$ is discrete and $X$
group-theoretically compactly generated.

\begin{proof}
Just combine all the previous results:\ By Lemma \ref{lemma_f_full} and Lemma
\ref{lemma_f_faithful}, the functor $\Theta$ is fully faithful. By Prop.
\ref{prop_ess_surjectivity} the functor $\Theta$ is essentially surjective.
Hence, the functor is an equivalence of categories. By Lemma
\ref{lemma_f_exact} the functor is exact and reflects exactness, so this
equivalence preserves the exact structure.
\end{proof}

We also need to suppy the proof for the right heart analogue:

\begin{proof}
[Proof of Theorem \ref{thma_main_rightheartversions}]The basic structure that
every object $X$ in the right heart is uniquely an extension%
\[%
{
\begin{tikzcd}
	B & X & A
	\arrow[hook, from=1-1, to=1-2]
	\arrow[two heads, from=1-2, to=1-3]
\end{tikzcd}
}%
\]
with $B\in\mathsf{LCA}$ [resp. $\mathsf{LCA}_{\operatorname*{vf}}$] and $A$
represented by a monic-epic%
\begin{equation}
\left[
{
\begin{tikzcd}
	{X'} && X
	\arrow["x"{inner sep=.8ex}, "\bullet"{marking}, from=1-1, to=1-3]
\end{tikzcd}
}%
\right]  \label{l_c3}%
\end{equation}
follows from dualizing various statements in Schneiders \cite{MR1779315}, and
both Rump \cite[\S 4, Corollary]{MR1856638} and Bondal--van den Bergh's
article \cite[Prop. B.3, (3)]{MR1996800} explicitly spell out that
$(\mathsf{T}^{\prime},\mathsf{F}^{\prime})$ is a tilting torsion pair with
$\mathsf{T}^{\prime}\cong\mathsf{LCA}$ [resp. $\mathsf{LCA}%
_{\operatorname*{vf}}$] as in Definition \ref{def_TorsionPair}. The
torsion-free part $\mathsf{F}^{\prime}$ still amounts to monic-epic morphisms
as in Eq. \ref{l_c3}, so this is literally the same category as in the left
heart variants. The only change is that it now occurs as the torsion-free part
inside the right heart, while it was the torsion part in the left heart.
\end{proof}

\section{Condensed Picture\label{sect_CondensedPicture}}

The left heart can be interpreted as a full subcategory of $\mathsf{Cond}%
(\mathsf{Ab})$ from \cite{condensedmath}: We set up the solid arrows in the
diagram%
\begin{equation}%
{
\begin{tikzcd}
	{\mathsf{LCA}} && {\mathsf{Cond}(\mathsf{Ab})} \\
	& {\mathcal{LH}(\mathsf{LCA})}
	\arrow["{e_0}", from=1-1, to=1-3]
	\arrow["{i_0}"', from=1-1, to=2-2]
	\arrow["{j_0}"', dashed, from=2-2, to=1-3]
\end{tikzcd}
}
\label{d_w1}%
\end{equation}
as follows: Clausen--Scholze \cite[Lecture 4]{condensedmath} construct the
functor $e_{0}$, even on the derived level. In particular, $e_{0}$ is an exact
functor. The downward diagonal arrow $i_{0}$ is the embedding of Prop.
\ref{prop_QAbTorsionPairInLeftHeart}. The dashed arrow remains to be constructed.

\begin{proposition}
\label{prop_toCondensed}The functors of Diagram \ref{d_w1} induce a
corresponding diagram of derived categories.

\begin{enumerate}
\item On the level of derived categories%
\begin{equation}%
{
\begin{tikzcd}
	{\operatorname*{D}\nolimits^{b}(\mathsf{LCA})} && {\operatorname*{D}%
\nolimits^{b}(\mathsf{Cond}(\mathsf{Ab}))} \\
	& {\operatorname*{D}\nolimits^{b}(\mathcal{LH}(\mathsf{LCA}))}
	\arrow["e", from=1-1, to=1-3]
	\arrow["i"', equals, from=1-1, to=2-2]
	\arrow["j"', from=2-2, to=1-3]
\end{tikzcd}
}
\label{d_w2}%
\end{equation}
the functor $i$ induces a triangulated equivalence of triangulated categories
(and lifts to a functor of stable $\infty$-categories which induces an
equivalence on their respective homotopy categories). The derived version of
$j$ is fully faithful.

\item The dashed arrow $j_{0}$ in Diagram \ref{d_w1} exists and is exact.

\item Diagram \ref{d_w1} commutes, including the dashed arrow $j_{0}$.

\item The functors $i_{0}$ and $j_{0}$ of Diagram \ref{d_w1} are fully
faithful functors, neither of which is essentially surjective.
\end{enumerate}
\end{proposition}

The fully faithful embedding%
\[
\mathcal{LH}(\mathsf{LCA})\longrightarrow\mathsf{Cond}(\mathsf{Ab})
\]
has been noticed by several people. It is, for example, in unpublished notes
of Flach, Geisser--Morin \cite[\S 6]{MR4699875}, a variant can be found
spelled out in Artusa \cite[Lemma 4.3]{MR4831262}. Without doubt many more
people are aware of it.

\begin{proof}
(1) This just combines known results: The claim that $i$ is a derived
equivalence comes from Schneiders \cite[Prop. 1.2.32]{MR1779315}, we had also
recorded this in Prop. \ref{prop_QAbTorsionPairInLeftHeart}. The horizontal
triangulated functor $e$ is fully faithful by Clausen--Scholze \cite[Lecture
4]{condensedmath}. As $i$ is an equivalence, hence invertible, it follows that
$j=e\circ i^{-1}$ must also be fully faithful. (2) Clausen--Scholze shows that
on the level of bounded complexes,%
\[
\operatorname*{Ch}\nolimits^{b}(\mathsf{LCA)}\overset{e_{0}}{\longrightarrow
}\operatorname*{Ch}\nolimits^{b}(\mathsf{Cond}(\mathsf{Ab}))\text{,}%
\]
strictly exact complexes get sent to exact complexes in $\operatorname*{Ch}%
\nolimits^{b}(\mathsf{Cond}(\mathsf{Ab}))$ in the sense of abelian categories
(= acyclic complexes). This means that $e_{0}\colon\mathsf{LCA}\rightarrow
\mathsf{Cond}(\mathsf{Ab})$ is a strictly exact functor in the sense of
Schneiders \cite[Definition 1.1.18, Definition 1.2.10]{MR1779315}. By a
theorem of Schneiders \cite[Prop. 1.2.34]{MR1779315} there is an equivalence
of functor categories between exact functors of abelian categories, depicted
below on the left,%
\begin{equation}
I^{\prime}\colon\operatorname*{Fun}\nolimits_{\text{exact}}(\mathcal{LH}%
(\mathsf{LCA}),\mathsf{Cond}(\mathsf{Ab}))\longrightarrow\operatorname*{Fun}%
\nolimits_{\mathrm{strictly}\text{ exact}}(\mathsf{LCA},\mathsf{Cond}%
(\mathsf{Ab}))\text{,} \label{l_eq2}%
\end{equation}
\newline and strictly exact functors from the quasi-abelian category
$\mathsf{LCA}$ to the abelian category $\mathsf{Cond}(\mathsf{Ab})$.
Alternatively, one can also use a variant of this theorem \cite[Prop.
3.14]{MR4575371} which circumvents the concept of strictly exact functors. The
functor $e_{0}$ thus extends uniquely, up to natural equivalence of functors,
to an exact functor
\[
\mathcal{LH}(\mathsf{LCA})\longrightarrow\mathsf{Cond}(\mathsf{Ab})
\]
and we take this as $j_{0}$ in Diagram \ref{d_w1}. (3) The commutativity of
Diagram \ref{d_w1} follows from the property of the correspondence in Eq.
\ref{l_eq2} that $I^{\prime}$ amounts to restricting a functor on
$\mathcal{LH}(\mathsf{LCA})$ to the embedding of $\mathsf{LCA}$ into the
heart, which is exactly what $i_{0}$ does. (4) The functor $i_{0}$ is fully
faithful by \cite[Cor. 1.2.28]{MR1779315}. The full faithfulness of $j_{0}$
means that%
\[
\operatorname*{Hom}\nolimits_{\mathcal{LH}(\mathsf{LCA})}(X,Y)\longrightarrow
\operatorname*{Hom}\nolimits_{\mathsf{Cond}(\mathsf{Ab})}(j_{0}X,j_{0}Y)
\]
is an isomorphism for all objects $X,Y$ in the left heart. However, we already
know that%
\[
\operatorname*{RHom}\nolimits_{\operatorname*{D}\nolimits^{b}\mathcal{LH}%
(\mathsf{LCA})}(X,Y)\overset{\sim}{\longrightarrow}\operatorname*{RHom}%
\nolimits_{\operatorname*{D}\nolimits^{b}\mathcal{LH}(\mathsf{LCA}%
)}(X,Y)\overset{\sim}{\longrightarrow}\operatorname*{RHom}%
\nolimits_{\operatorname*{D}\nolimits^{b}\mathsf{Cond}(\mathsf{Ab})}%
(j_{0}X,j_{0}Y)\text{,}%
\]
where the first arrow stems from $i$ being a derived equivalence in Diagram
\ref{d_w2} and the second arrow is an isomorphism since $j$ is fully faithful
in Diagram \ref{d_w2}. Now restrict to $\pi_{0}$ of the Hom-complexes. We
leave it as an easy exercise to exhibit examples showing that neither $i_{0}$
nor $j_{0}$ are essentially surjective.
\end{proof}

We obtain an entertaining side result.

\begin{theorem}
\label{thm_p1}The composition $j_{0}\circ\Theta$ induces a fully faithful and
exact embedding%
\[
\mathsf{PGA}/\mathsf{Ab}_{\operatorname*{fg}}\longrightarrow\mathsf{Cond}%
(\mathsf{Ab})
\]
resp.%
\[
\mathsf{PCA}/\mathsf{Ab}_{\operatorname*{fin}}\longrightarrow\mathsf{Cond}%
(\mathsf{Ab})
\]
It sends a precompactly generated group [resp. precompact] $X$ to the cokernel%
\begin{equation}
\operatorname*{coker}\nolimits_{\mathsf{Cond}(\mathsf{Ab})}\left(  e_{0}%
(X_{d})\longrightarrow e_{0}(cX)\right)  \text{,} \label{l_e2}%
\end{equation}
where $X_{d}$ is $X$, equipped with the discrete topology and then regarded as
a condensed abelian group, and $cX$ is the completion of $X$, a locally
compact group, which is then again regarded as a condensed abelian group. The
cokernel is evaluated in $\mathsf{Cond}(\mathsf{Ab})$.
\end{theorem}

The existence of this embedding is somewhat eccentic because if one uses the
usual embedding of CGWH topological groups into $\mathsf{Cond}(\mathsf{Ab})$
of Clausen--Scholze, it will not lead to a faithful functor since groups in
$\mathsf{PGA}$ or $\mathsf{PCA}$ need not be $k$-spaces. So the usual
machinery to treat topological groups in a condensed way is very much unavailable.

\begin{example}
[S. Gabriyelyan \cite{MR3968523}]\label{example_p1}There exist precompact
abelian groups whose underlying topological space is not CGWH. One possible
construction is found in the proof of \cite[Prop. 3.1]{MR3968523}.
\end{example}

\begin{example}
[J. F. Trigos-Arrieta \cite{MR1100517}]\label{example_p2}Every locally compact
topological abelian group $A$ can be equipped with a precompact topology
$\tau^{+}$ such that its $k$-ification is again $A$ \cite[Lemma 1.12]%
{MR1100517}. If $A$ is non-compact, this also yields precompact groups
$(A,\tau^{+})$ which are not CGWH.
\end{example}

\begin{proof}
We only discuss the variant for $\mathsf{PGA}/\mathsf{Ab}_{\operatorname*{fg}%
}$ and leave the necessary modifications for $\mathsf{PCA}/\mathsf{Ab}%
_{\operatorname*{fin}}$ to the reader. We use that%
\[
\mathsf{PGA}/\mathsf{Ab}_{\operatorname*{fg}}\overset{\Theta}{\longrightarrow
}\mathsf{Ghost}\longrightarrow\mathcal{LH}(\mathsf{LCA})\overset{j_{0}%
}{\longrightarrow}\mathsf{Cond}(\mathsf{Ab})
\]
is a composition of fully faithful exact functors (combine Theorem
\ref{thm_main_correspondence}, Prop. \ref{prop_QAbTorsionPairInLeftHeart} and
Prop. \ref{prop_toCondensed}). Since the cotilting torsion pair $(\mathsf{T}%
,\mathsf{F})$ of Prop. \ref{prop_QAbTorsionPairInLeftHeart} is made from full
subcategories, the inclusion of the torsion part $\mathsf{Ghost}:=\mathsf{T}$
into $\mathcal{LH}(\mathsf{LCA})$ is tautologically fully faithful. The
formula in Eq. \ref{l_e2} arises from unravelling each step on the level of
derived categories and then observing that the output lies inside the heart:%
\begin{equation}
\left(  j_{0}\circ\Theta\right)  (X)=j_{0}\left[
{
\begin{tikzcd}
	{X_d} && cX
	\arrow["\bullet"{marking}, from=1-1, to=1-3]
\end{tikzcd}
}%
\right]  _{\operatorname*{D}\nolimits^{b}\mathcal{LH}(\mathsf{LCA})}\nonumber
\end{equation}
Note that while $X_{d}\rightarrow cX$ is a monic and epic in $\mathsf{LCA}$,
its image under $e_{0}$ is only monic in $\mathsf{Cond}(\mathsf{Ab})$, but
(generally) not epic. The functor $j_{0}$ on complexes reduces to being just
$e_{0}$ objectwise:%
\[
=\left[
{
\begin{tikzcd}
	{e_0(X_d)} & {e_0(cX)}
	\arrow[from=1-1, to=1-2]
\end{tikzcd}
}%
\right]  _{\operatorname*{D}\nolimits^{b}\mathsf{Cond}(\mathsf{Ab})}%
\simeq\left[
{
\begin{tikzcd}
	0 & {\operatorname{coker}( \ldots)}
	\arrow[from=1-1, to=1-2]
\end{tikzcd}
}%
\right]  _{\operatorname*{D}\nolimits^{b}\mathsf{Cond}(\mathsf{Ab})}%
\]
The cokernel exists since $\mathsf{Cond}(\mathsf{Ab})$ is an abelian category
and the complexes are visibly quasi-isomorphic. The resulting final complex
lies in the essential image of the embedding $\mathsf{Cond}(\mathsf{Ab})$ into
its derived category, or said differently, it lies in the heart of the
standard $t$-structure on $\operatorname*{D}\nolimits^{b}\mathsf{Cond}%
(\mathsf{Ab})$. Hence, the formula lifts to taking values in $\mathsf{Cond}%
(\mathsf{Ab})$.
\end{proof}

\begin{remark}
For pro-limits, evaluated in topological Hausdorff abelian groups, the above
functor to $\mathsf{Cond}(\mathsf{Ab})$ is compatible with the construction in
\cite{ren2025coherentsixfunctorformalismspro}.
\end{remark}

\section{Duality}

Pontryagin duality induces an exact equivalence of exact categories%
\[
(-)^{\vee}\colon\mathsf{LCA}\longrightarrow\mathsf{LCA}^{op}\text{.}%
\]
This induces a derived equivalence%
\[
\operatorname*{D}\nolimits^{b}(\mathsf{LCA})\longrightarrow\operatorname*{D}%
\nolimits^{b}(\mathsf{LCA}^{op})
\]
which exchanges Schneider's left $t$-structure with the right $t$-structure on
the opposite category. This induces an equivalence of abelian categories
$\mathcal{LH}(\mathsf{LCA})\rightarrow\mathcal{RH}(\mathsf{LCA})^{op}$, given
by%
\begin{equation}
\left[
{
\begin{tikzcd}
	{X'} && X
	\arrow["x", from=1-1, to=1-3]
\end{tikzcd}
}%
\right]  \longmapsto\left[
{
\begin{tikzcd}
	{X^{\vee}} && {{X'}^{\vee}}
	\arrow["{x^{\vee}}", from=1-1, to=1-3]
\end{tikzcd}
}%
\right]  \text{,} \label{l_m1}%
\end{equation}
where $x$ is a (not necessarily admissible) monic and $x^{\vee}$ a (not
necessarily admissible) epic in $\mathsf{LCA}$. This equivalence sends a ghost
object, i.e., an object where $x$ is both monic and epic, to a ghost object.
We obtain an induced equivalence of quasi-abelian categories $\mathsf{Ghost}%
\rightarrow\mathsf{Ghost}^{op}$, given by restricting the functor of Eq.
\ref{l_m1} to the ghost objects:
\[
\left[
{
\begin{tikzcd}
	{X'} && X
	\arrow["x"{inner sep=.8ex}, "\bullet"{marking}, from=1-1, to=1-3]
\end{tikzcd}
}%
\right]  \longmapsto\left[
{
\begin{tikzcd}
	{X^{\vee}} && {X'^{\vee}}
	\arrow["{{x^{\vee}}}"{inner sep=.8ex}, "\bullet"{marking}, from=1-1, to=1-3]
\end{tikzcd}
}%
\right]  \text{.}%
\]
We observe that objects of type d-c (Definition
\ref{def_ObjectsOfTypeDC_resp_DCG}) are preserved under this functor since the
Pontryagin dual of a compact group is discrete, and reversely.

All of the above carries over verbatim to $\mathsf{LCA}_{\operatorname*{vf}}$
and $\mathsf{Ghost}_{\operatorname*{vf}}$. The dual of a group without a real
vector space summand also does not have a real vector space summand.

\begin{definition}
\label{def_WeakDual}Suppose $X$ is a Hausdorff abelian group. The \emph{weak
dual }of $X$ is%
\[
X^{\ast}:=(\operatorname*{Hom}(X,\mathbb{T}),\tau_{\operatorname*{weak}%
})\text{,}%
\]
where $\operatorname*{Hom}(X,\mathbb{T})$ is the abelian group of continuous
group morphisms $X\rightarrow\mathbb{T}$ (the continuous characters on $X$).
We equip this group with the weak topology $\tau_{\operatorname*{weak}}%
$:\ This is the weakest topology such that the pointwise evaluation map%
\[
\operatorname*{ev}\nolimits_{X}\colon\operatorname*{Hom}(X,\mathbb{T}%
)\longrightarrow\prod_{x\in X}\mathbb{T}\text{,}\qquad\chi\longmapsto
(\chi(x))_{x\in X}\text{,}%
\]
is continuous.\footnote{Said differently: The preimages $\operatorname*{ev}%
\nolimits_{X}^{-1}(U)$ of the opens $U\subseteq\prod_{x\in X}\mathbb{T}$
generate the topology on $X^{\ast}$.} The weak topology is also known as the
\emph{finite-open topology} or as the \emph{topology of pointwise convergence}.
\end{definition}

In general, this concept of dual differs from the Pontryagin dual $X^{\vee}$,
where we equip the group $\operatorname*{Hom}(X,\mathbb{T})$ with the
compact-open topology. Both $X^{\vee}$ and $X^{\ast}$ have the same underlying
abelian group, but possibly drastically different topologies.

It is sensible to use the term \textquotedblleft\emph{strong dual}%
\textquotedblright\ for the Pontryagin duals in order to stress the relation
between these two notions of a dual.

\begin{theorem}
[Weak-Strong Duality]\label{thm_dual_1}Under the equivalence of categories
$\Theta$ for $\mathsf{LCA}_{\operatorname*{vf}}$, weak duality gets exchanged
with strong (Pontryagin) duality, i.e.,%
\begin{equation}%
{
\begin{tikzcd}
	{\mathsf{PCA}/{\mathsf{Ab}_{\operatorname{fin} }}} && {\mathsf{Ghost}%
_{\operatorname{vf} }} \\
	\\
	{(\mathsf{PCA}/{\mathsf{Ab}_{\operatorname{fin} }})^{op }} && {\mathsf
{Ghost}_{\operatorname{vf} }^{op }}
	\arrow["\Theta", from=1-1, to=1-3]
	\arrow["{(-)^{\ast}}"', from=1-1, to=3-1]
	\arrow["{(-)^{\vee}}", from=1-3, to=3-3]
	\arrow["{{\Theta}^{op }}"', from=3-1, to=3-3]
\end{tikzcd}
}
\label{d_8}%
\end{equation}
commutes. The functor $(-)^{\ast}$ on the left induces an equivalence of
categories, just like $(-)^{\vee}$ on the right.
\end{theorem}

\begin{proof}
This is essentially an argument taken from the ideas of Rump in his insightful
article \cite{MR4535276}. For the proof of this claim we need to rely on the
duality theory of precompact groups which has not played a role so far
anywhere in this text: Weak duality induces an equivalence of categories%
\begin{equation}
\mathsf{PCA}\longrightarrow\mathsf{PCA}^{op},\qquad X\mapsto X^{\ast}\text{,}
\label{l_c1}%
\end{equation}
and in particular the weak dual of a precompact group is (1) functorial and
(2) again precompact. See \cite[Theorem 13.6.2]{MR4510389} or \cite{MR1831580}
for proofs. Moreover, the weak dual of a finite group is again finite, so this
induces the refined equivalence%
\begin{equation}
\mathsf{PCA}/\mathsf{Ab}_{\operatorname*{fin}}\longrightarrow\left(
\mathsf{PCA}/\mathsf{Ab}_{\operatorname*{fin}}\right)  ^{op} \label{l_c2}%
\end{equation}
of the left side in Diagram \ref{d_8}. This is the topological viewpoint on
duality. Our proof will rest on an idea of Rump to give a category-theoretic
description of the same duality. Rump proves in \cite[Prop. 2.1]{MR4535276}
that there is an equivalence of categories%
\begin{align*}
\mathsf{PCA}  &  \longrightarrow\mathsf{DCA}\\
X  &  \longmapsto(X_{d}\longrightarrow bX)\text{,}%
\end{align*}
where $\mathsf{DCA}$ is the category of dense subgroups $D\subseteq C$ of
compact topological groups $C$ (see loc. cit. for details) and $bX$ denotes
the Bohr compactification.\footnote{We follow Rump's viewpoint and phrase this
as Bohr compactification of $X$, but in view of Lemma
\ref{lemma_OnPCACompletionAgreesWithBohrCompactification}, one could also
reformulate this as the completion $cX$, exhibiting a certain parallelism to
our functor $\Theta$ at least in the vector-free situation.} Then he shows by
combining this with \cite[Corollary 3]{MR4535276} and \cite[Theorem
5.1]{MR4535276} that the duality of Eq. \ref{l_c1} (by a uniqueness of
possible dualities) must agree with%
\[
(x\colon X_{d}\longrightarrow bX)\longmapsto(x^{\vee}\colon(bX)^{\vee
}\longrightarrow(X_{d})^{\vee})\text{,}%
\]
i.e., Pontryagin duality on the arrow embedding the dense subgroup into the
compact subgroup. As we only need exactly the same claim modulo finite groups
in the version of Eq. \ref{l_c2}, Rump's argument generalizes to show our claim.
\end{proof}

We spell out the claim of Theorem \ref{thm_dual_1} in terms of the notation of
Definition \ref{def_SubspaceTopologyNotation}:

\begin{corollary}
\label{cor_w1}For any object of type d-c%
\[
\left[
{
\begin{tikzcd}
	D && C
	\arrow["x"{inner sep=.8ex}, "\bullet"{marking}, from=1-1, to=1-3]
\end{tikzcd}
}%
\right]
\]
we have%
\[
\left(  D_{\subseteq C}\right)  ^{\ast}\cong\left.  C^{\vee}\right.
_{\subseteq D^{\vee}}\text{.}%
\]

\end{corollary}

Regarding this corollary, there is no difference whether we are in
$\mathsf{PCA}$ or $\mathsf{PGA}$.

For objects of type d-cg the situation is less clear. It is natural to ask for
a variant of Theorem \ref{thm_dual_1} for the left heart $\mathcal{LH}%
(\mathsf{LCA})$. Certainly, there is also a (tautologically defined) duality
functor $D:=(\Theta^{op})^{-1}\circ\left(  -\right)  ^{\vee}\circ\Theta$ so
that%
\[%
{
\begin{tikzcd}
	{\mathsf{PGA}/{\mathsf{Ab}_{\operatorname{fg} }}} && {\mathsf{Ghost}} \\
	\\
	{(\mathsf{PGA}/{\mathsf{Ab}_{\operatorname{fg} }})^{op }} && {\mathsf
{Ghost}^{op }}
	\arrow["\Theta", from=1-1, to=1-3]
	\arrow["D"', from=1-1, to=3-1]
	\arrow["{(-)^{\vee}}", from=1-3, to=3-3]
	\arrow["{{\Theta}^{op }}"', from=3-1, to=3-3]
\end{tikzcd}
}%
\]
commutes, but we do not know if it has a nice description in terms of
point-set topology like in Definition \ref{def_WeakDual}. By Corollary
\ref{cor_EveryObjectIsOfTypeDC} every object in $\mathsf{Ghost}$ in
$\mathcal{LH}(\mathsf{LCA})$ is isomorphic to an object of type d-c. On such a
representative, $D$ will agree with the description in Corollary \ref{cor_w1}.
However, $D$ cannot possibly agree with weak duality on all of $\mathsf{PGA}%
/\mathsf{Ab}_{\operatorname*{fg}}$:

\begin{itemize}
\item While the weak dual of an object in $\mathsf{Ab}_{\operatorname*{fin}}$
is again in $\mathsf{Ab}_{\operatorname*{fin}}$, this is not true for
$\mathsf{Ab}_{\operatorname*{fg}}$. In fact, the weak dual $\mathbb{Z}^{\ast}$
is the usual circle $\mathbb{T}$ and not even a discrete group (\cite[Example
4]{MR1831580}). Even worse, $\mathbb{Z}^{\ast\ast}\ncong\mathbb{Z}$.

\item More generally, the weak duality functor can be seen not to be
essentially surjective onto $\mathsf{PGA}/\mathsf{Ab}_{\operatorname*{fg}}$,
so it cannot be an equivalence.

\item And the strong dual of objects of type d-cg is not of type d-cg again.
\end{itemize}

This leaves some loose threads as to how to implement duality into our main
correspondence of Theorem \ref{thm_main_correspondence}. G\'{a}bor Luk\'{a}cs
has informed us on some unpublished work of his which might resolve these
matters \cite{gabor1,gabor2}.%

\appendix

\section{Locally (pre)compact groups\label{Appendix_LocallyPrecompactGroups}}

The concepts of \textit{precompact} and\textit{ locally precompact}
topological groups originate from a manuscript of Andr\'{e} Weil
\cite{weil_uniforme}. Weil phrases everything in terms of his new\footnote{new
in 1937} theory of \textit{uniform spaces}, a theory he develops in the same
memoir. This viewpoint has fallen a little out of favour since then, see
\cite{MR644485} though, and we follow the trend to describe the entire
formalism exclusively through the language of topology.

See \cite[Ch. 3, \S 3.7]{MR2433295} or \cite[\S 10]{MR4510389} for excellent
modern textbook treatments.

\begin{definition}
A \emph{neighbourhood} $U$ of a point is a subset which contains an open set
$V$ such that $x\in V\subseteq U$. A\ neighbourhood is \emph{symmetric} if
$x\in U$ implies $-x\in U$.
\end{definition}

If $U$ is any neighbourhood, $U\cup(-U)$ is a symmetric neighbourhood. For any
subset $\Sigma\subseteq X$ in a topological abelian group $X$, write
$\left\langle \Sigma\right\rangle $ for the subgroup (algebraically) generated
by the elements in $\Sigma$. This need not be a closed subgroup, see Caution
\ref{caution_2}.

As a slogan, a group is \emph{precompact}\footnote{Other authors use the term
\textquotedblleft totally bounded\textquotedblright\ for the same concept. For
some authors \emph{precompact} is defined to mean `totally bounded and
Hausdorff', so that totally bounded alone can be used for non-Hausdorff
groups.} if every non-empty open in $X$ can cover all of $\Sigma$ by finitely
many of its translates. We phrase this with a little more detail:

\begin{definition}
[{\cite[\S 5]{MR2681374}, \cite[Def. 3]{MR163985}}]\label{def_Precompact}Let
$X$ be a Hausdorff abelian group.

\begin{enumerate}
\item We call a subset $\Sigma\subseteq X$ \emph{precompact} if one of the
following (then all) hold:

\begin{enumerate}
\item For every non-empty open $U\subseteq X$ there exist finitely many
$x_{1},\ldots,x_{n}\in X$ such that $\Sigma\subseteq\bigcup_{i=1}^{n}%
(x_{i}+U)$.

\item For every open neighbourhood $U$ of $0$ there exist finitely many
$x_{1},\ldots,x_{n}\in X$ such that $\Sigma\subseteq\bigcup_{i=1}^{n}%
(x_{i}+U)$.

\item For every open neighbourhood $U$ of $0$ there exist finitely many
$\sigma_{1},\ldots,\sigma_{n}\in\Sigma$ (this is more restrictive than
before!) such that $\Sigma\subseteq\bigcup_{i=1}^{n}(\sigma_{i}+U)$.
\end{enumerate}

\item The group $X$ is called \emph{locally precompact}\footnote{Some authors
call this \textquotedblleft locally bounded Hausdorff\textquotedblright, e.g.
\cite{MR1326826}.} if $0$ has a precompact neighbourhood.

\item A locally precompact group $X$ is called \emph{precompactly generated}
if $0$ has a precompact neighbourhood which generates $X$ as an abstract
group.\footnote{Said differently: There is a precompact neighbourhood $U$ of
$0$ such that any subgroup $Z\subseteq X$ containing $U$ must satisfy $Z=X$.}
\end{enumerate}
\end{definition}

\begin{proof}
We need to prove that the three characterizations in (1) are equivalent.
(a$\Rightarrow$b) Obvious. (b$\Rightarrow$c) Suppose $V$ is an open
neighbourhood of $0$. By the continuity of $X\times X\rightarrow X$, pick an
open neighbourhood $U$ of $0$ such that $U+U\subseteq V$ and perhaps upon
replacing $U$ by $U\cap-U$, ensure that $U$ is a symmetric neighbourhood of
$U$. By characterization (b), we have $\Sigma\subseteq\bigcup_{i=1}^{n}%
(x_{i}+U)$ and we may assume $\Sigma\cap(x_{i}+U)\neq\varnothing$ for all $i$,
for otherwise a proper subset of $\{x_{1},\ldots,x_{n}\}$ is enough to cover
$\Sigma$. Hence, we can pick $\sigma_{i}\in\Sigma\cap(x_{i}-U)$ (since $-U=U$)
and observe that $x_{i}+U\subseteq\sigma_{i}+U+U\subseteq\sigma_{i}+V$. It
follows that $\Sigma\subseteq\bigcup_{i=1}^{n}(\sigma_{i}+V)$ with all
$\sigma_{i}\in\Sigma$. This shows (c). (c$\Rightarrow$a) Pick any $x_{0}$ such
that the translate $x_{0}+U$ contains $0$. Now use the open $x_{0}+U$ in (c)
and then $\Sigma\subseteq\bigcup_{i=1}^{n}(\sigma_{i}+x_{0}+U)$ so that we may
take $x_{i}^{\prime}:=\sigma_{i}+x_{0}$ in (a).
\end{proof}

\begin{caution}
\label{caution_2}In Definitions \ref{def_CompactlyGeneratedTopologicalGroup}
and \ref{def_Precompact} we speak of generation \textquotedblleft\textit{as an
abstract group}\textquotedblright\ to prevent confusion with the concept of
topological generation. As a contrast: The circle $\mathbb{T}$ cannot be
generated by a single element. For any choice of $x\in\mathbb{T}$, the
subgroup $\left\langle x\right\rangle $ generated by $x$ is either finite
cyclic or a countable dense subgroup. A subset $\Sigma$ \emph{topologically
generates} a topological group $X$ if the closure $\overline{\left\langle
\Sigma\right\rangle }$ of the group algebraically generated by $\Sigma$ is the
entire group. With this definition, all tori $\mathbb{T}^{\kappa}$ with
$\kappa$ a cardinal $\leq\left\vert \mathbb{R}\right\vert $ (at most the
cardinality of the continuum) turn out to be topologically generated by a
single element because the orbit of a single element is just about enough to
be dense. One can also rephrase this concept as follows: For any non-empty
open $U\subset X$, we have $U+\left\langle \Sigma\right\rangle =X$. Although
widely used in other literature, in the present text we \textit{never} use the
concept of topological generators.
\end{caution}

\begin{example}
\label{example_2}Dense subgroups of compact groups are precompact: For any
open neighbourhood $U$ of $0$ in a compact group $C$, the translates $x+U$ are
open and non-empty, so $C=\bigcup_{c\in C}(c+U)$ is an open cover. Since $C$
is compact, finitely many translates suffice. Next, we shall show that for any
given dense subset $X\subseteq C$, finitely many translates $x+U$ with $x$
\emph{taken exclusively from} $X$ suffice. This is easy to see: For any $c\in
C$, the set $(c-U)\cap X$ must be non-empty since $c-U$ is open and $X$ dense
in $C$. Pick some $x_{c}$ in this intersection, so that $c\in x_{c}+U$.
Evidently, we now get $C=\bigcup_{c\in C}(x_{c}+U)\subseteq\bigcup_{x\in
X}(x+U)$. Since $C$ is compact, finitely many of these translates suffice.
Hence, if $X$ is a dense subgroup in $C$ and $V$ an open neighbourhood of $0$
in $X$, we can realize $V=U\cap X$ for some open $U$ around $0$ in $C$, and it
follows that finitely many translates $x+(U\cap X)$ with $x\in X$ cover all of
$X$. Hence, $X$ is precompact.
\end{example}

Using the structure theory for LCA groups, one can give a very different
characterization of the groups in $\mathsf{PGA}$:

\begin{proposition}
[{Comfort--Luk\'{a}cs \cite[Theorem 5.3]{MR2681374}}]%
\label{prop_ComfortLukacs}Let $X$ be a Hausdorff abelian group. Then $X$ is
locally precompact and precompactly generated if and only if $X$ is isomorphic
to a (not necessarily closed) subgroup of a connected locally compact abelian group.
\end{proposition}

\begin{proof}
The cited theorem is phrased a little different and requires $X$ to be locally
precompact from the outset, so we need to show that this condition is also
true in our formulation: If $X$ is locally precompact and precompactly
generated, use \cite[Theorem 5.3]{MR2681374} directly. For the converse:
Suppose $X$ is a subgroup of an LCA group $W$. Then $X$ is dense inside
$\overline{X}$, the closure taken with respect to $W$, and the group
$\overline{X}$ is also locally compact since it is a closed subgroup of an
LCA\ group (Lemma \ref{lemma_LCAPermanence}). It follows that $\overline{X}$
is complete \cite[Prop. 8.2.6]{MR4510389} in the sense of
\S \ref{sect_Completion}. The inclusion $i\colon X\rightarrow\overline{X}$
factors uniquely%
\[%
{
\begin{tikzcd}
	X & {\overline{X}} \\
	& cX
	\arrow["i", from=1-1, to=1-2]
	\arrow[from=1-1, to=2-2]
	\arrow["j"', dashed, from=2-2, to=1-2]
\end{tikzcd}
}%
\]
by the universal property of completion \cite[Prop. 7.1.18]{MR4510389} and $j$
must be an isomorphism since it is a closed embedding by \cite[Prop.
7.1.22]{MR4510389} and has dense image by construction. By Lemma
\ref{lemma_PrecompactAreThoseWithCompactClosureInCompletion} we deduce that
$X$ is locally precompact. Now we can again invoke the formulation as provided
in \cite[Theorem 5.3]{MR2681374}.
\end{proof}

\begin{example}
The group $\bigoplus_{I}\mathbb{Z}$ for some index set $I$ is precompactly
generated iff $I$ is finite: Any non-empty open $U\subseteq\mathbb{Z}$
contains at least one element, so the precompact subsets of $\mathbb{Z}$ are
the finite subsets, analogously for $\bigoplus_{I}\mathbb{Z}$. Thus, a
subgroup is precompactly generated if and only if it is finitely generated. If
it is, the embedding $\bigoplus_{I}\mathbb{Z}\subset\bigoplus_{I}\mathbb{R}$
realizes the group as a subgroup of a connected LCA group, as we already knew
abstractly to be possible by Prop. \ref{prop_ComfortLukacs}.
\end{example}

The following examples rely on the notation from Definition
\ref{def_SubspaceTopologyNotation}.

\begin{example}
The group $\mathbb{Z}_{\subseteq\mathbb{Z}_{p}}$ is precompactly generated:
The set $\mathbb{Z}$ is dense in $\mathbb{Z}_{p}$, and as $\mathbb{Z}_{p}$ is
compact, Example \ref{example_2} shows that $\mathbb{Z}_{\subseteq
\mathbb{Z}_{p}}$ is precompact, hence tautologically precompactly generated.
An embedding as described by Prop. \ref{prop_ComfortLukacs} can be produced as
follows: Pick a projective resolution $\mathbb{Z}^{\oplus I}\hookrightarrow
\mathbb{Z}^{\oplus J}\overset{q}{\twoheadrightarrow}\mathbb{Q}_{p}%
/\mathbb{Z}_{p}$ in $\mathsf{LCA}$ and take its Pontryagin dual. Then
$q^{\vee}\colon\mathbb{Z}_{p}\hookrightarrow\prod_{J}\mathbb{T}$ is a closed
immersion and the product of tori is compact and connected.
\end{example}

\begin{example}
The group $\mathbb{Q}_{\subseteq\mathbb{Q}_{p}}$ is not precompactly
generated. If $\Sigma$ is a precompact subset, $\overline{\Sigma}$ will be
compact in $\mathbb{Q}_{p}$ and by the same argument as in Example
\ref{example_f0}, this implies $\Sigma\subseteq\frac{1}{p^{n}}\mathbb{Z}$, so
$\left\langle \Sigma\right\rangle $ will be a proper subgroup of $\mathbb{Q}$.
\end{example}

\begin{acknowledgement}
The name of the category $\mathsf{Ghost}$ is inspired by Oren Ben-Bassat, who
spontaneously called it that during a video call. The idea to work with
complexes $\left[
{
\begin{tikzcd}
	{X'} && X
	\arrow["x"{inner sep=.8ex}, "\bullet"{marking}, from=1-1, to=1-3]
\end{tikzcd}
}%
\right]  $, where $X^{\prime}$ is discrete and the category $\mathsf{LPA}$ can
already found in unpublished work of G\'{a}bor Luk\'{a}cs from around 2006
\cite{gabor1,gabor2,gabor3,gabor4} and an article of Wolfgang Rump
\cite{MR4535276}, but not in the context of the left heart. Both instead were
interested in extensions of Pontryagin duality. We thank Marco Artusa, Oren
Ben-Bassat, Alexei Bondal, Sondre Kvamme, G\'{a}bor Luk\'{a}cs, Wolfgang Rump
and Sven-Ake Wegner for correspondence, Reid Barton and Johan Commelin for a
minicourse (actually two!) on Condensed Maths. We thank Matthias Flach and
G\'{a}bor Luk\'{a}cs for sharing unpublished notes with us, Javier
Trigos-Arrieta for helping us to obtain a copy of his PhD Thesis and related
correspondence, Norbert Hoffmann for spotting a mistake in an earlier
version.\medskip\newline O.B. thanks Sven-Ake Wegner and Marianne Lawson for
interesting discussions in Hamburg.
\end{acknowledgement}

\bibliographystyle{amsalpha}
\bibliography{bibf}

\end{document}